\documentclass[a4paper,11pt]{amsart}
\usepackage{amsmath,amssymb,url,stmaryrd}
\usepackage[T1]{fontenc}
\usepackage{lmodern}
\usepackage[all]{xy}
\usepackage[top=30truemm,bottom=30truemm,left=25truemm,right=25truemm]{geometry}
\usepackage{tikz}
\usetikzlibrary{positioning}
\usetikzlibrary{arrows.meta}
\usetikzlibrary{calc}

\newtheorem{Thm}{Theorem}[section]
\newtheorem{Cor}[Thm]{Corollary}
\newtheorem{Prop}[Thm]{Proposition}
\newtheorem{Lem}[Thm]{Lemma}

\newtheorem*{Thm*}{Theorem}

\theoremstyle{definition}
\newtheorem{Def}[Thm]{Definition}

\newtheorem{Ex}[Thm]{Example}

\theoremstyle{remark}

\makeatletter

\@addtoreset{equation}{section}
\makeatother

\DeclareMathOperator{\maks}{\mathsf{max}}

\DeclareMathOperator{\gcdsf}{\mathsf{gcd}}

\DeclareMathOperator{\supsf}{\mathsf{sup}}

\DeclareMathOperator{\twopresilt}{\mathsf{2-psilt}}
\DeclareMathOperator{\twosilt}{\mathsf{2-silt}}

\DeclareMathOperator{\ftors}{\mathsf{f-tors}}
\DeclareMathOperator{\ftorf}{\mathsf{f-torf}}
\DeclareMathOperator{\tors}{\mathsf{tors}}
\DeclareMathOperator{\torf}{\mathsf{torf}}

\DeclareMathOperator{\simple}{\mathsf{sim}}
\DeclareMathOperator{\Fac}{\mathsf{Fac}}

\DeclareMathOperator{\Sub}{\mathsf{Sub}}

\DeclareMathOperator{\brick}{\mathsf{brick}}
\DeclareMathOperator{\sbrick}{\mathsf{sbrick}}

\DeclareMathOperator{\inj}{\mathsf{inj}}
\DeclareMathOperator{\proj}{\mathsf{proj}}
\DeclareMathOperator{\fd}{\mathsf{fd}}

\DeclareMathOperator{\Hom}{\mathsf{Hom}}
\DeclareMathOperator{\PHom}{\mathsf{Hom}}
\DeclareMathOperator{\Aut}{\mathsf{Aut}}
\DeclareMathOperator{\Ext}{\mathsf{Ext}}
\DeclareMathOperator{\End}{\mathsf{End}}
\DeclareMathOperator{\rad}{\mathsf{rad}}
\DeclareMathOperator{\soc}{\mathsf{soc}}
\DeclareMathOperator{\add}{\mathsf{add}}
\DeclareMathOperator{\ind}{\mathsf{ind}}
\DeclareMathOperator{\Ker}{\mathsf{Ker}}
\DeclareMathOperator{\Image}{\mathsf{Im}}
\DeclareMathOperator{\Coker}{\mathsf{Coker}}

\DeclareMathOperator{\dimension}{\mathsf{dim}}

\DeclareMathOperator{\IR}{\mathsf{IR}}
\DeclareMathOperator{\INR}{\mathsf{INR}}

\DeclareMathOperator{\red}{\mathsf{red}}

\DeclareMathOperator{\MP}{\mathsf{MP}}
\DeclareMathOperator{\ovMP}{\overline{\mathsf{MP}}}
\DeclareMathOperator{\Cyc}{\mathsf{Cyc}}

\DeclareMathOperator{\fac}{\mathsf{fac}}
\DeclareMathOperator{\sub}{\mathsf{sub}}

\newcommand{\bbA}{\mathbb{A}}

\newcommand{\Q}{\mathbb{Q}}
\newcommand{\R}{\mathbb{R}}
\newcommand{\Z}{\mathbb{Z}}

\newcommand{\calC}{\mathcal{C}}

\newcommand{\calF}{\mathcal{F}}

\newcommand{\calO}{\mathcal{O}}
\newcommand{\calS}{\mathcal{S}}
\newcommand{\calT}{\mathcal{T}}
\newcommand{\calU}{\mathcal{U}}

\newcommand{\calW}{\mathcal{W}}

\newcommand{\ovcalT}{\overline{\calT}}
\newcommand{\ovcalF}{\overline{\calF}}

\newcommand{\rmb}{\mathrm{b}}
\newcommand{\rmc}{\mathrm{c}}

\newcommand{\sfD}{\mathsf{D}}

\newcommand{\sfK}{\mathsf{K}}

\newcommand{\Cone}{\mathsf{Cone}}
\newcommand{\NR}{\mathsf{NR}}

\renewcommand{\max}{\maks}

\renewcommand{\gcd}{\gcdsf}

\renewcommand{\sup}{\supsf}

\renewcommand{\Gamma}{\varGamma}
\renewcommand{\Delta}{\varDelta}
\renewcommand{\epsilon}{\varepsilon}
\renewcommand{\Theta}{\varTheta}
\renewcommand{\Lambda}{\varLambda}
\renewcommand{\Pi}{\varPi}
\renewcommand{\Sigma}{\varSigma}
\renewcommand{\Phi}{\varPhi}
\renewcommand{\phi}{\varphi}
\renewcommand{\Psi}{\varPsi}
\renewcommand{\Omega}{\varOmega}

\renewcommand{\Im}{\Image}
\renewcommand{\dim}{\dimension}

\title
[Non-rigid regions of real Grothendieck groups of special biserial algebras]
{Non-rigid~regions~of~real~Grothendieck~groups 
of~gentle~and~special~biserial~algebras}
\author{Sota Asai} 
\address{Sota Asai: Department of Pure and Applied Mathematics, Graduate School of Information Science and Technology, Osaka University, 1-5 Yamadaoka, Suita-shi, Osaka-fu, 565-0871, Japan}
\email{s-asai@ist.osaka-u.ac.jp}
\date{\today}

\begin{document}

\begin{abstract}
In the representation theory of finite-dimensional algebras $A$ over a field,
the classification of 2-term (pre)silting complexes is an important problem.
One of the useful tool is the g-vector cones 
associated to the 2-term presilting complexes
in the real Grothendieck group $K_0(\proj A)_\R:=K_0(\proj A) \otimes_\Z \R$.
The aim of this paper is to study the complement $\NR$ of the union $\Cone$ 
of all g-vector cones, which we call the non-rigid region.
By the work of Iyama and us, 
$\NR$ is determined by 2-term presilting complexes and 
a certain closed subset $R_0 \subset K_0(\proj A)_\R$,
which is called the purely non-rigid region.
In this paper, 
we give an explicit description of $R_0$
for complete special biserial algebras
in terms of a finite set of maximal nonzero paths in the Gabriel quiver of $A$.
We also prove that $\NR$ has some kind of fractal property
and that $\NR$ is contained in a union of countably many hyperplanes of codimension one.
Thus, any complete special biserial algebra is g-tame, 
that is, $\Cone$ is dense in $K_0(\proj A)_\R$.
\end{abstract}

\maketitle

\setcounter{tocdepth}{1}
\tableofcontents

\section{Introduction}

In the representation theory of finite-dimensional algebras $A$ 
over an algebraically closed field $K$,
\textit{silting theory} introduced by Keller-Vossieck \cite{KV}
is an important tool to study the derived categories of algebras.
They considered complexes called \textit{silting complexes} satisfying certain conditions 
in the homotopy category $\sfK^\rmb(\proj A)$ of the bound complexes over
the category $\proj A$ of finitely generated projective $A$-modules.
Silting complexes contain \textit{tilting complexes} in $\sfK^\rmb(\proj A)$,
which control \textit{derived equivalences} of finite-dimensional algebras \cite{Rickard}.
One of the nice properties of silting complexes is that
they admit the operation called \textit{mutation}
exchanging any indecomposable direct summand to get another silting complex \cite{AiI}.

Since then, many authors found one-to-one correspondences between silting complexes
and other notions in the category $\fd A$ of finite-dimenisonal modules
and its bounded derived category $\sfD^\rmb(\fd A)$.
For example, Koenig-Yang \cite{KY} constructed bijections called 
the \textit{silting-t-structure correspondences} from
the set of basic silting complexes in $\sfK^\rmb(\proj A)$
to the sets of \textit{bounded t-structures} with length heart 
and \textit{simple-minded collections} in $\sfD^\rmb(\fd A)$.
Moreover, Adachi-Iyama-Reiten \cite{AIR} proved that
there exist bijections between the basic 2-term silting complex 
in $\sfK^\rmb(\proj A)$ and \textit{functorially finite torsion classes} in $\fd A$,
and Br\"{u}stle-Yang \cite{BY} obtained the 2-term versions of the silting-t-structure correspondences.
These bijections have been also developed in terms of 
\textit{wide subcategories} \cite{IT,MS} and \textit{semibricks} \cite{Asai1}.

In this paper, we aim to understand 2-term silting complexes in $\sfK^\rmb(\proj A)$ 
for \textit{complete gentle algebras}.
Complete gentle algebras are an important class of algebras
studied in many works \cite{AY,OPS,PPP1,PPP2} by using \textit{marked surfaces}
as well as \cite{HKK,LP} from the point of view of \textit{Fukaya categories}.
It is also known that 2-term silting complexes for Brauer graph algebras
can be investigated in the same way as complete gentle algebras;
see the explanation before Corollary \ref{Cor_AAC}.
In this paper, we also consider a wider class 
called \textit{complete special biserial algebras},
which can be expressed as quotient algebras of complete gentle algebras.

To study 2-term silting complexes in $\sfK^\rmb(\proj A)$, 
the \textit{Grothendieck group} $K_0(\proj A)$ 
and the real Grothendieck group $K_0(\proj A)_\R:=K_0(\proj A) \otimes_\Z \R$ are useful.
We write $\twosilt A$ (resp.~$\twopresilt A$)
for the set of basic 2-term silting (resp.~presilting) complexes in $\sfK^\rmb(\proj A)$.
For $U=\bigoplus_{i=1}^m U_i \in \twopresilt A$ with each $U_i$ indecomposable, 
we define the \textit{g-vector cones} in $K_0(\proj A)_\R$ by
\begin{align*}
C(U)&:=\left\{ \sum_{i=1}^m a_i[U_i] \mid a_i \in \R_{\ge 0} \right\}, &
C^+(U)&:=\left\{ \sum_{i=1}^m a_i[U_i] \mid a_i \in \R_{>0} \right\}
\end{align*}
as in \cite{DIJ,BST}. 
The elements $[U_1],[U_2],\ldots,[U_m] \in K_0(\proj A)$ 
are linearly independent by \cite{AiI},
so $C(U)$ and $C^+(U)$ are $m$-dimensional.
Moreover, if $U''$ is the maximal common direct summand of $U,U' \in \twopresilt A$,
then $C(U'')=C(U) \cap C(U')$ holds \cite{DIJ}.
Thus, g-vector cones reflect mutation of 2-term silting complexes.

We also consider the subsets
\begin{align*}
\Cone&:=\bigcup_{T \in \twosilt A}C(T)=\bigcup_{U \in \twopresilt A} C^+(U), &
\NR&:=K_0(\proj A)_\R \setminus \Cone.
\end{align*}
We call $\NR$ the \textit{non-rigid region}.
It is well-known that $\NR=\emptyset$ holds if and only if $A$ is 
\textit{$\tau$-tilting finite}, that is, $\twosilt A$ is a finite set \cite{ZZ,Asai2}.

To study $\NR$, the $\R$-bilinear form called the \textit{Euler form}
$K_0(\proj A)_\R \times K_0(\fd A)_\R \to \R$ is helpful.
Via the Euler form,
we can see $K_0(\proj A)_\R$ as a dual $\R$-vector space of $K_0(\fd A)_\R$,
and any element $\theta \in K_0(\proj A)_\R$ is identified with 
$\theta:=\langle \theta,? \rangle \colon K_0(\fd A)_\R \to \R$.
See also the beginning of Subsection \ref{Subsec_Grothendieck}.

By using this, King \cite{King} introduced \textit{stability conditions} 
for modules $M \in \fd A$ and elements $\theta \in K_0(\proj A)_\R$
as a collection of linear inequalities.
In a similar way, Baumann-Kamnitzer-Tingley \cite{BKT} associated
two \textit{numerical torsion pairs} 
$(\ovcalT_\theta,\calF_\theta)$ and 
$(\calT_\theta,\ovcalF_\theta)$ in $\fd A$
to each $\theta \in K_0(\proj A)_\R$.
In the paper \cite{Asai2}, we defined the equivalence relation on $K_0(\proj A)_\R$
called \textit{TF equivalence}: 
we say that $\theta$ and $\theta'$ are \textit{TF equivalent}
if $\calT_\theta=\calT_{\theta'}$ and $\calF_\theta=\calF_{\theta'}$.

The numerical torsion pairs are strongly related to silting theory.
For any $U \in \twopresilt A$,
Yurikusa \cite{Yurikusa} and Br\"{u}stle-Smith-Treffinger \cite{BST} showed that
$\theta \in C^+(U)$ implies that 
$\calT_\theta=\calT_{[U]}$ and $\calF_\theta=\calF_{[U]}$.
By modifying their arguments, 
we showed that $C^+(U)$ is a TF equivalence class in \cite{Asai2};
in other words, 
\begin{align*}
C^+(U)=\{ \theta \in K_0(\proj A)_\R \mid 
\calT_\theta=\calT_{[U]}, \ \calF_\theta=\calF_{[U]} \}.
\end{align*}

It is also useful to consider the open neighborhood 
\begin{align*}
N_U&=\{ \theta \in K_0(\proj A)_\R \mid \calT_{[U]} \subset \calT_\theta, \ 
\calF_{[U]} \subset \calF_\theta \} 
\end{align*}
of $C^+(U)$ associated to each $U \in \twopresilt A$ as in \cite{Asai2}.
By using this neighborhood, 
Iyama and we \cite{AsI} defined the \textit{purely non-rigid region}
\begin{align*}
R_0:=K_0(\proj A)_\R \setminus \bigcup_{U \in \twopresilt A} N_U,
\end{align*}
and proved that 
\begin{align*}
\NR = \coprod_{U \in \twopresilt A} 
\{\theta_1+\theta_2 \mid \theta_1 \in C^+(U),\ 0 \ne \theta_2 \in \overline{N_U} \cap R_0\}.
\end{align*}

Therefore, it is important to determine $R_0$ to understand $\NR$.
In this paper, we describe $\NR$ in the case that 
$A=\widehat{KQ}/I$ is a complete special biserial algebra
by using the following sets of ``maximal'' paths and cycles in $Q$:
\begin{itemize}
\item 
$\MP(A)$ consists of all paths $p$
with $p \ne 0$ and $\alpha p=p \alpha=0$ in $A$ for any arrow $\alpha \in Q_1$;
\item 
$\ovMP(A)$ is the union of $\MP(A)$ and the set of paths $e_i$ of length 0
for $i \in Q_0$ such that there exists at most one arrow starting at $i$ and 
at most one arrow ending at $i$;
\item
$\Cyc(A)$: the set of irreducible cycles $c$ in $Q$ such that $c^m \ne 0$ in $A$
for any $m \in \Z_{\ge 1}$.
\end{itemize}
See Definition \ref{Def_MP} for the precise definition.
We write $M(s)$ for the string module associated to each string $s$ admitted in $A$.
Moreover, we set $p_c:=\alpha_1\alpha_2\cdots\alpha_{l-1}$
if $c=\alpha_1\alpha_2\cdots\alpha_l \in \Cyc(A)$ with $\alpha_i \in Q_1$.
Then, in the case of complete gentle algebras, we have the following result.

\begin{Thm}[Theorem \ref{Thm_R_0_gentle}]\label{Thm_intro_R_0_gentle}
Let $A=\widehat{KQ}/I$ be a complete gentle algebra.
For each $\theta \in K_0(\proj A)_\R$, the following conditions are equivalent.
\begin{itemize}
\item[(a)]
The element $\theta$ belongs to $R_0$.
\item[(b)]
We have $\theta(M(p_c))=0$ for each $c \in \Cyc(A)$, 
and $M(p) \in \calW_\theta$ for each $p \in \ovMP(A)$.
\end{itemize}
In particular, $R_0$ is a rational polyhedral cone of $K_0(\proj A)_\R$,
that is, there exist finitely many elements 
$\theta_1,\theta_2,\ldots,\theta_m \in K_0(\proj A)$ 
such that $R_0=\sum_{i=1}^m \R \theta_i$.
\end{Thm}

More generally, 
we give an explicit desrciption of $R_0$ for complete special biserial algebras
$A=\widehat{KQ}/I$ in Theorem \ref{Thm_R_0_sp} 
by taking a smaller ideal $\widetilde{I} \subset I$
such that $\widetilde{A}:=\widehat{KQ}/\widetilde{I}$ is a complete gentle algebra.
In this case, $R_0$ is a subset of $R_0(\widetilde{A})$, 
but it is not necessarily a rational polyhedral cone or even convex;
see Example \ref{Ex_rectangle_sp}.

As an application, we also prove that 
any complete special biserial algebra $A$ is \textit{g-tame},
that is, $\Cone$ is a dense subset of $K_0(\proj A)_\R$.
For each $U \in \twopresilt A$, 
we write $\calW_U:=\ovcalT_U \cap \ovcalF_U$,
which is a wide subcategory of $\fd A$.
Set $\simple \calW_U$ as the set of isoclasses of simple objects of $\calW_U$,
and $h_U \in K_0(\fd A)$ by $h_U:=\sum_{X \in \simple \calW_U}[X]$.

\begin{Cor}[Corollary \ref{Cor_cone_dense}]\label{Cor_cone_dense_intro}
Let $A$ be a complete special biserial algebra.
Then, we have
\begin{align*}
\NR=K_0(\proj A)_\R \setminus \Cone \subset 
\bigcup_{U \in \twopresilt A, \ |U| \le n-2} \Ker \langle ?,h_U \rangle.
\end{align*}
In particular, $\Cone$ is dense in $K_0(\proj A)_\R$.
\end{Cor}

The g-tameness of complete special biserial algebras has been already proved in \cite{AY},
and it was extended to representation-tame algebras in \cite{PY}.
However, we get the stronger result that $\NR$ 
is contained in a union of countably many hyperplanes of codimension one,
since $\twopresilt A$ is a countable set by \cite{DIJ}.

We also show the following result on mutation of 2-term silting complexes.
We remark that the corresponding property does not hold
for all connected finite-dimensional algebras \cite{Terland}.

\begin{Cor}[following from Corollary \ref{Cor_conn_comp_sp}]\label{Cor_intro_A_A[1]}
If $A$ is a connected special biserial algebra,
then any 2-term silting complex in $\sfK^\rmb(\proj A)$ 
can be obtained by iterated mutations from $A$ or $A[1]$.
\end{Cor}

Our proofs of the two properties above are related to 
\textit{$\tau$-tilting reduction} by Jasso \cite{Jasso}.
It is an important tool to study the set $\twosilt_U A$ of 2-term silting complexes
which have a fixed $U \in \twopresilt A$ as a direct summand.
Jasso considered the algebra $B:=\End_A(H^0(T))/[H^0(U)]$,
where $T \in \twosilt A$ is the \textit{Bongartz completion} of $T$, and 
$[H^0(U)]$ is the ideal of $\End_A(H^0(T))$ consisting
of all endomorphisms factoring through some module in $\add H^0(U)$. 
Then, Jasso proved that there exist a category equivalence
$\calW_U \to \fd B$ of abelian categories and a bijection $\twosilt_U A \to \twosilt B$.
In \cite{Asai2}, we proved that there exists a linear projection 
$\pi \colon K_0(\proj A)_\R \to K_0(\proj B)_\R$
which induces a bijection between
the TF equivalence classes in $N_U$ and those in $K_0(\proj B)_\R$;
see Proposition \ref{Prop_Jasso_Grothendieck}.

To prove the two corollaries above, we also need to understand what the algebra $B$ is.
We show that the class of finite-dimensional special biserial algebras
is closed under $\tau$-tilting reduction.

\begin{Thm}[Theorem \ref{Thm_sp_closed_under_Jasso}]\label{Thm_intro_sp_closed_under_Jasso}
Let $A$ be a finite-dimensional special biserial algebra,
and $U \in \twopresilt A$.
Then, the algebra $B$ above is isomorphic to a finite-dimensional special biserial algebra.
\end{Thm}

Therefore, we can say that the non-rigid region $\NR$ has a
``fractal'' property for a complete special biserial algebra $A$, 
since $\NR$ is a union of some TF equivalence classes;
see Corollary \ref{Cor_fractal} for a precise statement.

For example, Iyama and we obtained in \cite[p.~44]{AsI} 
that $\NR$ is described by the following picture
for the complete gentle algebra of the Markov quiver (Example \ref{Ex_triangle} (2)).
More precisely, the picture below shows the intersection of 
$\NR$ and the surface of the standard octahedron 
\begin{align*}
\left\{ \theta=\sum_{i=1}^3 a_i[P_i] \in K_0(\proj A)_\R \mid 
\sum_{i=1}^3 |a_i|=1 \right\},
\end{align*}
where we omit the two faces 
$\sum_{i=1}^3 \R_{>0} [P_i]$ and $\sum_{i=1}^3 (-\R_{>0}) [P_i]$,
since they do not intersect with $\NR$:
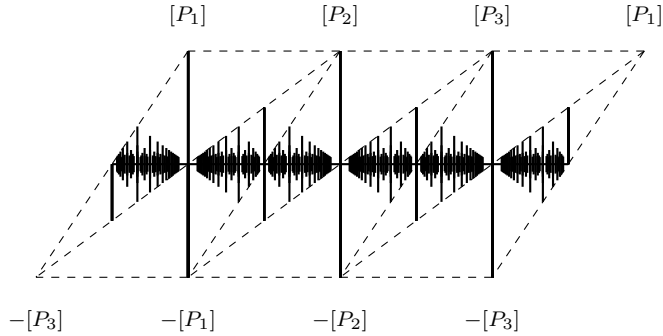
\begin{align*}
\begin{tikzpicture}[every node/.style={circle},baseline=0pt]
\coordinate (X) at (-2,0);
\coordinate (Y) at (0,1.5);
\node (+P0) [coordinate,label=above:{$\scriptstyle{ [P_1]}$}] at ($-1*(X)+(Y)$) {};
\node (+P1) [coordinate,label=above:{$\scriptstyle{ [P_3]}$}] at ($ 0*(X)+(Y)$) {};
\node (+P2) [coordinate,label=above:{$\scriptstyle{ [P_2]}$}] at ($ 1*(X)+(Y)$) {};
\node (+P3) [coordinate,label=above:{$\scriptstyle{ [P_1]}$}] at ($ 2*(X)+(Y)$) {};
\node (-P1) [coordinate,label=below:{$\scriptstyle{-[P_3]}$}] at ($ 0*(X)-(Y)$) {};
\node (-P2) [coordinate,label=below:{$\scriptstyle{-[P_2]}$}] at ($ 1*(X)-(Y)$) {};
\node (-P3) [coordinate,label=below:{$\scriptstyle{-[P_1]}$}] at ($ 2*(X)-(Y)$) {};
\node (-P4) [coordinate,label=below:{$\scriptstyle{-[P_3]}$}] at ($ 3*(X)-(Y)$) {};
\foreach \i in {0,1,2}{
\foreach \n [evaluate=\n as \l using int((\n-1)/2)] in {3,...,16}{
\draw[thick] ($-1/2*(X)+\i*(X)$)--($-1/2*(X)+\i*(X)+1/2*(Y)$);
\foreach \m in {1,...,\l}{
\draw[thick] ($-\m/\n*(X)+\i*(X)-1/\n*(Y)$)--($-\m/\n*(X)+\i*(X)+1/\n*(Y)$);
\draw[thick] ($\i*(X)-(Y)$)--($\i*(X)$);
}
\draw[thick] ($\i*(X)+(Y)$)--($\i*(X)$);
\draw[thick] ($1/2*(X)+\i*(X)$)--($1/2*(X)+\i*(X)-1/2*(Y)$);
\foreach \m in {1,...,\l}{
\draw[thick] ($\m/\n*(X)+\i*(X)+1/\n*(Y)$)--($\m/\n*(X)+\i*(X)-1/\n*(Y)$);
}
}
}
\draw[thick] ($-0.5*(X)$)--($2.5*(X)$);
\draw[dashed] (+P0)--(-P1);
\draw[dashed] (+P1)--(-P2);
\draw[dashed] (+P2)--(-P3);
\draw[dashed] (+P3)--(-P4);
\draw[dashed] (+P0)--(-P2);
\draw[dashed] (+P1)--(-P3);
\draw[dashed] (+P2)--(-P4);
\draw[dashed] (+P0)--(+P3);
\draw[dashed] (-P1)--(-P4);
\end{tikzpicture}.
\end{align*}

Now, we explain the organization of this paper.

In Subsection \ref{Subsec_sp}, 
we recall the precise definition of complete special biserial algebras 
and some basic properties,
including the classification of the indecomposable modules \cite{BR,WW} and 
the homomorphisms between them \cite{CB2,Krause}.
Then, we consider complete gentle algebras in Subsection \ref{Subsec_gentle}.
We give a combinatorial construction of complete gentle algebras 
by lines and cycles
in Definition \ref{Def_gentle_explicit} and Proposition \ref{Prop_gentle_explicit_axiom},
which is known to some experts.
This construction is compatible with the sets $\ovMP(A)$ and $\Cyc(A)$,
which we use to describe the purely non-rigid region $R_0$
in Theorem \ref{Thm_intro_R_0_gentle}.

Section \ref{Sec_pre} is devoted to 
recalling important results on silting theory and canonical decompositions 
in the case of finite-dimensional algebras and modifying them for our setting
of complete special biserial algebras.

In Subsection \ref{Subsec_silt}, we collect some results on the relationship 
between 2-term silting complexes in $\sfK^\rmb(\proj A)$ and 
torsion pairs and bricks in $\fd A$
obtained in $\tau$-tilting theory by \cite{AIR,Asai1}.
We also refer to reduction theorems of bricks and silting complexes 
by \cite{Kimura,VG},
which enable us to deal with a complete special biserial algebra $A$
almost in the same way as the finite-dimensional algebra 
$\overline{A}:=A/I_\rmc$, where $I_\rmc$ denotes the ideal of $A$
generated by $\Cyc(A)$.

Then, we recall basic properties of the Grothendieck groups, numerical torsion pairs, and 
the wall-chamber structures on $K_0(\proj A)_\R$ introduced by \cite{BST,Bridgeland} 
in Subsection \ref{Subsec_Grothendieck}.

In Subsection \ref{Subsec_canon_decomp}, 
we refer to the works \cite{DF,Plamondon} 
on canonical decompositions of elements in the Grothendieck group $K_0(\proj A)$.
Derksen-Fei \cite{DF} introduced a \textit{canonical decomposition}
$\theta=\bigoplus_{i=1}^m \theta_i$ for each element $\theta \in K_0(\proj A)_\R$
based on decompositions of 2-term complexes in $\sfK^\rmb(\proj A)$
into indecomposable direct summands.
This is an analogue of canonical decompositions of dimension vectors 
in quiver representations by Kac \cite{Kac}.
Derksen-Fei and Plamondon \cite{Plamondon} proved that
any element $\theta \in K_0(\proj A)$ has a unique canonical decomposition
up to reordering.
In this subsection, 
we also prepare some properties on indecomposable elements in $K_0(\proj A)$
for complete special biserial algebras, 
which we need to consider the non-rigid region $\NR$.

We deal with $\tau$-tilting reduction in Section \ref{Sec_tau_reduction}.
In Subsection \ref{Subsec_neighbor}, 
we explain basic properties of the neighborhood $N_U$ for each $U \in \twopresilt A$ 
and the (purely) non-rigid region $R_0,\NR$ in more detail
obtained in the prior researches \cite{Jasso,Asai2,AsI}.
Then, we prove Theorem \ref{Thm_intro_sp_closed_under_Jasso} 
in Subsection \ref{Subsec_sp_closed_red}.
We will also show that the class of complete gentle algebras
``is closed'' under $\tau$-tilting reduction in a certain sense;
see Corollary \ref{Cor_gentle_closed_under_Jasso}.

In Section \ref{Sec_R_0}, we will determine
$R_0$ and $\overline{N_U} \cap R_0$ for $U \in \twopresilt A$ explicitly
for any complete special biserial algebra $A$.
Subsection \ref{Subsec_result} is devoted to stating 
the descriptions of the purely non-rigid region $R_0$ 
(including Theorem \ref{Thm_intro_R_0_gentle}) and examples,
and the proofs are done in Subsections \ref{Subsec_key} and \ref{Subsec_proof}.
The subset $\overline{N_U} \cap R_0$ for $U \in \twopresilt A$ 
is considered in Subsection \ref{Subsec_complement}.

We give some applications of our results in Section \ref{Sec_app}.
We first prove Corollary \ref{Cor_cone_dense_intro} 
on the g-tameness of complete special biserial algebras 
in Subsection \ref{Subsec_g-tame}
by using Theorems \ref{Thm_intro_R_0_gentle} and \ref{Thm_intro_sp_closed_under_Jasso}.
In Subsection \ref{Subsec_conn}, we show Corollary \ref{Cor_intro_A_A[1]}
on mutation of 2-term silting complexes.
We will also give other proofs of several known properties by \cite{STV, AAC, Adachi}
on $\tau$-tilting finiteness of special biserial algebras
by applying our results in Subsection \ref{Subsec_fin}.
We finally remark that Mousavand \cite{Mousavand} studied $\tau$-tilting finiteness
of special biserial algebras in a different way.

\subsection{Notations}

In this paper, we assume that the base field $K$ is algebraically closed.
A quiver $Q$ is always a finite quiver,
and the composite $\alpha \beta$ of arrows $\alpha,\beta \in Q_1$
is a path from the source of $\alpha$ to the target of $\beta$. 
The symbol $e_i$ denotes the path of length $0$
associated to each vertex $i \in Q_0$.
We set $\widehat{KQ}$ as the complete path algebra
with respect to the ideal 
$R:=\langle \alpha \mid \alpha \in Q_1 \rangle$ generated by all arrows.
Thus, we can see 
\begin{align*}
\widehat{KQ}=\prod_{\text{$p$: all paths in $Q$}} Kp
\end{align*}
as $K$-vector spaces.
Unless otherwise stated,
$A=\widehat{KQ}/I$ is a complete algebra over the field $K$ 
with $I \subset R^2$ an ideal of $\widehat{KQ}$.
For each $i \in Q_0$,
we write $P_i:=e_i A$ for the indecomposable projective module,
and $S_i$ for its simple top.

We write $\proj A$ for the category of finitely generated right $A$-modules,
and $\sfK^\rmb(\proj A)$ for the homotopy category of bounded complexes on $\proj A$.
Similarly, we define $\fd A$ as the category of finite-dimensional right $A$-modules.
Then, $\fd A$ is an abelian length category;
in particular, $\fd A$ satisfies the Jordan-H\"{o}lder property.

All subcategories in this paper are assumed to be full subcategories 
and closed under isomorphism classes.
For any additive category $\calC$ and any object $M \in \calC$,
we set $\add M$ as the subcategory 
of objects $C \in \calC$ which is isomorphic to a direct summand
of $M^{\oplus s}$ for some $s \ge 1$.
If $\calC=\fd A$, then $\Fac M$ (resp.~$\Sub M$) is defined to be the subcategory 
of objects $X \in \calC$ which has a surjection from (resp.~an injection to)
$M^{\oplus s}$ for some $s \ge 1$.
Moreover, $M^\perp$ (resp.~${^\perp M}$) consists of 
all $X \in \fd A$ such that $\Hom_A(M,X)=0$ (resp.~$\Hom_A(X,M)=0$).

\section{Basic properties of complete special biserial algebras}

\subsection{Complete special biserial algebras}
\label{Subsec_sp}

We first recall the definition of complete special biserial algebras.

\begin{Def}\label{Def_sp_bi_alg}
We say that $A=\widehat{KQ}/I$ is a \textit{complete special biserial algebra}
if $Q$ is a finite quiver and $I \subset \widehat{KQ}$ is an ideal
and the following conditions are satisfied:
\begin{itemize}
\item[(a)]
The ideal $I$ is generated by some finite subset $X \subset \widehat{KQ}$ 
such that each element of $X$ is a path of length $\ge 2$ in $Q$ or
of the form $p-q$ with $p,q$ paths of length $\ge 2$ in $Q$.
\item[(b)]
For each $i \in Q_0$, the number of arrows starting at $i$ in $Q$ is at most two.
\item[(c)]
For each $i \in Q_0$, the number of arrows ending at $i$ in $Q$ is at most two.
\item[(d)]
If $\alpha \in Q_1$ is an arrow ending at $i \in Q_0$ and 
$\beta \ne \gamma \in Q_1$ are arrows starting at $i$,
then $\alpha\beta \in I$ or $\alpha\gamma \in I$.
\item[(e)]
If $\alpha \in Q_1$ is an arrow starting at $i \in Q_0$ and 
$\beta \ne \gamma \in Q_1$ are arrows ending at $i$,
then $\beta\alpha \in I$ or $\gamma\alpha \in I$.
\end{itemize}
Moreover, $A$ is called a \textit{complete string algebra} if we can choose $X$ so that
each element of $X$ is a path of length $\ge 2$.
\end{Def}

Throughout this paper, we assume that $A$ is a complete special biserial algebra
and $n:=\# Q_0$ unless otherwise stated.

In the definition above,
we define $X_+ \supset X$ by adding two paths $p,q$ in $Q$ to $X$
for each element of the form $p-q$ in $X$,
and set $I_+:=\langle X_+ \rangle \supset I$ and $A_+:=\widehat{KQ}/I_+$.
Then, $A_+$ is a complete string algebra.

In this paper, ``maximal nonzero paths'' in the quiver $Q$ play an important role.
Thus, we prepare some symbols on paths in $Q$.

\begin{Def}\label{Def_MP}
We define the following symbols.
\begin{itemize}
\item[(1)]
For $i \in Q_0$, we write $e_i$ for the path at $i$ of length $0$.
\item[(2)]
We define $\MP_*(A),\MP^*(A),\MP(A)$ by
\begin{align*}
\MP_*(A) &:= \{ \text{$p$: paths of length $\ge 1$ in $Q$} \mid
\text{$p \ne 0$ and $p \alpha=0$ in $A_+$ for any $\alpha \in Q_1$} \}, \\
\MP^*(A) &:= \{ \text{$p$: paths of length $\ge 1$ in $Q$} \mid
\text{$p \ne 0$ and $\alpha p=0$ in $A_+$ for any $\alpha \in Q_1$} \}, \\
\MP(A) &:= \MP_*(A) \cap \MP^*(A).
\end{align*}
\item[(3)]
We define $\ovMP_*(A),\ovMP^*(A),\ovMP(A)$ by
\begin{align*}
\ovMP_*(A)&:=\MP_*(A) \cup 
\{e_i \mid \text{there exists at most one arrow starting at $i$} \}, \\
\ovMP^*(A)&:=\MP^*(A) \cup 
\{e_i \mid \text{there exists at most one arrow ending at $i$} \}, \\
\ovMP(A)&:=\ovMP_*(A) \cap \ovMP^*(A).
\end{align*}
\item[(4)]
A cycle $c$ in $Q$ is said to be 
\textit{repeatable} in $A$ if $c^m \ne 0$ in $A$ for all $m \in \Z_{\ge 1}$.
We set $\Cyc(A)$ as the set of repeatable cycles $c$ in $Q$ in $A$
such that no shorter repeatable cycles $c'$ in $Q$ and $m \in \Z_{\ge 2}$ 
satisfy $c=(c')^m$ as paths.
\end{itemize}
\end{Def}

Clearly, the sets defined above do not change when $A$ is replaced to $A_+$.

We give an example.

\begin{Ex}\label{Ex_rectangle_sp_MP}
Let $Q$ be the quiver
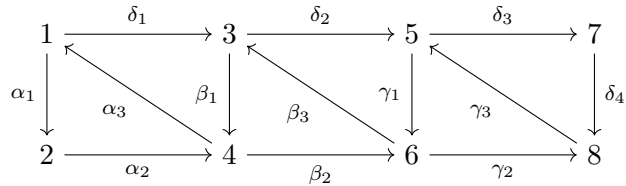
\begin{align*}
\begin{tikzpicture}[scale=0.1, baseline=0pt,->]
\node (1) at ( 0,  8) {$1$};
\node (2) at ( 0, -8) {$2$};
\node (3) at (24,  8) {$3$};
\node (4) at (24, -8) {$4$};
\node (5) at (48,  8) {$5$};
\node (6) at (48, -8) {$6$};
\node (7) at (72,  8) {$7$};
\node (8) at (72, -8) {$8$};
\draw (1) to [edge label'=$\scriptstyle \alpha_1$] (2);
\draw (2) to [edge label'=$\scriptstyle \alpha_2$] (4);
\draw (4) to [edge label =$\scriptstyle \alpha_3$] (1);
\draw (3) to [edge label'=$\scriptstyle \beta_1$] (4);
\draw (4) to [edge label'=$\scriptstyle \beta_2$] (6);
\draw (6) to [edge label =$\scriptstyle \beta_3$] (3);
\draw (5) to [edge label'=$\scriptstyle \gamma_1$] (6);
\draw (6) to [edge label'=$\scriptstyle \gamma_2$] (8);
\draw (8) to [edge label =$\scriptstyle \gamma_3$] (5);
\draw (1) to [edge label =$\scriptstyle \delta_1$] (3);
\draw (3) to [edge label =$\scriptstyle \delta_2$] (5);
\draw (5) to [edge label =$\scriptstyle \delta_3$] (7);
\draw (7) to [edge label =$\scriptstyle \delta_4$] (8);
\end{tikzpicture}
\end{align*}
and $I \subset \widehat{KQ}$ be the ideal generated by the paths
\begin{align*}
&
\alpha_3\delta_1, 
\delta_1\beta_1,  \beta_3 \delta_2,
\alpha_2\beta_2,  \beta_1 \alpha_3,
\delta_2\gamma_1, \gamma_3\delta_3,
\beta_2 \gamma_2, \gamma_1\beta_3,
\delta_4\gamma_3,
\delta_1\delta_2\delta_3\delta_4.
\end{align*}
Set $\alpha_{i+3}:=\alpha_i$, and so on.
Then, we can check that $A:=\widehat{KQ}/I$ is a complete special biserial algebra.
In this case, we get
\begin{align*}
\MP_*(A)&=\{\delta_1\delta_2\delta_3,\delta_2\delta_3\delta_4,
\delta_3\delta_4,\delta_4\},&
\ovMP_*(A)&=\MP_*(A) \cup \{e_2,e_7,e_8\},\\
\MP^*(A)&=\{\delta_1,\delta_1\delta_2,\delta_1\delta_2\delta_3,
\delta_2\delta_3\delta_4\},&
\ovMP^*(A)&=\MP^*(A) \cup \{e_1,e_2,e_7\},\\
\MP(A)&=\{\delta_1\delta_2\delta_3,\delta_2\delta_3\delta_4\},&
\ovMP(A)&=\MP(A) \cup \{e_2,e_7\},
\end{align*}
and 
\begin{align*}
\Cyc(A)&=\{\alpha_i\alpha_{i+1}\alpha_{i+2}, \beta_i\beta_{i+1}\beta_{i+2}, 
\gamma_i\gamma_{i+1}\gamma_{i+2} \mid i \in \{1,2,3\} \}.
\end{align*}
\end{Ex}

We remark that there may exist cycles which are nonzero in $A$ but are not repeatable.

\begin{Ex}
Let $A=K[[x,y]]/(x^2,y^2)$.
Then, $A$ is a complete special biserial algebra,
and $\Cyc(A)=\{xy,yx\}$.
The loops $x$ and $y$ are clearly nonzero in $A$, 
but neither is a repeatable cycle with respect to $A$,
since $x^2=y^2=0$ in $A$.
\end{Ex}

It is easy to see that
any $M \in \fd A$ admits a finite-dimensional quotient algebra $A'$ of $A$ 
such that $M \in \fd A'$,
and the indecomposable modules over finite-dimensional special biserial algebras 
and the homomorphisms between them were completely classified by \cite{BR,WW,Krause,CB2}
in a combinatorial way.
Though we will recall necessary properties below,
we remark that \cite[Section 5]{GLFS} is a good summary of modules over 
special biserial algebras.

We first recall the notion of walks in the quiver $Q$.
We define the \textit{double quiver} $\overline{Q}$ of $Q$
so that the vertices set $\overline{Q}_0$ is $Q_0$,
and that the arrows set $\overline{Q}_1$ is 
\begin{align*}
Q_1 \amalg \{ \alpha^{-1} \colon j \to i \mid (\alpha \colon i \to j) \in Q_1 \}.
\end{align*}
Then, any path in the double quiver $\overline{Q}$ is called a \text{walk} in $Q$.
In other words, a walk can be written as 
$\alpha_1^{\epsilon_1}\alpha_2^{\epsilon_2}\cdots\alpha_l^{\epsilon_l}$
with $l \in \Z_{\ge 0}$, $\alpha_k \in Q_1$ and $\epsilon_k \in \{ \pm 1 \}$
such that the target of $\alpha_k^{\epsilon_k}$ 
is the source of $\alpha_{k+1}^{\epsilon_{k+1}}$ in $\overline{Q}$.

Now, we can define strings and bands in $A$.
First, a walk $s$ in $Q$ is called a \textit{string} admitted in $A$ 
if $s$ does not have a subwalk
of the form $\alpha \alpha^{-1}$, $\alpha^{-1} \alpha$, $p$, or $p^{-1}$, where
$\alpha \in Q_1$ and $p$ is a path in $Q$ such that $p \in I_+$.
Second, a walk $b$ in $Q$ is called a \textit{band} admitted in $A$
if $b^2$ is a string and $b$ admits no subwalk $b'$ such that $b=(b')^m$
for some $m \ge 2$.

For any string $s$ in $A$, the reversed walk $s^{-1}$ is also admitted in $A$,
and $s$ and $s^{-1}$ are the only strings isomorphic to $s$.
On the other hand, 
we say that two bands $b,b'$ in $A$ are \textit{isomorphic as bands} in this paper
if $b'$ is a cyclic permutation of $b$ or $b^{-1}$.

We associate $A$-modules to strings and bands in $A$ as follows.
The construction below is due to \cite{BR,WW}.
Let $s=\alpha_1^{\epsilon_1}\alpha_2^{\epsilon_2}\cdots\alpha_l^{\epsilon_l}$ 
be a string in $A$,
and $i_k$ be the source of $\alpha_k^{\epsilon_k}$ for $k \in \{1,2,\ldots,l\}$,
and $i_{l+1}$ be the target of $\alpha_l^{\epsilon_l}$.
Then, we set a \textit{string module} $M(s)$ by
\begin{itemize}
\item
$M(s):=\bigoplus_{k=1}^{l+1}V_k$ with $V_k:=K$ as a $K$-vector space;
\item
for each $k \in \{1,2,\ldots,l+1\}$, 
the vector space $V_k$ is associated to the vertex $i_k \in Q_0$;
\item
for each $k \in \{1,2,\ldots,l\}$, the action of the arrow $\alpha_k$ is defined so that
$\alpha_k$ induces the identity map $K=V_k \to V_{k+1}=K$ if $\epsilon_k=1$,
and the identity map $K=V_{k+1} \to V_k=K$ if $\epsilon_k=-1$.
\end{itemize}
Next, let $b=\alpha_1^{\epsilon_1}\alpha_2^{\epsilon_2}\cdots\alpha_l^{\epsilon_l}$ 
be a band,
and $i_k$ be the source of $\alpha_k^{\epsilon_k}$ for $k \in \{1,2,\ldots,l\}$, 
and $m \in \Z_{\ge 1}$, $\lambda \in K^\times$.
We set a \textit{band module} $M(b,m,\lambda)$ by
\begin{itemize}
\item
$M(b,m,\lambda):=\bigoplus_{k=1}^l V_k$ with $V_k:=K^m$ as a $K$-vector space;
\item
for each $k \in \{1,2,\ldots,l\}$, 
the vector space $V_k$ is associated to the vertex $i_k \in Q_0$;
\item
for each $k \in \{1,2,\ldots,l-1\}$, the action of the arrow $\alpha_k$ is defined so that
$\alpha_k$ induces the identity map $K^m=V_k \to V_{k+1}=K^m$ if $\epsilon_k=1$,
and the identity map $K^m=V_{k+1} \to V_k=K^m$ if $\epsilon_k=-1$;
\item
the action of the arrow $\alpha_l$ is defined so that
$\alpha_l$ induces the $K$-linear map 
$J_\lambda \colon K^m=V_k \to V_{k+1}=K^m$ if $\epsilon_k=1$,
and the identity map $J_\lambda \colon K^m=V_{k+1} \to V_k=K^m$ if $\epsilon_k=-1$,
where
\begin{align*}
J_\lambda := \lambda \cdot 1_m + \begin{bmatrix} 
0 & 1_{m-1} \\
0 & 0
\end{bmatrix}
\end{align*}
is the Jordan block of size $m$ and eigenvalue $\lambda$.
\end{itemize}
We mainly deal with the case $m=1$ in this paper,
so we write $M(b,\lambda):=M(b,1,\lambda)$,
and call such modules \textit{simple band modules}.

For example, in the setting of Example \ref{Ex_rectangle_sp_MP},
$s:=\beta_1\beta_2\gamma_1^{-1}\delta_3\delta_4$ is a string admitted in $A$,
and $b:=\gamma_2^{-1}\gamma_1^{-1}\delta_3\delta_4$ is a band admitted in $A$.
The corresponding modules can be depicted as
\begin{align*}
M(s)&=\begin{tikzpicture}[baseline=(0.base),->]
\node (0) at (   0,   0) {$\scriptstyle \mathstrut$};
\node (1) at (   0, 0.9) {$\scriptstyle 3$};
\node (2) at ( 0.6, 0.3) {$\scriptstyle 4$};
\node (3) at ( 1.2,-0.3) {$\scriptstyle 6$};
\node (4) at ( 1.8, 0.3) {$\scriptstyle 5$};
\node (5) at ( 2.4,-0.3) {$\scriptstyle 7$};
\node (6) at ( 3.0,-0.9) {$\scriptstyle 8$};
\draw (1) to (2);
\draw (2) to (3);
\draw (4) to (3);
\draw (4) to (5);
\draw (5) to (6);
\end{tikzpicture},&
M(b,\lambda)&=
\begin{tikzpicture}[scale=0.1, baseline=(1.base),->]
\node (1) at ( 0,  0) {$\scriptstyle 6$};
\node (2) at ( 6,  6) {$\scriptstyle 5$};
\node (3) at (12,  0) {$\scriptstyle 7$};
\node (4) at ( 6, -6) {$\scriptstyle 8$};
\draw (2) to (1);
\draw (1) to (4);
\draw (2) to (3);
\draw (3) to [edge label=$\scriptstyle \lambda$] (4);
\end{tikzpicture}.
\end{align*}

We have the following classification of indecomposable modules.

\begin{Prop}\label{Prop_class_module}\cite[Section 3]{BR}\cite[Proposition 2.3]{WW}
Let $M \in \fd A$ be indecomposable.
Then, $M$ is isomorphic to one of the following modules:
\begin{itemize}
\item[(a)]
the string module $M(s)$ for some string $s$ in $A$;
\item[(b)]
the band module $M(b,m,\lambda)$ for some band $b$ in $A$,
$m \in \Z_{\ge 1}$ and $\lambda \in K^\times$;
\item[(c)]
the indecomposable projective-injective module $P_i$ for some vertex $i \in Q_0$
which admits two paths $p,q$ from $i$ such that $p,q \notin I$ and $p-q \in I$.
\end{itemize}
\end{Prop}

Next, we explain the results on the homomorphisms between indecomposable modules
by \cite{CB2,Krause}.
Though their results cover all indecomposable modules,
we only consider string modules and simple band modules here,
which are enough for our purpose.
We need the following symbols.

Let $s=\alpha_1^{\epsilon_1}\alpha_2^{\epsilon_2}\cdots\alpha_l^{\epsilon_l}$ 
be a string in $A$.
For any $u,v \in \{1,2,\ldots,l+1\}$ with $u \le v$, we set 
\begin{align*}
s|_{u,v}:=
\alpha_u^{\epsilon_u}\alpha_{u+1}^{\epsilon_{u+1}}\cdots\alpha_{v-1}^{\epsilon_{v-1}}.
\end{align*}
This is a string from $i_u$ to $i_v$.
In particular, $s_{u,u}=e_{i_u}$.

Let $b=\alpha_1^{\epsilon_1}\alpha_2^{\epsilon_2}\cdots\alpha_l^{\epsilon_l}$ 
be a band in $A$.
For any $u,v \in \{1,2,\ldots,l\}$, we set 
\begin{align*}
b|_{u,v}:=
\begin{cases}
\alpha_u^{\epsilon_u}\alpha_{u+1}^{\epsilon_{u+1}}\cdots\alpha_{v-1}^{\epsilon_{v-1}} 
& (u \le v) \\
\alpha_u^{\epsilon_u}\alpha_{u+1}^{\epsilon_{u+1}}\cdots\alpha_{v+l-1}^{\epsilon_{v+l-1}}
& (u > v) \\
\end{cases},
\end{align*} 
where $\alpha_{i+l}:=\alpha_i$ and $\epsilon_{i+l}:=\epsilon_i$.
This is a string from $i_u$ to $i_v$, and strictly shorter than $b$.
We have $b|_{u,u}=e_{i_u}$ also in this case.

We define the symbols $\fac M$ and $\sub M$ as follows
to describe the homomorphisms between indecomposable modules.

If $M=M(s)$ for some string $s$ in $A$, we set
\begin{align*}
\fac M &:= 
\{(u,v,s|_{u,v}) \mid 1 \le u \le v \le l+1,\ 
\epsilon_{u-1} \in \{-1,0\},\ \epsilon_v \in \{1,0\} \},\\
\sub M &:= 
\{(u,v,s|_{u,v}) \mid 1 \le u \le v \le l+1,\ 
\epsilon_{u-1} \in \{1,0\},\ \epsilon_v \in \{-1,0\} \},
\end{align*}
where we set $\epsilon_0=\epsilon_{l+1}:=0$. 

Similarly, if $M=M(b,\lambda)$ for some band $b$ in $A$, we define
\begin{align*}
\fac M &:= 
\{(u,v,b|_{u,v}) \mid 1 \le u,v \le l,\ \epsilon_{u-1}=-1,\ \epsilon_v=1\},\\
\sub M &:= 
\{(u,v,b|_{u,v}) \mid 1 \le u,v \le l,\ \epsilon_{u-1}=1,\ \epsilon_v=-1\}.
\end{align*}

We finally consider the case $M$ is indecomposable projective-injective.
In this case, we set
\begin{align*}
\fac M &:= \fac(M/{\soc M}), & 
\sub M &:= \sub(\rad M).
\end{align*}

For any $M,M'$ 
which are string modules or simple band modules or 
indecomposable projective-injective modules, we set
\begin{align*}
H_{M,M'}:=\{ ((u,v,t),(u',v',t')) \in \fac M \times \sub M' \mid 
\textup{$t'=t$ or $t'=t^{-1}$} \}.
\end{align*}
For each $h=((u,v,t),(u',v',t')) \in H_{M,M'}$,
we set $f_h \colon M \to M(t) \to M'$ if $t'=t$, and 
$f_h \colon M \to M(t) \cong M(t^{-1}) \to M'$ if $t'=t^{-1}$,
where the first map is the canonical surjection
and the last map is the canonical injection in either case.

\begin{Prop}\label{Prop_special_biserial_hom}
\cite[Section 2, Theorem]{CB2}\cite[Theorem]{Krause}
Let $M,M' \in \fd A$ be string modules, simple band modules or 
indecomposable projective-injective modules.
\begin{itemize}
\item[(1)]
If $M$ and $M'$ are isomorphic and not string modules,
then $f_h$ for all $h \in H_{M,M'}$ and the isomorphism $M \cong M'$ 
give a $K$-basis of $\Hom_A(M,M')$.
\item[(2)]
Otherwise, $f_h$ for all $h \in H_{M,M'}$ give a $K$-basis of $\Hom_A(M,M')$.
\end{itemize}
\end{Prop}

The homomorphisms appearing above are called \textit{standard homomorphisms}.

Also, the Auslander-Reiten translate $\tau M$ of $M \in \fd A$ 
can be combinatorially described 
in the case that $A$ is finite-dimensional.
To explain this, we prepare some symbols.

Let $s=\alpha_1^{\epsilon_1}\alpha_2^{\epsilon_2}\cdots\alpha_l^{\epsilon_l}$ 
be a string in $A$.

Define $l_1,l_2 \in \{0,1,\ldots,l\}$ as the maximal integers such that
$\epsilon_1=\epsilon_2=\cdots=\epsilon_{l_1}=-1$ and
$\epsilon_{l-l_2+1}=\cdots=\epsilon_{l-1}=\epsilon_l=1$, respectively,
and two paths $p_1,p_2$ in $Q$ by $p_1:=(s|_{1,l_1+1})^{-1}$ and
$p_2:=s|_{l-l_2+1,l+1}$.
Then, their lengths are $l_1$ and $l_2$, respectively.
We can check that $l_1+l_2 \ne l-1$, so $l_1+l_2<l$ implies $l_1+l_2 \le l-2$.

Assume $l \ge 1$.
If $p_1 \notin \ovMP_*(A)$, 
then we take $\beta_1 \in Q_1$ so that $p_1 \beta_1 \ne 0$,
and that $\beta_1 \ne \alpha_1$ if moreover $l_1=0$.
Such $\beta_1$ is uniquely determined.
Otherwise, we do not define $\beta_1$.
Similarly, in the case $p_2 \notin \ovMP_*(A)$, 
we define (or do not define) $\beta_2 \in Q_1$ in the same rule.

If $l=0$ and $s=\epsilon_i$, 
then we choose distinct arrows $\beta_1,\beta_2,\ldots,\beta_m$ 
so that $\{\beta_i\}_{i=1}^m \subset Q_1$ is the set of arrows starting at $i$.

For each $i \in \{1,2\}$ such that $\beta_i$ is defined,
let $j_i$ be the target of $\beta_i$.
We can uniquely take a path $q_i \in \ovMP^*(A)$ ending at $j_i$ such that 
the length of $q_i$ is 0 or the last arrow of $q_i$ is not $\beta_i$.

Now, we can explicitly write down the Auslander-Reiten translate of $M(s)$.

\begin{Prop}\label{Prop_AR}\cite[Lemmas 3.1, 3.2]{WW}
Let $A$ be a finite-dimensional special biserial algebra.
\begin{itemize}
\item[(1)]
Let $s$ be a string in $A$,
and $X \subset \{1,2\}$ be the set of $i \in \{1,2\}$ such that $\beta_i$ is defined.
\begin{itemize}
\item[(i)]
If $X=\emptyset$, then
\begin{align*}
\tau M(s)=\begin{cases}
0
& (l_1+l_2=l, \ M(s) \in \proj A) \\
\rad P
& (l_1+l_2=l, \ M(s)=P/{\soc P}, \ P \in \proj A \cap \inj A) \\
M(s|_{l_1+2,l-l_2}) 
& (l_1+l_2 \ne l)
\end{cases}.
\end{align*}
\item[(ii)]
If $X=\{2\}$, then
\begin{align*}
\tau M(s)=\begin{cases}
M(q_2^{-1})
& (l_1=l) \\
M(s|_{l_1+2,l+1} \cdot \beta_2 q_2^{-1}) 
& (l_1 \ne l)
\end{cases}.
\end{align*}
\item[(iii)]
If $X=\{1\}$, then
\begin{align*}
\tau M(s)=\begin{cases}
M(q_1)
& (l_2=l) \\
M(q_1 \beta_1^{-1} \cdot s|_{1,l-l_2}) 
& (l_2 \ne l)
\end{cases}.
\end{align*}
\item[(iv)]
If $X=\{1,2\}$, then
\begin{align*}
\tau M(s)=M(q_1 \beta_1^{-1} \cdot s \cdot  \beta_2 q_2^{-1}).
\end{align*}
\end{itemize}
\item[(2)]
Let $b$ be a band in $A$, and $m \in \Z_{\ge 1}$, $\lambda \in K^\times$.
Then, $\tau M(b,m,\lambda)=M(b,m,\lambda)$. 
\end{itemize}
\end{Prop}

As we have seen above, the original algebra $A$ and the string algebra $A_+$
share almost all properties.
One can see that there is little difference between $A$ and $A_+$ also in our results later.

\subsection{Complete gentle algebras}
\label{Subsec_gentle}

To investigate complete special biserial algebras,
we will use the nice subclass of complete special biserial algebras 
called complete gentle algebras.

\begin{Def}\label{Def_gentle_axiom}
Let $A=\widehat{KQ}/I$ be a complete special biserial algebra.
Then, we call $A$ a \textit{complete gentle algebra}
if $I$ is generated by some set $X$ of paths in $Q$ of length $2$ such that
\begin{itemize}
\item[(a)]
if $\alpha \in Q_1$ is an arrow ending at $i \in Q_0$ and 
$\beta \ne \gamma \in Q_1$ are arrows starting at $i$,
then $\alpha\beta \notin X$ or $\alpha\gamma \notin X$.
\item[(b)]
if $\alpha \in Q_1$ is an arrow starting at $i \in Q_0$ and 
$\beta \ne \gamma \in Q_1$ are arrows ending at $i$,
then $\beta\alpha \notin X$ or $\gamma\alpha \notin X$.
\end{itemize}
\end{Def}

The following property is easily deduced.

\begin{Prop}\label{Prop_sp_quot_gentle}
Let $A=\widehat{KQ}/I$ be a complete special biserial algebra. 
Then, there exists an admissible ideal $\widetilde{I}$ of $\widehat{KQ}$ 
such that $\widetilde{I} \subset I$ and 
$\widetilde{A}:=\widehat{KQ}/\widetilde{I}$ is a complete gentle algebra.
\end{Prop}

We remark that the choice of the ideal $\widetilde{I}$ is not necessarily unique.
This happens, for example, when there exists $i \in Q_0$ such that
$\alpha \in Q_1$ is an arrow ending at $i$, 
$\beta \ne \gamma \in Q_1$ are arrows starting at $i$, 
and that $\alpha\beta=0$, $\alpha\gamma=0$ in $A$.

Complete gentle algebras are constructed also in the following way.
This construction is known to experts.

\begin{Def}\label{Def_gentle_explicit}
Suppose that a pair $(\widetilde{Q},\sim)$ satisfies the following conditions:
\begin{itemize}
\item[(a)]
$\widetilde{Q}$ is the disjoint union $\amalg_{x \in X}{Q^{(x)}}$ with $X=X_1 \amalg X_2$
a finite set and each $Q^{(x)}$ is a quiver
\begin{align*}
\text{(i)} \quad &
\begin{tikzpicture}[baseline=(1.base),->]
\node (1) at (  0,  0) {$(x,1)$};
\node (2) at (1.8,  0) {$(x,2)$};
\node (3) at (3.6,  0) {$\cdots$};
\node (4) at (5.4,  0) {$(x,n_x)$};
\draw (1) to (2);
\draw (2) to (3);
\draw (3) to (4);
\end{tikzpicture}
\quad \text{if $x \in X_1$},\\
\text{(ii)} \quad &
\begin{tikzpicture}[baseline=(1.base),->]
\node (1) at (  0,  0) {$(x,1)$};
\node (2) at (1.8,  0) {$(x,2)$};
\node (3) at (3.6,  0) {$\cdots$};
\node (4) at (5.4,  0) {$(x,n_x)$};
\draw (1) to (2);
\draw (2) to (3);
\draw (3) to (4);
\draw (4) to [bend left=30] (1);
\end{tikzpicture}
\quad \text{if $x \in X_2$}
\end{align*}
with $n_x \in \Z_{\ge 1}$;
\item[(b)]
$\sim$ is an equivalence relation on the vertices set $\widetilde{Q}_0$ such that
\begin{itemize}
\item
every equivalence class $[(x,i)]$ with respect to $\sim$ has at most two elements; and that
\item 
$[(x,1)]=\{(x,1)\}$ if $x \in X_1$ and $n_x=1$.
\end{itemize}
\end{itemize}
Then, we associate a complete gentle algebra $A:=\widehat{KQ}/I$
for the pair $(\widetilde{Q},\sim)$ given as follows:
\begin{itemize}
\item[(a)]
$Q$ is the quiver whose vertices set $Q_0$ and whose arrows set $Q_1$ are 
\begin{align*}
Q_0 &:= \widetilde{Q}_0/{\sim}, &
Q_1 &:= 
\{ \gamma_\alpha \colon [(x,i)] \to [(x,i')] \mid 
(\alpha \colon (x,i) \to (x,i')) \in \widetilde{Q}_1 \};
\end{align*}
\item[(b)]
$I$ is the ideal of $\widehat{KQ}$ generated by 
the paths of the form $\gamma_\alpha \gamma_\beta$ of length 2 in $Q$ 
with $\alpha,\beta \in \widetilde{Q}_1$ such that 
the target $(x,i)$ of $\alpha$ and the source $(y,j)$ of $\beta$ in $\widetilde{Q}$ satisfy
$(x,i) \ne (y,j)$ and $(x,i) \sim (y,j)$ in $\widetilde{Q}_0$.
\end{itemize}
\end{Def}

Note that $Q$ has the same number of arrows as $\widetilde{Q}$,
while it has less vertices than $\widetilde{Q}$ unless
any equivalence class with respect to $\sim$ has one element.
We simply write $\alpha$ for the arrow $\gamma_\alpha$ in $Q$ above
if there is no confusion.

\begin{Ex}
We define a pair $(\widetilde{Q},\sim)$ as follows:
\begin{itemize}
\item[(a)]
$\widetilde{Q}$ is the quiver
\begin{align*}
\begin{tikzpicture}[baseline=(35.base),->]
\node (11) at (  0, 1.6) {$(1,1)$};
\node (12) at (1.8, 1.6) {$(1,2)$};
\node (13) at (3.6, 1.6) {$(1,3)$};
\node (14) at (5.4, 1.6) {$(1,4)$};
\node (15) at (7.2, 1.6) {$(1,5)$};
\node (16) at (9.0, 1.6) {$(1,6)$};
\node (21) at (  0, 0.8) {$(2,1)$};
\node (22) at (1.8, 0.8) {$(2,2)$};
\node (31) at (  0,   0) {$(3,1)$};
\node (32) at (1.8,   0) {$(3,2)$};
\node (33) at (3.6,   0) {$(3,3)$};
\node (34) at (5.4,   0) {$(3,4)$};
\node (35) at (7.2,   0) {$(3,5)$};
\draw (11) to [edge label=$\scriptstyle \alpha_1$] (12);
\draw (12) to [edge label=$\scriptstyle \alpha_2$] (13);
\draw (13) to [edge label=$\scriptstyle \alpha_3$] (14);
\draw (14) to [edge label=$\scriptstyle \alpha_4$] (15);
\draw (15) to [edge label=$\scriptstyle \alpha_5$](16);
\draw (21) to [edge label=$\scriptstyle \beta$] (22);
\draw (31) to [edge label=$\scriptstyle \gamma_1$] (32);
\draw (32) to [edge label=$\scriptstyle \gamma_2$] (33);
\draw (33) to [edge label=$\scriptstyle \gamma_3$] (34);
\draw (34) to [edge label=$\scriptstyle \gamma_4$] (35);
\draw (35) to [edge label=$\scriptstyle \gamma_5$, bend left=15] (31);
\end{tikzpicture};
\end{align*}
\item[(b)]
$\sim$ is the equivalence relation on $\widetilde{Q}_0$ such that
the nontrivial equivalences are
$(1,2) \sim (1,5)$, $(1,3) \sim (3,1)$, $(1,6) \sim (2,1)$ and $(3,3) \sim (3,5)$.
\end{itemize}
Then, this pair gives the gentle algebra $A:=\widehat{KQ}/I$, where $Q$ is the quiver
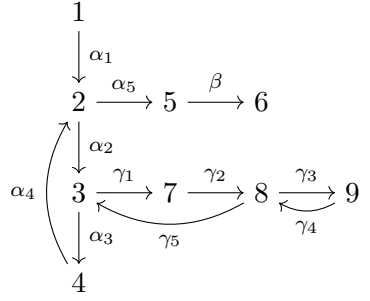
\begin{align*}
\begin{tikzpicture}[baseline=0pt,->]
\node (1) at (  0, 1.8) {$1$};
\node (2) at (  0, 0.6) {$2$};
\node (3) at (  0,-0.6) {$3$};
\node (4) at (  0,-1.8) {$4$};
\node (5) at (1.2, 0.6) {$5$};
\node (6) at (2.4, 0.6) {$6$};
\node (7) at (1.2,-0.6) {$7$};
\node (8) at (2.4,-0.6) {$8$};
\node (9) at (3.6,-0.6) {$9$};
\draw (1) to [edge label=$\scriptstyle \alpha_1$] (2);
\draw (2) to [edge label=$\scriptstyle \alpha_2$] (3);
\draw (3) to [edge label=$\scriptstyle \alpha_3$] (4);
\draw (4) to [edge label=$\scriptstyle \alpha_4$, bend left=30] (2);
\draw (2) to [edge label=$\scriptstyle \alpha_5$] (5);
\draw (5) to [edge label=$\scriptstyle \beta$]    (6);
\draw (3) to [edge label=$\scriptstyle \gamma_1$] (7);
\draw (7) to [edge label=$\scriptstyle \gamma_2$] (8);
\draw (8) to [edge label=$\scriptstyle \gamma_3$] (9);
\draw (9) to [edge label=$\scriptstyle \gamma_4$, bend left=30] (8);
\draw (8) to [edge label=$\scriptstyle \gamma_5$, bend left=30] (3);
\end{tikzpicture},
\end{align*}
and the ideal $I \subset \widehat{KQ}$ is generated by the paths
\begin{align*}
&
\alpha_1\alpha_5, \alpha_4\alpha_2, 
\alpha_5\beta, 
\gamma_5\alpha_3, \alpha_2\gamma_1,
\gamma_2\gamma_5, \gamma_4\gamma_3.
\end{align*}
Therefore, 
\begin{align*}
\MP(A)&=\{\alpha_1\alpha_2\alpha_3\alpha_4\alpha_5, \beta\}, \\
\ovMP(A)&=\MP(A) \cup \{e_1, e_4, e_5, e_6, e_7, e_9\}, \\
\Cyc(A)&=\{ \gamma_i\gamma_{i+1}\cdots\gamma_{i+4} \mid i \in \{1,2,\ldots,5\}\},
\end{align*}
where $\gamma_{i+5}:=\gamma_i$.
\end{Ex}

Actually, all complete gentle algebras are obtained in the way above.
This fact is well-known among experts,
but we include the proof for the convenience of the reader.

\begin{Prop}\label{Prop_gentle_explicit_axiom}
Any complete gentle algebra is obtained from some pair $(\widetilde{Q},\sim)$
in Definition \ref{Def_gentle_explicit}.
\end{Prop}

\begin{proof}
We may assume that every $i \in Q_0$ has some arrow starting or ending at $i$.
We construct a pair $(\widetilde{Q},\sim)$.

We consider an equivalence relation on $\Cyc(A)$ such that
$p,q \in \Cyc(A)$ are equivalent if and only if $q$ is a cyclic permutation of $p$.
We can take a complete system $X_2 \subset \Cyc(A)$ of represensetatives 
with respect to this equivalence relation.
Set $X_1:=\MP(A)$ and $X:=X_1 \amalg X_2$.

For each $p \in X$ of length $l_p$,
we set a quiver $Q^{(p)}$ by the following rule:
\begin{itemize}
\item
If $p \in X_1$, then $Q^{(p)}$ is
\begin{align*}
\begin{tikzpicture}[baseline=(1.base),->]
\node (1) at (  0,  0) {$(p,1)$};
\node (2) at (1.8,  0) {$(p,2)$};
\node (3) at (3.6,  0) {$\cdots$};
\node (4) at (5.4,  0) {$(p,l_p)$};
\draw (1) to (2);
\draw (2) to (3);
\draw (3) to (4);
\end{tikzpicture}.
\end{align*}
\item
If $p \in X_2$, then $Q^{(p)}$ is
\begin{align*}
\begin{tikzpicture}[baseline=(1.base),->]
\node (1) at (  0,  0) {$(p,1)$};
\node (2) at (1.8,  0) {$(p,2)$};
\node (3) at (3.6,  0) {$\cdots$};
\node (4) at (5.4,  0) {$(p,l_p)$};
\draw (1) to (2);
\draw (2) to (3);
\draw (3) to (4);
\draw (4) to [bend left=30] (1);
\end{tikzpicture}.
\end{align*}
\end{itemize}

Now, we set $\widetilde{Q}:=\coprod_{p \in X} Q^{(p)}$.
Then, we introduce an equivalence relation $\sim$ on the vertices set $\widetilde{Q}_0$
so that $(p,u) \sim (q,v)$ if and only if $i_u=j_v$ in $Q_0$, where
\begin{align*}
p&=(\begin{tikzpicture}[baseline=(1.base),->]
\node (1) at (  0,  0) {$i_1$};
\node (2) at (1.2,  0) {$i_2$};
\node (3) at (2.4,  0) {$\cdots$};
\node (4) at (3.6,  0) {$i_{l_p}$};
\node (5) at (4.8,  0) {$i_1$};
\draw (1) to (2);
\draw (2) to (3);
\draw (3) to (4);
\draw[dashed] (4) to (5);
\end{tikzpicture}), &
q&=(\begin{tikzpicture}[baseline=(1.base),->]
\node (1) at (  0,  0) {$j_1$};
\node (2) at (1.2,  0) {$j_2$};
\node (3) at (2.4,  0) {$\cdots$};
\node (4) at (3.6,  0) {$j_{l_q}$};
\node (5) at (4.8,  0) {$j_1$};
\draw (1) to (2);
\draw (2) to (3);
\draw (3) to (4);
\draw[dashed] (4) to (5);
\end{tikzpicture}),
\end{align*}
where the dashed last arrow in $p$ exist if and only if $p \in X_2$,
and similar for $q$.

Since $A$ is a complete gentle algebra, each arrow $\alpha \in Q_1$ admits
exactly one $p \in \MP(A) \cup \Cyc(A)$ such that $\alpha$ appears in $p$.
Thus, each vertex $i \in Q_0$ is 
involved in at most two $p \in \MP(A) \cup \Cyc(A)$.
This means that each equivalence class with respect to $\sim$ contains at most two elements.

Now, we define a quiver $Q$ and an ideal $I \subset \widehat{KQ}$ as in Definition 
\ref{Def_gentle_explicit}, then we have $A \cong \widehat{KQ}/I$.
\end{proof}

By using this construction, 
we can describe maximal nonzero paths of complete gentle algebras as follows.

\begin{Lem}\label{Lem_MP_cycle_explicit}
Let $A$ be the complete gentle algebra for a pair $(\widetilde{Q},\sim)$ satisfying
the conditions in Definition \ref{Def_gentle_explicit}.
\begin{itemize}
\item[(1)]
The set $\MP(A)$ consists of the paths
$[(x,1)] \to [(x,2)] \to \cdots \to [(x,n_x)]$ for $x \in X_1$ with
$n_x \ge 2$;
\item[(2)]
The set $\ovMP(A)$ consists of the paths in $\MP(A)$ and
$e_{[(x,1)]}$ for $x \in X_1$ such that $(x,1) \sim (y,n_y)$ for some $y \in X_1$.
\item[(3)]
The set $\Cyc(A)$ consists of the cycles
$[(x,i)] \to [(x,i+1)] \to \cdots \to [(x,i)]$ 
for $x \in X_2$ and $i \in \{1,2,\ldots,n_x\}$.
\end{itemize}
\end{Lem}

\section{Preliminary on silting theory and canonical decompositions}
\label{Sec_pre}

\subsection{Torsion pairs and 2-term silting complexes}
\label{Subsec_silt}

To study complete special algebras, we would like to use existing results 
on the real Grothendieck group $K_0(\proj A)_\R$ for finite-dimensional algebras.
However, complete special algebras are not necessarily finite-dimensional,
so we need some preparation.

Let $\calT,\calF \subset \fd A$ be full subcategories.
Then, we call the pair $(\calT,\calF)$ a \textit{torsion pair} in $\fd A$ if
$\calT={^\perp \calF}$ and $\calF=\calT^\perp$.
We can check that $(\calT,\calF)$ is a torsion pair if and only if
$\Hom_A(\calT,\calF)=0$ and any $M \in \fd A$ has a short exact sequence
$0 \to M' \to M \to M'' \to 0$ such that $M' \in \calT$ and $M'' \in \calF$.
A full subcategory $\calT \subset \fd A$ is called a \textit{torsion class}
if $\calT$ is closed under taking extensions and quotients.
In this case, $(\calT,\calT^\perp)$ is a torsion pair, 
since $\fd A$ is an abelian length category.
Dually, $\calF \subset \fd A$ is called a \textit{torsion-free class}
if $\calF$ is closed under taking extensions and submodules.
We write $\tors A$ (resp.~$\torf A$) for the set of torsion classes 
(resp.~torsion-free classes) in $\fd A$.

A finite-dimensional module 
$S \in \fd A$ is called a \textit{brick} if $\End_A(S) \cong K$,
and write $\brick A$ for the set of isomorphism classes of bricks in $\fd A$.
Then, the following property originally proved for finite-dimensional algebras hold
also in our setting, since $\fd A$ is an abelian length category.

\begin{Lem}\label{Lem_DIRRT}
Let $\calT \in \tors A$ and $\calF \in \torf A$ in $\fd A$.
\begin{itemize}
\item[(1)]
\cite[Lemma 3.8]{DIRRT}
If $\calT \cap \calF \ne \{0\}$, then there exists $S \in \brick A$ such that
$S \in \calT \cap \calF$.
\item[(2)]
\cite[Lemma 3.9]{DIRRT}
The torsion class $\calT$ is the smallest torsion class containing $\calT \cap \brick A$,
and the torsion class $\calF$ is the smallest torsion-free class containing $\calF \cap \brick A$.
\end{itemize}
\end{Lem}

In the rest of this section, 
let $J \subset I_\rmc :=\langle \Cyc(A) \rangle$,
and $\overline{A}:=A/J$.
The next proposition coming from results of Kimura \cite{Kimura} 
(cf.~\cite{EJR}) is crucial in our study. 

\begin{Prop}\label{Prop_reduction_brick}
The following properties hold.
\begin{itemize}
\item[(1)]\cite[Lemma 5.1]{Kimura}
If $S \in \brick A$, then $SJ=0$.
Thus, $\brick A=\brick \overline{A}$.
\item[(2)]\cite[Theorem 5.4]{Kimura}
There exists a bijection $\tors A \to \tors \overline{A}$
given by $\calT \mapsto \calT \cap \fd \overline{A}$
preserving inclusions.
\end{itemize}
\end{Prop}

\begin{proof}
We first consider the element $x:=\sum_{c \in \Cyc(A)}c$. 
We can check that $x$ is in the center $Z(A)$,
and the elements $x^k$ for $k \in \{1,2\ldots,\}$ are 
linearly independent over $K$.
Thus, $A$ can be considered as a module-finite $K[[t]]$-algebra by 
$K[[t]] \ni t \mapsto x \in A$.
Then, we can apply \cite[Lemma 5.1, Theorem 5.4]{Kimura} to get (1) and (2) 
for any ideal $J \subset \langle x \rangle$.

We move to the general case $J \subset I_\rmc$.
Let $c \in \Cyc(A)$ be from $i \in Q_0$ to $i \in Q_0$. 
If there exists $c' \in \Cyc(A) \setminus \{c\}$ such that $c'$ is also from $i$ to $i$
(such $c'$ is unique if exists),
then $c=-c'$ in $A/\langle x \rangle$; otherwise, $c=0$.
Thus, $A/\langle x \rangle$ is isomorphic to 
a finite-dimensional special biserial algebra.
If $M$ is an indecomposable module in $\fd(A/\langle x \rangle)$
which is not in $\fd(A/I_\rmc)$, then 
$M$ is a projective-injective module whose top and socle coincide 
by Proposition \ref{Prop_class_module},
so $M$ is not a brick.
Therefore, we get $\brick (A/\langle x \rangle)=\brick (A/I_\rmc)$ and (1).
This implies (2) by Lemma \ref{Lem_DIRRT},
as in the proof of \cite[Lemma 5.3, Theorem 5.4]{Kimura}.
\end{proof}

Next, we recall some notions in silting theory of the homotopy category
$\sfK^\rmb(\proj A)$.
Since $A/I_\rmc^i$ is clearly finite-dimensional for all $i$,
the homotopy category $\sfK^\rmb(\proj A)$ is Krull-Schmidt by \cite[Corollary 4.6]{KM}.
Therefore, any $U \in \sfK^\rmb(\proj A)$ has a unique decomposition
$U=\bigoplus_{i=1}^m U_i$ in $\sfK^\rmb(\proj A)$ into indecomposable objects up to 
reordering.
Thus, we write $|U|$ for the number of non-isomorphic indecomposable direct summands of 
$U \in \sfK^\rmb(\proj A)$, so $n=\# Q_0=|A|$.
In this notation, if $U_i \not\cong U_j$ for any $i \ne j$,
then $U$ is said to be \textit{basic}.
We say that a complex $U \in \sfK^\rmb(\proj A)$ is \textit{2-term}
if the terms of $U$ except $-1$st and 0th ones vanish.

\begin{Def}
Let $U$ be a 2-term complex in $\sfK^\rmb(\proj A)$.
\begin{itemize}
\item[(1)]
The complex $U$ is said to be \textit{2-term presilting}
if $\Hom_{\sfK^\rmb(\proj A)}(U,U[k])=0$ for any $k \in \Z_{>0}$.
We write $\twopresilt A$ 
for the set of isoclasses of basic 2-term presilting complexes in $\sfK^\rmb(\proj A)$.
\item[(2)]
The complex $U$ is said to be \textit{2-term silting}
if $U$ is 2-term presilting and 
the smallest thick subcategory generated by $U$ is $\sfK^\rmb(\proj A)$ itself.
We write $\twosilt A$ 
for the set of isoclasses of basic 2-term silting complexes in $\sfK^\rmb(\proj A)$.
\end{itemize}
\end{Def}

Note that we have to check only $\Hom_{\sfK^\rmb(\proj A)}(U,U[1])=0$ in (1),
since $U$ is assumed to be a 2-term complex.

Thanks to the following properties,
we can deal with 2-term presilting complexes as in the case of finite-dimensional algebras.
In the rest, we set $\overline{U}:=U \otimes_A \overline{A}$ 
for any 2-term complex $U \in \sfK^\rmb(\proj A)$.

\begin{Prop}\label{Prop_reduction_silt}
The following statements hold.
\begin{itemize}
\item[(1)]\cite[Propositions 4.2, 4.4]{Kimura}
For any 2-term complexes $U,V \in \sfK^\rmb(\proj A)$,
the condition $\Hom_A(U,V[1])=0$ holds if and only if 
$\Hom_{\overline{A}}(\overline{U},\overline{V}[1])=0$.
\item[(2)]\cite[Lemma 2.6]{VG}
There exist bijections $\twosilt A \to \twosilt \overline{A}$ 
and $\twopresilt A \to \twopresilt \overline{A}$ given by 
$U \mapsto \overline{U}$.
Moreover, $|U|=|\overline{U}|$ holds.
\end{itemize}
\end{Prop}

Therefore, we can apply the following fundamental properties of 2-term silting complexes
to complete special biserial algebras.

\begin{Prop}\label{Prop_silt_fund}
For any $U \in \twopresilt A$, we have the following assertions.
\begin{itemize}
\item[(1)]\cite[Proposition 2.17]{Aihara}
There exists some $T \in \twosilt A$ which has $U$ as a direct summand.
\item[(2)]\cite[Proposition 3.3]{AIR}
The condition $U \in \twosilt A$ holds if and only if $|U|=n$.
\item[(3)]\cite[Proposition 3.8]{AIR}
If $|U|=n-1$, then there exist exactly two $T \in \twosilt A$
which have $U$ as a direct summand.
\end{itemize}
\end{Prop}

Let $U \in \twopresilt A$.
If $A$ is finite-dimensional, then 
we can consider the Nakayama functor $\nu \colon \sfK^\rmb(\proj A) \to \sfK^\rmb(\inj A)$.
Then, $H^0(U)$ is a \textit{$\tau$-rigid module},
and $H^{-1}(\nu U)$ is a \textit{$\tau^{-1}$-rigid module},
that is, $\Hom_A(H^0(U),\tau H^0(U))=0$ and 
$\Hom_A(\tau^{-1}(H^{-1}(\nu U)),H^{-1}(\nu U))=0$
\cite[Lemma 3.4]{AIR},
so we have two functorially finite torsion pairs
$({^\perp H^{-1}(\nu U)},\Sub H^{-1}(\nu U))$ and $(\Fac H^0(U),H^0(U)^\perp)$
in $\fd A$ by \cite[Theorem 5.10]{AS}.
This is extended in our situation as follows.

\begin{Def}\label{Def_silt_tors}
Let $U \in \twopresilt A$, and $J \subset I_\rmc$ satisfy 
that $\overline{A}$ is finite-dimensional.
Then, we set two torsion pairs 
$(\ovcalT_U,\calF_U)$ and $(\calT_U,\ovcalF_U)$ in $\fd A$ so that
\begin{align*}
\ovcalT_U \cap \fd \overline{A} &= {^\perp H^{-1}(\nu \overline{U})}, &
\calF_U \cap \fd \overline{A} &= \Sub H^{-1}(\nu \overline{U}), \\
\calT_U \cap \fd \overline{A} &= \Fac H^0(\overline{U}), & 
\ovcalF_U \cap \fd \overline{A} &= H^0(\overline{U})^\perp,
\end{align*}
which are uniquely defined by Proposition \ref{Prop_reduction_brick},
not depending on the choice of $J$.
Moreover, we set $\calW_U:=\ovcalT_U \cap \ovcalF_U$.
\end{Def}

These two torsion pairs coincide if and only if $U \in \twosilt A$
\cite[Proposition 2.16, Theorems 2.12, 3.2]{AIR}.
If $A$ is a finite-dimensional algebra, then
the correspondences $T \mapsto \calT_T=\ovcalT_T$ and $T \mapsto \calF_T=\ovcalF_T$
induce bijections $\twosilt A \mapsto \ftors A$ and 
$\twosilt A \mapsto \ftorf A$, respectively \cite[Theorem 2.7]{AIR}.
Here, $\ftors \overline{A}$ (resp.~$\ftorf \overline{A}$) 
is the set of functorially finite torsion 
(resp.~ torsion-free) classes in $\fd A$.
Therefore, for each $U \in \twopresilt A$,
there uniquely exists $T \in \twosilt A$ such that $\ovcalT_T=\ovcalT_U$
and dually, there uniquely exists $T' \in \twosilt A$ such that $\ovcalF_{T'}=\ovcalF_U$.
This property is verified in our setting of complete special biserial algebras
by Proposition \ref{Prop_reduction_brick}.

\begin{Def}\label{Def_Bongartz}
Let $U \in \twopresilt A$.
Then, the \textit{Bongartz completion} of $U$ is defined as 
the unique $T \in \twosilt A$ such that $\ovcalT_T=\ovcalT_U$ 
obtained from the bijections 
$\twosilt A \to \twosilt \overline{A} \to \ftors \overline{A}$.
Similarly, the \textit{Bongartz cocompletion} of $U$ is set as 
the unique $T' \in \twosilt A$ such that $\ovcalF_{T'}=\ovcalF_U$ 
given by the bijection $\twosilt A \to \twosilt \overline{A} \to \ftorf \overline{A}$.
\end{Def}

Note that neither $(\ovcalT_U,\calF_U)$ nor $(\calT_U,\ovcalF_U)$
depends on the choice of the ideal $J$.
We also remark that $(\ovcalT_U,\calF_U)=(\calT_U,\ovcalF_U)$
if and only if $U \in \twosilt A$.

To give nice ``generators'' of the torsion classes and 
the torsion-free classes, we recall the notion of semibricks.
Let $\calS \subset \brick A$.
Then, we say that $\calS$ is a \textit{semibrick}
if $\Hom_A(S,S')=0$ for any $S \ne S' \in \calS$.
By Proposition \ref{Prop_reduction_brick}, we get the following useful property.

\begin{Lem}\label{Lem_tors_gen}
Let $\calT \in \tors A$, and $\calS \subset \brick \overline{A}$ be a semibrick 
in $\fd \overline{A}$. 
If $\calT \cap \fd \overline{A}$ 
is the smallest torsion class in $\fd \overline{A}$ containing $\calS$,
then $\calT$ is the smallest torsion class in $\fd A$ containing $\calS$.
\end{Lem}

Assume that $\overline{A}$ is finite-dimensional.
Since $H^0(\overline{U})$ is a $\tau$-rigid $\overline{A}$-module,
by \cite[Lemma 2.5 (5)]{Asai1},
there uniquely exists a semibrick $\calS_U \subset \brick \overline{A}$ 
such that $\calT_U \cap \fd \overline{A}=\Fac H^0(\overline{U})$ 
is the smallest torsion class in $\fd \overline{A}$ containing $\calS_U$.
Then, $\calT_U$ is 
the smallest torsion class in $\fd A$ containing $\calS_U$
by Lemma \ref{Lem_tors_gen}.
Similarly, we can define a semibrick $\calS'_U \subset \sbrick \overline{A}$ 
such that $\calF_U=\Sub H^{-1}(\nu \overline{U})$ 
is the smallest torsion class in $\fd A$ containing $\calS'_U$.
The explicit descriptions of $\calS_U$ and $\calS'_U$ are given as follows.

\begin{Lem}\label{Lem_semibrick}
Let $U \in \twopresilt A$, and $J \subset I_\rmc$ satisfy 
that $\overline{A}$ is finite-dimensional.
Set $B:=\End_{\overline{A}}(H^0(\overline{U}))$ and 
$B':=\End_{\overline{A}}(H^{-1}(\nu \overline{U}))$.
Decompose $U=\bigoplus_{i=1}^m U_i$ with $U_i$ indecomposable,
and define
\begin{align*}
X_i&:=H^0(\overline{U_i})/
\sum_{f \in \rad_A(H^0(\overline{U}),H^0(\overline{U_i}))} \Im f, &
X'_i&:=\bigcap_{f \in \rad_A(H^{-1}(\nu \overline{U_i}),H^{-1}(\nu \overline{U}))} \Ker f
\end{align*}
for each $i$.
Then, the following assertions hold.
\begin{itemize}
\item[(1)]\cite[Theorem 2.3]{Asai1}
We have
\begin{align*}
\calS_U&=\ind (H^0(\overline{U})/{\rad_B H^0(\overline{U})})
=\left\{ X_i \mid i \in \{1,2,\ldots,m\} \right\} \setminus \{0\}, \\
\calS'_U&=\ind (\soc_{B'}H^{-1}(\nu \overline{U}))
=\left\{ X'_i \mid i \in \{1,2,\ldots,m\} \right\} \setminus \{0\}.
\end{align*}
Thus, neither $\calS_U$ nor $\calS'_U$ depends on the choice of $J$.
\item[(2)]
For any $i,j \in \{1,2,\ldots,m\}$, we have
\begin{align*}
\Hom_A(H^0(\overline{U_i}),X_j) &\cong \begin{cases}
K & (i=j, \ X_j \ne 0) \\
0 & (\textup{otherwise})
\end{cases}, \\
\Hom_A(X'_j,H^{-1}(\nu \overline{U_i})) &\cong \begin{cases}
K & (i=j,\ X'_j \ne 0) \\
0 & (\textup{otherwise})
\end{cases}.
\end{align*}
\item[(3)]
For any $i$, we have $X_i \ne 0$ or $X'_i=0$.
\end{itemize}
\end{Lem}

\begin{proof}
(2)
Take the Bongartz completion $T \in \twosilt A$ of $U$.
Then, $\overline{T} \in \twosilt \overline{A}$ is the Bongartz completion of $\overline{U}$.
Then, we can apply \cite[Theorem 3.3, Lemma 3.14]{Asai1} (see also \cite[Lemma 5.3]{KY})
to $\overline{T}$, and have
\begin{align*}
\Hom_A(H^0(\overline{U_i}),X_j) &\cong \Hom_{\sfD^\rmb(\fd A)}(\overline{U_i},X_j) \cong
\begin{cases}
K & (i=j, \ X_j \ne 0) \\
0 & (\textup{otherwise})
\end{cases}.
\end{align*}
The other isomorphism is similarly shown by using the Bongartz cocompletion.

(3)
Assume that $X'_i=0$.
Then, since we can regard
\begin{align*}
\rad_A(H^{-1}(\nu \overline{U_i}),H^{-1}(\nu \overline{U}))
\subset \rad_A(H^{-1}(\nu \overline{U_i}),H^{-1}(\nu \overline{T})), 
\end{align*}
we have
\begin{align*}
\bigcap_{f \in \rad_A(H^{-1}(\nu \overline{U_i}),H^{-1}(\nu \overline{T}))} \Ker f=0.
\end{align*}
Obviously, $\overline{U_i}$ is an indecomposable direct summand of $\overline{T}$,
so by \cite[Lemma 3.13 (1)]{Asai1},
\begin{align*}
H^0(\overline{U_i})/\sum_{f \in \rad_A(H^0(\overline{T}),H^0(\overline{U_i}))} \Im f \ne 0.
\end{align*}
This implies that $X_i \ne 0$.
\end{proof}

\subsection{Grothendieck groups}
\label{Subsec_Grothendieck}

We next prepare some notions on Grothendieck groups.
We first remark that 
the definitions we give below are specialized for complete special biserial algebras,
and that van Garderen \cite[Section 2.1]{VG} 
summarized important properties of Grothendieck groups
for all algebras over complete local Noetherian commutative rings.

In our setting, the complete special biserial algebra 
$A$ is given by a quiver and relations,
so the projective $A$-modules $P_i=e_i A$ for all $i \in Q_0$ 
are the indecomposable objects of $\proj A$, 
and the simple tops $S_i$ of $P_i$ are the simple objects of 
the abelian length category $\fd A$.
Thus, $(P_i)_{i \in Q_0}$ is a $\Z$-basis of $K_0(\proj A)$ and 
$(S_i)_{i \in Q_0}$ is a $\Z$-basis of $K_0(\fd A)$.
Then, $\dim_K \Hom_A(P_i,S_j)=\delta_{ij}$.

Therefore, we can define the \textit{Euler form} 
$\langle ?,! \rangle \colon K_0(\proj A) \times K_0(\fd A) \to \Z$
as in the case of finite-dimensional algebras
by $\langle P_i,S_j \rangle:=\delta_{ij}$.

The notions above are easily extended to the real Grothendieck groups
$K_0(\proj A)_\R:=K_0(\proj A) \otimes_\Z \R$ and 
$K_0(\fd A)_\R:=K_0(\fd A) \otimes_\Z \R$,
and we can regard each 
$\theta \in K_0(\proj A)_\R$ is an $R$-linear form
$\langle \theta,? \rangle=K_0(\fd A)_\R \to \R$.

Now, we can define stability conditions 
and numerical torsion pairs for complete special algebras $A$.

\begin{Def}
Let $\theta \in K_0(\proj A)_\R$.
\begin{itemize}
\item[(1)]
\cite[Definition 1.1]{King}
Let $M \in \fd A$.
Then, $M$ is said to be \textit{$\theta$-semistable}
if $\theta(M)=0$ and any quotient module $N$ of $M$ satisfies $\theta(N) \ge 0$.
We define $\calW_\theta \subset \fd A$ 
as the category of all $\theta$-semistable modules $M \in \fd A$.
\item[(2)]
\cite[Subsection 3.1]{BKT}
We define the following torsion pairs $(\ovcalT_\theta,\calF_\theta)$ 
and $(\calT_\theta,\ovcalF_\theta)$ of $\fd A$:
\begin{align*}
\ovcalT_\theta&:=
\{ M \in \fd A \mid \text{any quotient module $N$ of $M$ satisfies $\theta(M)\ge0$} \},\\
\calF_\theta&:=
\{ M \in \fd A \mid \text{any submodule $L \ne 0$ of $M$ satisfies $\theta(L)<0$} \},\\
\calT_\theta&:=
\{ M \in \fd A \mid \text{any quotient module $N \ne 0$ of $M$ satisfies $\theta(N)>0$} \},\\
\ovcalF_\theta&:=
\{ M \in \fd A \mid \text{any submodule $L$ of $M$ satisfies $\theta(L)\le0$} \}.
\end{align*}
\end{itemize}
\end{Def}

By definition, $\calW_\theta=\ovcalT_\theta \cap \ovcalF_\theta$.
It is easy to see that $\calW_\theta$ is a wide subcategory of $\fd A$,
that is, a full subcategory closed under extensions, kernels and cokernels.
In particular, $\calW_\theta$ is an abelian length category,
so any simple object $S$ in $\calW_\theta$ is a brick.

\begin{Def}
We define the following notions.
\begin{itemize}
\item[(1)]
\cite[Definition 3.2]{BST}
To each $M \in \fd A \setminus \{0\}$,
we associate the \textit{wall} 
\begin{align*}
\Theta_M:=\{ \theta \in K_0(\proj A)_\R \mid M \in \calW_\theta \}.
\end{align*}
In the rest, we consider the \textit{wall-chamber structure} on $K_0(\proj A)_\R$
whose walls are $\Theta_M$.
\item[(2)]
\cite[Definition 2.13]{Asai2}
Let $\theta,\theta' \in K_0(\proj A)_\R$.
Then, we say that $\theta$ and $\theta'$ are \textit{TF equivalent}
if $(\ovcalT_\theta,\calF_\theta)=(\ovcalT_{\theta'},\calF_{\theta'})$ 
and $(\calT_\theta,\ovcalF_\theta)=(\calT_{\theta'},\ovcalF_{\theta'})$.
\end{itemize}
\end{Def}

By Lemma \ref{Lem_DIRRT} and Proposition \ref{Prop_reduction_brick}, 
the following properites hold also for complete special biserial algebras.

\begin{Prop}\cite[Proposition 2.8]{Asai2}
Let $M \in \fd A \setminus \{0\}$. 
Then, there exists a brick $S \in \brick A$ such that $\Theta_S \supset \Theta_M$.
\end{Prop}

\begin{Prop}\label{Prop_TF_brick}\cite[Theorem 2.17]{Asai2}
Let $\theta,\theta' \in K_0(\proj A)_\R$ be distinct elements.
Then, the following conditions are equivalent.
\begin{itemize}
\item[(a)]
The elements $\theta$ and $\theta'$ are TF equivalent.
\item[(b)]
The semistable subcategory $\calW_{\theta''}$ is constant $\theta$ 
in the line segment $[\theta,\theta']$.
\item[(c)]
There exists no finite-dimensional brick 
$S \in \brick A$ such that $[\theta,\theta'] \cap \Theta_S$ is one point.
\end{itemize}
\end{Prop}

Recall that $J \subset I_\rmc=\langle \Cyc(A) \rangle$ and $\overline{A}=A/J$.
It is clear that the quotient map $A \to \overline{A}$ induces
canonical isomorphisms $K_0(\proj A) \cong K_0(\proj \overline{A})$ 
and $K_0(\fd A) \cong K_0(\fd \overline{A})$.
Then, Proposition \ref{Prop_reduction_brick} yields that
the TF equivalence in $K_0(\proj A)$ is the same as that in $K_0(\proj \overline{A})$.

\begin{Lem}
Let $\theta,\theta' \in K_0(\proj A)_\R$.
Then, $\theta$ and $\theta'$ are TF equivalent in $K_0(\proj A)_\R$
if and only if they are TF equivalent in $K_0(\proj \overline{A})_\R$.
\end{Lem}

For any indecomposable $U \in \twopresilt A$,
the element $[U] \in K_0(\proj A)$ is called the \textit{g-vector} of $U$.
Obviously, the bijection $\twopresilt A \to \twopresilt \overline{A}$ 
in Proposition \ref{Prop_reduction_silt} preserves
the g-vectors of indecomposable 2-term presilting complexes in $\sfK^\rmb(\proj A)$;
see also \cite[Lemma 2.7]{VG}.

Let $U=\bigoplus_{i=1}^m U_i \in \twopresilt A$ with $U_i$ indecomposable, 
and consider the \textit{g-vector cones}
\begin{align*}
C(U)&:=\left\{ \sum_{i=1}^m a_i[U_i] \mid a_i \in \R_{\ge 0} \right\}, &
C^+(U)&:=\left\{ \sum_{i=1}^m a_i[U_i] \mid a_i \in \R_{>0} \right\}
\end{align*}
spanned by the g-vectors of the indecomposable direct summands. 
By Proposition \ref{Prop_silt_fund} and 
the following property, the cones are $m$-dimensional.

\begin{Prop}\label{Prop_silt_basis}\cite[Theorem 2.27]{AiI}
Let $T=\bigoplus_{i=1}^n T_i \in \twosilt A$ with $T_i$ indecomposable.
Then, the family $([T_i])_{i=1}^n$ is a $\Z$-basis of $K_0(\proj A)$.
\end{Prop}

Moreover, g-vector cones are compatible 
with common direct summands of 2-term presilting complexes.

\begin{Prop}\label{Prop_max_common}\cite[Corollary 6.7]{DIJ}
Let $U,U' \in \twopresilt A$, and $U''$ be their maximal common direct summands.
Then, we have $C(U) \cap C(U')=C(U'')$.
\end{Prop}

In \cite{Asai2}, we proved that $C^+(U)$ is a TF equivalence class
based on the results of \cite{Yurikusa,BST} for finite-dimensional algebras.
It can be extended by Propositions \ref{Prop_reduction_brick} 
and \ref{Prop_reduction_silt} as follows.

\begin{Prop}\label{Prop_cone_TF}\cite[Proposition 3.11]{Asai2}
Let $U \in \twopresilt A$.
Then, $C^+(U)$ is a TF equivalence class such that
\begin{align*}
C^+(U)&=\{ \theta \in K_0(\proj A)_\R \mid 
(\ovcalT_\theta,\calF_\theta)=(\ovcalT_U,\calF_U),\ 
(\calT_\theta,\ovcalF_\theta)=(\calT_U,\ovcalF_U)\}.
\end{align*}
\end{Prop}

In general, we cannot obtain all TF equivalence classes from g-vector cones,
because the disjoint union $\coprod_{U \in \twopresilt A}C^+(U)$
does not necessarily coincide with $K_0(\proj A)_\R$.
Actually, the following property, 
which was originally proved for finite-dimensional algebras, holds in our setting 
thanks to Proposition \ref{Prop_reduction_silt}.
We say that $A$ is \textit{$\tau$-tilting finite} if $\twosilt A$ is a finite set.

\begin{Prop}\label{Prop_tau_finite_cone}\cite[Theorem 4.7]{Asai2}\cite{ZZ}
The algebra $A$ is $\tau$-tilting finite if and only if 
$\coprod_{U \in \twopresilt A}C^+(U)=K_0(\proj A)_\R$ holds.
\end{Prop}

\subsection{Canonical decompositions}
\label{Subsec_canon_decomp}

As we have seen in the end of the previous subsection,
the cones $C^+(U)$ constructed by 2-term presilting complexes
are not enough to study the wall-chamber structure of $K_0(\proj A)_\R$.
One of the aims of this paper is to study the regions of $K_0(\proj A)_\R$
where the cones $C^+(U)$ do not exist.
For this reason, we first recall the notions of 
presentation spaces and canonical decompositions
established by Derksen-Fei \cite{DF}.

\begin{Def}\cite[Section 1]{DF}
Let $\theta \in K_0(\proj A)$.
\begin{itemize}
\item[(1)]
We define $P_0^\theta,P_1^\theta \in \proj A$ satisfy 
$\theta=[P_0^\theta]-[P_1^\theta]$ and $\add P_0^\theta \cap \add P_1^\theta = \{0\}$.
\item[(2)]
We set $\PHom(\theta):=\Hom_A(P_1^\theta,P_0^\theta)$,
and call it the \textit{presentation space} of $\theta$.
\item[(3)]
For any $f \in \PHom(\theta)$,
we write $P_f$ for the 2-term complex $P_1^\theta \xrightarrow{f} P_0^\theta$
in $\sfK^\rmb(\proj A)$.
\end{itemize}
\end{Def}

These notions were originally defined for finite-dimensional algebras.
In this case, $\PHom(\theta)$ is a finite-dimensional $K$-vector space,
so $\PHom(\theta)$ can be regarded as an irreducible algebraic variety
with Zariski topology.
Let  $(\mathbf{C}_f)$ be a condition for $f \in \PHom(\theta)$, 
then we say that 
$(\mathbf{C}_f)$ \textit{holds for any general element $f \in \PHom(\theta)$}
if the set $\{ f \in \PHom(\theta) \mid (\mathbf{C}_f)\}$ 
is a dense subset of $\PHom(\theta)$.

By using presentation spaces, we can consider direct sums in the Grothendieck group.
In (2), we naturally identify $K_0(\proj A)$ and $K_0(\proj A')$ 
for any quotient algebra $A'$ of $A$ such that $|A'|=|A|$.

\begin{Def}\cite[Definition 4.3]{DF}
Let $\theta,\theta_1,\theta_2,\ldots,\theta_m \in K_0(\proj A)$.
\begin{itemize}
\item[(1)]
Assume that $A$ is a finite-dimensional algebra.
\begin{itemize}
\item[(i)]
We write $\bigoplus_{i=1}^m \theta_i$ in $K_0(\proj A)$
if any general $f \in \PHom(\sum_{i=1}^m \theta_i)$ admits $f_i \in \PHom(\theta_i)$ 
such that $P_f \cong \bigoplus_{i=1}^m P_{f_i}$ in $\sfK^\rmb(\proj A)$.
\item[(ii)]
We say that $\theta$ is \textit{indecomposable} in $K_0(\proj A)$
if $P_f$ is indecomposable in $\sfK^\rmb(\proj A)$ for any general $f \in \PHom(\theta)$.
\end{itemize}
\item[(2)]
Assume that $A$ is a complete special biserial algebra and that
$A'$ is a finite-dimensional quotient algebra $A'$ of $A$ such that $|A'|=|A|$.
\begin{itemize}
\item[(i)]
We write $\bigoplus_{i=1}^m \theta_i$ in $K_0(\proj A)$
if $\bigoplus_{i=1}^m \theta_i$ in $K_0(\proj A')$ for any $A'$ above.
\item[(ii)]
We say that $\theta$ is \textit{indecomposable} in $K_0(\proj A)$
if $\theta$ is indecomposable in $K_0(\proj A')$ for some $A'$ above.
\end{itemize}
\end{itemize}
\end{Def}

We recall the important results in the case of finite-dimensional algebras.

\begin{Prop}\label{Prop_canon_decomp_fd}
Let $A$ be a finite-dimensional algebra.
\begin{itemize}
\item[(1)]
\cite[Theorem 4.4]{DF}
Let $\theta_1,\theta_2,\ldots,\theta_m \in K_0(\proj A)$.
The direct sum $\bigoplus_{i=1}^m \theta_i$ holds if and only if, 
for any $i,j \in \{1,2,\ldots,m\}$ with $i \ne j$,
there exist a pair $(f,g) \in \PHom(\theta_i) \times \PHom(\theta_j)$ such that
\begin{align*}
\Hom_{\sfK^\rmb(\proj A)}(P_f,P_g[1])&=0, &
\Hom_{\sfK^\rmb(\proj A)}(P_g,P_f[1])&=0.
\end{align*}
In particular, $\bigoplus_{i=1}^m \theta_i$ is equivalent to that
$\theta_i \oplus \theta_j$ holds for any $i \ne j$.
\item[(2)]
\cite[Theorem 2.7]{Plamondon}\cite{DF}
Let $\theta \in K_0(\proj A)$.
Then, there exist indecomposable elements 
$\theta_1,\theta_2,\ldots,\theta_m \in K_0(\proj A)$ 
such that $\theta=\bigoplus_{i=1}^m \theta_i$ and 
that $P_{f_i}$ is indecomposable in $\sfK^\rmb(\proj A)$ for 
any $i$ and general $f_i \in \PHom(\theta_i)$.
Moreover, this decomposition is unique up to reordering.
\end{itemize}
\end{Prop}

The unique decomposition $\theta=\bigoplus_{i=1}^m \theta_i$ above
is called the \textit{canonical decomposition} of $\theta$ in $K_0(\proj A)$.

For example, if $A$ is finite-dimensional and 
$U=\bigoplus_{i=1}^m U_i$ is a (not necessarily basic)
2-term presilting complex in $\sfK^\rmb(\proj A)$ with $U_i$ indecomposable, 
then $[U]=\bigoplus_{i=1}^m [U_i]$
is a canonical decomposition.
If $U$ is basic, then $[U_1],[U_2],\ldots,[U_m]$ can be extended 
to a $\Z$-basis of $K_0(\proj A)$ by Proposition \ref{Prop_silt_basis},
so any $\theta \in C^+(T) \cap K_0(\proj A)$ has a canonical decomposition of the form
$\theta=\bigoplus_{i=1}^m [U_i]^{\oplus l_i}$ with $l_i \in \Z_{\ge 1}$.
This observation is valid also if $A$ is a complete special biserial algebra 
by Proposition \ref{Prop_reduction_silt}.

To verify the notion of canonical decompositions in the case of 
complete special biserial algebras, we prepare the following properties.

\begin{Lem}\label{Lem_direct_sum_shift}
Let $J \subset I_\rmc$ be an ideal of $A$ such that 
$\overline{A}=A/J$ is finite-dimensional.
Then, the following assertions hold.
\begin{itemize}
\item[(1)]
For any $\theta_1, \theta_2, \ldots, \theta_m \in K_0(\proj A)$,
the following conditions are equivalent.
\begin{itemize}
\item[(a)]
The direct sum $\bigoplus_{i=1}^m \theta_i$ holds in $K_0(\proj A)$.
\item[(b)]
The direct sum $\bigoplus_{i=1}^m \theta_i$ holds in $K_0(\proj \overline{A})$.
\item[(c)]
For any $i,j \in \{1,2,\ldots,m\}$ with $i \ne j$,
there exist $f \in \PHom(\theta_i)$ and $g \in \PHom(\theta_j)$ such that
$\Hom_{\sfK^\rmb(\proj A)}(P_f,P_g[1])=0$ and 
$\Hom_{\sfK^\rmb(\proj A)}(P_g,P_f[1])=0$.
\end{itemize}
\item[(2)]
Let $\theta \in K_0(\proj A)$.
Then, $\theta$ is indecomposable in $K_0(\proj A)$ if and only if
$\theta$ is indecomposable in $K_0(\proj \overline{A})$.
\end{itemize}
\end{Lem}

\begin{proof}
(1)
$\text{(a)} \Rightarrow \text{(b)}$: 
It is clear.

$\text{(b)} \Rightarrow \text{(c)}$:
Let $i,j \in \{1,2,\ldots,m\}$ with $i \ne j$.
By assumption, we can take 
$f' \in \PHom_{\overline{A}}(\theta_i)$ 
and $g' \in \PHom_{\overline{A}}(\theta_j)$ such that
$\Hom_{\sfK^\rmb(\proj \overline{A})}(P_{f'},P_{g'}[1])=0$ and 
$\Hom_{\sfK^\rmb(\proj \overline{A})}(P_{g'},P_{f'}[1])=0$.
It is easy to see that $? \otimes_A \overline{A}$ induces 
a surjection $\PHom(\theta) \to \PHom_{\overline{A}}(\theta)$,
so there exist $f \in \PHom(\theta_i)$ and $g \in \PHom(\theta_j)$ such that
$\overline{P_f}=P_{f'}$ and $\overline{P_g}=P_{g'}$.
Then, Proposition \ref{Prop_reduction_silt} (1) implies (c).

$\text{(c)} \Rightarrow \text{(a)}$: 
Let $i,j \in \{1,2,\ldots,m\}$ with $i \ne j$ and take $f,g$ in (c).
Let $A'$ be a finite-dimensional quotient algebra of $A$,
and set $P'_f:=P_f \otimes_A A'$ and $P'_g:=P_g[1] \otimes_A A'$.
Then, by the same argument as \cite[Proposition 3.35]{AsI}, we have
\begin{align*}
\Hom_{\sfK^\rmb(\proj A')}(P'_f,P'_g[1])&=0, &
\Hom_{\sfK^\rmb(\proj A')}(P'_g,P'_f[1])&=0.
\end{align*}
Therefore, Proposition \ref{Prop_canon_decomp_fd} (1) implies that
$\bigoplus_{i=1}^m \theta_i$ holds in $K_0(\proj A')$.

(2) follows from (1) and Proposition \ref{Prop_canon_decomp_fd} (2).
\end{proof}

By Lemma \ref{Lem_direct_sum_shift}, 
the canonical decomposition of each element $\theta$ is well-defined
also in the case of complete special biserial algebras.

\begin{Prop}\label{Prop_canon_decomp}
Let $A$ be a complete special biserial algebra.
Then, Proposition \ref{Prop_canon_decomp_fd} holds.
\end{Prop}

In our purpose, it is important to distinguish the elements in $K_0(\proj A)$
corresponding to 2-term presilting complexes from the other elements.
Thus, we use the following notions.

\begin{Def}\cite[Definition 4.6]{DF}
Let $\theta \in K_0(\proj A)$.
\begin{itemize}
\item[(1)]
We say that $\theta$ is \textit{rigid}
if there exists some $f \in \PHom(\theta)$ such that $P_f$ is 2-term presilting.
\item[(2)]
We set $\IR(A)$ (resp.~$\INR(A)$)
as the set of all indecomposable rigid (resp.~non-rigid) elements in $K_0(\proj A)$.
\end{itemize}
\end{Def}

\cite[Theorem 6.5]{DIJ} implies that
if (not necessarily basic) 2-term presilting complexes $U,U'$ in $K^\rmb(\proj A)$ 
satisfy $[U]=[U'] \in K_0(\proj A)$, then $U \cong U' \in \sfK^\rmb(\proj A)$.
In the case that $A$ is finite-dimensional, 
they also showed there that $P_f \cong U$ in $\sfK^\rmb(\proj A)$ holds 
for any general $f \in \PHom([U])$ by using \cite[Lemma 2.16]{Plamondon}.
By this fact and Lemma \ref{Lem_direct_sum_shift},
we can check that direct sums of rigid elements in $K_0(\proj A)$ 
are nothing but the compatibilities of 2-term presilting complexes.

\begin{Lem}\label{Lem_direct_sum_rigid}
Let $U_1,U_2$ be 2-term presilting complexes in $\sfK^\rmb(\proj A)$. 
Then, $[U_1] \oplus [U_2]$ holds in $K_0(\proj A)$
if and only if $U_1 \oplus U_2$ is 2-term presilting.
\end{Lem}

Recently, a very strong result \cite[Theorem 3.8]{PY},
which is a restatement of \cite[Theorem 3.2]{GLFS},
on finite-dimensional representation-tame algebra was found. 
The points of the proof are results in \cite{CB1} on 1-parameter families of modules
over representation-tame algebras.
In our setting, this can be written as follows by Proposition \ref{Prop_reduction_brick}.

\begin{Prop}\label{Prop_E-tame}
Any complete special biserial algebra $A$ is E-tame, that is,
any element $\theta \in K_0(\proj A)$ satisfies $\theta \oplus \theta$.
\end{Prop}

Let $A$ be a finite-dimensional algebra.
For any $M \in \fd A$, $\theta \in K_0(\proj A)$ and $f \in \PHom(\theta)$, 
we can check that
\begin{align*}
\theta(M)
&=\dim_K \Hom_{\sfD^\rmb(\fd A)}(P_f,M)-\dim_K \Hom_{\sfD^\rmb(\fd A)}(P_f,M[1]) 
\nonumber \\
&= \dim_K \Hom_{\sfD^\rmb(\fd A)}(P_f,M)-\dim_K \Hom_{\sfD^\rmb(\fd A)}(M,\nu P_f[-1]) 
\nonumber \\
&= \dim_K \Hom_A(\Coker f,M)-\dim_K \Hom_A(M,\Ker \nu f).
\end{align*}
Thus, if $\Hom_A(M,\Ker \nu f)=0$, then we get $M \in \ovcalT_\theta$.
By Proposition \ref{Prop_E-tame}, 
the converse also holds if $A$ is a finite-dimensional special biserial algebra.

\begin{Prop}\label{Prop_TF_perp}\cite[Lemma 2.13]{Fei}
\textup{(cf.~\cite[Theorem 3.20]{AsI})}
Assume that $A$ is a finite-dimensional special biserial algebra.
Let $\theta \in K_0(\proj A)$ and $M \in \fd A$.
\begin{itemize}
\item[(1)]
The condition $M \in \ovcalT_\theta$ holds if and only if $\Hom_A(M,\Ker \nu f)=0$
for some $f \in \PHom(\theta)$.
In this case, $\Hom_A(M,\Ker \nu f)=0$ holds for any general $f \in \PHom(\theta)$.
\item[(2)]
The condition $M \in \ovcalF_\theta$ holds if and only if $\Hom_A(\Coker f,M)=0$
for some $f \in \PHom(\theta)$.
In this case, $\Hom_A(\Coker f,M)=0$ holds for any general $f \in \PHom(\theta)$.
\end{itemize}
\end{Prop}

\begin{proof}
They directly follow from \cite[Lemma 2.13]{Fei} 
together with Proposition \ref{Prop_E-tame}.
\end{proof}

As an application, Iyama and we obtained the following property.

\begin{Lem}\label{Lem_direct_sum_Coker}\cite[Proposition 3.22]{AsI}
Assume that $A$ is a finite-dimensional special biserial algebra.
Let $\theta_1,\theta_2 \in K_0(\proj A)$.
Then, $\theta_1 \oplus \theta_2$ holds if and only if
$\Coker f \subset \ovcalT_{\theta_2}$ and 
$\Ker \nu f \subset \ovcalF_{\theta_2}$ for some $f \in \PHom(\theta_1)$.
\end{Lem}

The following characterization of direct sums is also useful.
Note that it is valid for all complete special biserial algebras
by Proposition \ref{Prop_reduction_brick}.

\begin{Prop}\label{Prop_direct_sum_TF}\cite[Proposition 4.9, Theorem 3.14]{AsI}
Let $\theta_1,\theta_2 \in K_0(\proj A)$.
Then, $\theta_1 \oplus \theta_2$ holds if and only if
$\calT_{\theta_1} \subset \ovcalT_{\theta_2}$ and 
$\calF_{\theta_1} \subset \ovcalF_{\theta_2}$.
In this case, 
$\calT_{\theta_i} \subset \calT_{\theta_1+\theta_2} \subset
\ovcalT_{\theta_1+\theta_2} \subset \ovcalT_{\theta_i}$ and 
$\calF_{\theta_i} \subset \calF_{\theta_1+\theta_2} \subset
\ovcalF_{\theta_1+\theta_2} \subset \ovcalF_{\theta_i}$ for any $i \in \{1,2\}$.
\end{Prop}

The above proposition implies the following result, which was 
originally proved by Plamondon \cite{Plamondon}.
We say that $\theta=\sum_{i=1}^n a_i[P_i]$ and $\theta'=\sum_{i=1}^n a'_i[P_i]$ are 
\textit{sign-coherent} if $a_i a'_i \ge 0$ holds for all $i \in \{1,2,\ldots,n\}$.

\begin{Prop}\label{Prop_sign_coherent}\cite[Lemma 2.10]{Plamondon}
Let $\theta_1,\theta_2 \in K_0(\proj A)$.
If $\theta_1 \oplus \theta_2$, then $\theta_1$ and $\theta_2$ are sign-coherent.
\end{Prop}

In the rest of this subsection, 
we assume that $A$ is a finite-dimensional special biserial algebra, 
and focus on indecomposable elements in $K_0(\proj A)$.
The argument in the proof of Proposition \ref{Prop_E-tame} in \cite{PY}
gives the next explicit results.

\begin{Prop}\label{Prop_tau_reduced}
The following assertions hold.
\begin{itemize}
\item[(1)]
If $\sigma \in \IR(A)$,
then there uniquely exists $U \in \twopresilt A$
such that $[U]=\sigma$,
and each of $H^0(U)$ and $H^{-1}(\nu U)$ is zero or a string module or a projective-injective module.
Moreover, for any general $f \in \PHom(\theta)$, 
we have $P_f \cong U$ in $\sfK^\rmb(\proj A)$.
\item[(2)]
If $\eta \in \INR(A)$, then
there exists a band $b_\eta$ such that, for any general $f \in \PHom(\eta)$, 
both $\Coker f$ and $\Ker \nu f$ are 
isomorphic to the band module $M(b_\eta,\lambda)$ 
with $\lambda \in K^\times$ depending on $f$.
Moreover, $b_\eta$ is unique up to isomorphisms of bands,
and $M(b_\eta,\lambda)$ is a brick.
\end{itemize}
\end{Prop}

\begin{proof}
(1) is clear by Lemma \ref{Lem_direct_sum_rigid} and the explanation before it.

(2)
For each band $b$ in $A$ and $\lambda \in K^\times$, 
there exists a $K[t]$-$A$-bimodule $X$ such that
$K[t]/(t-\lambda) \otimes_{K[t]} M \cong M(b,\lambda)$
by \cite[Corollary 2.4]{WW}.
By using this, the proof of \cite[Theorem 3.8]{PY} implies the assertion.
\end{proof}

We prepare symbols for the modules appearing above.

\begin{Def}\label{Def_M_sigma_M_eta}
We associate the following modules to indecomposable elements in $K_0(\proj A)$.
\begin{itemize}
\item[(1)]
If $\sigma \in \IR(A)$, then $M_\sigma:=H^0(U)$ and $M'_\sigma:=H^{-1}(\nu U)$.
\item[(2)]
If $\eta \in \INR(A)$, then fix a band $b_\eta$ in Proposition \ref{Prop_tau_reduced} (2),
and set $M_\eta(\lambda):=M(b_\eta,\lambda)$.
\end{itemize}
\end{Def}

By Lemma \ref{Lem_direct_sum_Coker}, we obtain the following properties.

\begin{Lem}\label{Lem_tau_reduced_TF}
Let $\theta \in K_0(\proj A)$.
\begin{itemize}
\item[(1)]
Let $\sigma \in \IR(A)$. 
Then, $\theta \oplus \sigma$ holds if and only if
$M_\sigma \in \ovcalT_\theta$ and $M'_\sigma \in \ovcalF_\theta$.
\item[(2)]
Let $\eta \in \INR(A)$. 
Then, $\theta \oplus \eta$ holds if and only if
$M_\eta(\lambda) \in \calW_\theta$ for all $\lambda \in K^\times$.
\end{itemize}
\end{Lem}

We also use the next property following from Proposition \ref{Prop_tau_reduced} (2)
(cf.~\cite[Remark 4.12]{AsI}).

\begin{Lem}\label{Lem_tau_reduced_simple}
Let $\eta \in \INR(A)$. Then, $M_\eta(\lambda)$ is a simple object in $\calW_\eta$.
\end{Lem}

The following lemma is useful 
when we want to show an element in $K_0(\proj A)$ is in $\INR(A)$.

\begin{Lem}\label{Lem_tame_band_check}
Let $\eta \in K_0(\proj A)$ admit a band $b$ in $A$ such that, 
$M(b,\lambda)$ is a brick and that there exists $f \in \PHom(\eta)$ such
that $\Coker f \cong M(b,\lambda) \cong \Ker \nu f$ for some $\lambda \in K^\times$.
Then, $\eta \in \INR(A)$, and $b$ is isomorphic to $b_\eta$ as bands.
\end{Lem}

\begin{proof}
We can write $b=p_1^{-1}q_1p_2^{-1}q_2 \cdots p_m^{-1}q_m$
with $p_i$ and $q_i$ paths of length $\ge 1$ admitted in $A$.
Define $i_k,j_k \in Q_0$ so that $q_k$ is a path from $i_k$ to $j_k$.

For each $\lambda \in K^\times$,
a miminal projective presentation of $M(b,\lambda)$ is given by
\begin{align*}
f_\lambda:=\begin{bmatrix}
(q_1 \cdot) & 0 & \cdots & 0 & (-p_1 \cdot) \\
(-p_2 \cdot) & (q_2 \cdot) & \cdots & 0 & 0 \\
0 & (-p_3 \cdot) & \ddots & 0 & 0\\
\vdots & \vdots & \ddots & \ddots & \vdots \\
0 & 0 & \cdots &(-p_m \cdot) & (\lambda q_m \cdot) \\
\end{bmatrix}\colon
\bigoplus_{k=1}^m P_{j_k} \to \bigoplus_{k=1}^m P_{i_k}.
\end{align*}
Then, $f_\lambda \in \PHom(\eta)$ by assumption.

Let $\calO_\lambda:=\{ g_0 f_\lambda g_1^{-1} \mid g_i \in \Aut_A(P_i^\eta) \}$ 
for each $\lambda \in K^\times$.
Then, this set and 
$X:=\bigcup_{\lambda \in K^\times} \calO_\lambda
=\{ g_0 f_\lambda g_1^{-1} \mid g_i \in \Aut_A(P_i^\eta), \ \lambda \in K \}$ 
are both constructible by Chevalley's Lemma.

We have
\begin{align*}
\Hom_{\sfK^\rmb(\proj A)}(P_{f_\lambda},P_{f_\lambda}[1]) 
&\cong \Hom_{\sfK^\rmb(\proj A)}(P_{f_\lambda},\nu P_{f_\lambda}[-1]) \\
&\cong \Hom_A(M(b,\lambda),\tau M(b,\lambda))\\ 
&\cong \Hom_A(M(b,\lambda),M(b,\lambda)),
\end{align*}
and this is one-dimensional, since $M(b,\lambda)$ is a brick by assumption.
Thus, by \cite[Lemma 2.16]{Plamondon}, the codimension of $\calO_\lambda$ is one
in $\PHom(\eta)$ for each $\lambda \in K^\times$.

We can regard the tangent spaces 
$T_{f_\lambda}(\calO_\lambda)$ and $T_{f_\lambda}(X)$ at $f_\lambda$
as $K$-vector subspaces of $\PHom(\eta)$ so that
$T_{f_\lambda}(\calO_\lambda) \subset T_{f_\lambda}(X) \subset \PHom(\eta)$.
Since the codimension of $\calO_\lambda$ is one,
that of $T_{f_\lambda}(\calO_\lambda)$ is also one.
Thus, it suffices to show 
$T_{f_\lambda}(\calO_\lambda) \subsetneq T_{f_\lambda}(X)$
to obtain the codimension of $T_{f_\lambda}(X)$ is zero.

Let $\mu \in K^\times \setminus \{ \lambda \}$.
Then, 
\begin{align*}
T_{f_\lambda}(\calO_\lambda)=\{ f_\lambda h_1-h_0 f_\lambda \mid h_i \in \End_A(P_i^\eta) \}
\subset \PHom(\eta)
\end{align*}
by \cite[Lemma 2.15]{Plamondon}.
This set does not have $f_\mu-f_\lambda$ as an element; otherwise, 
the element 
$f':=\left[\begin{smallmatrix} 
f_\lambda & 0 \\ 
f_\mu-f_\lambda & f_\lambda 
\end{smallmatrix}\right]$
satisfies $P_{f'} \cong (P_{f_\lambda})^{\oplus 2}$,
but the cokernel of the left-hand side is 
$M(b,2,\left[\begin{smallmatrix} 
\lambda & 0 \\ 
\mu-\lambda & \lambda 
\end{smallmatrix}\right])$,
which is not isomorphic to the cokernel $M(b,\lambda)^{\oplus 2}$ of the right-hand side.
On the other hand, the tangent space $T_{f_\lambda}(X)$ of $X$ at $f_\lambda$ has
$f_\mu-f_\lambda$ as an element by the construction of $f_\lambda$.
Thus, we get $T_{f_\lambda}(\calO_\lambda) \subsetneq T_{f_\lambda}(X)$.

Therefore, $T_{f_\lambda}(X)$ is of codimension zero in $\PHom(\eta)$, 
and so is $X$.
Thus, the constructible set $X$ is a dense subset of $\PHom(\eta)$, 
so $\eta \in \INR(A)$.
Now, Lemma \ref{Prop_tau_reduced} implies that $b$ must be isomorphic to $b_\eta$.
\end{proof}

Take a $K[t]$-$A$-bimodule $M$ such that
$K[t]/(t-\lambda) \otimes_{K[t]} M \cong M(b,\lambda)$ for each band $b$
as in the proof of Proposition \ref{Prop_tau_reduced}.
Then, the construction of $f_\lambda$ above means that we can take  
a minimal projective resolution 
\begin{align*}
K[t] \otimes_K P_1^\eta \xrightarrow{\widetilde{f}} K[t] \otimes_K P_0^\eta \to M \to 0
\end{align*}
such that the $(k,l)$-entry 
$K[t] \otimes_K P_{j_l} \xrightarrow{\widetilde{f}_{kl}} K[t] \otimes_K P_{i_k}$
of the matrix expression of $\widetilde{f}$
sends $1 \otimes e_{j_l}$ to some element $1 \otimes a_0 + t \otimes a_1$ 
with $a_0,a_1 \in e_{i_k}Ae_{j_l}$.

\section{$\tau$-Tilting reduction}
\label{Sec_tau_reduction}

In this section, we deal with $\tau$-tilting reduction by Jasso \cite{Jasso},
which is a great method to consider the 2-term (pre)silting complexes
which have a fixed direct summand $U \in \twopresilt A$
by using a certain algebra $B$ associated to $U$.

In \cite{Asai2,AsI}, we studied the relationship 
between $\tau$-tilting reduction, TF equivalence classes and canonical decompositions,
and found that this is useful to describe the union of g-vector cones.
We recall some results in those papers in Subsection \ref{Subsec_neighbor}.

Since we use the new algebra $B$ in $\tau$-tilting reduction,
it is important to know which kind of algebra $B$ is. 
Subsection \ref{Subsec_sp_closed_red} is devoted to showing that
$B$ is a finite-dimensional special biserial algebra if so is $A$,
which is crucial in our study in this paper.

\subsection{Neighborhoods associated to 2-term presilting complexes}
\label{Subsec_neighbor}

We start with defining the following subset of $K_0(\proj A)_\R$.

\begin{Def}\cite[Subsection 4.1]{Asai2}
For any $U \in \twopresilt A$, we set 
\begin{align*}
N_U=N_{[U]}:=\{ \theta \in K_0(\proj A)_\R \mid \calT_U \subset \calT_\theta, \ 
\calF_U \subset \calF_\theta \}.
\end{align*}
\end{Def}

Recall that we have defined two semibricks $\calS_U$ and $\calS'_U$
before Lemma \ref{Lem_semibrick} so that
$\calT_U$ is the smallest torsion class containing $\calS_U$ and that 
$\calF_U$ is the smallest torsion class containing $\calS'_U$.
Therefore,
\begin{align*}
N_U=\{ \theta \in K_0(\proj A)_\R \mid \calS_U \subset \calT_\theta, \ 
\calS'_U \subset \calF_\theta \}.
\end{align*}
We prepare some basic properties of $N_U$ from \cite{Asai2,AsI}.

\begin{Lem}\label{Lem_neighbor_basic}
Let $U,U' \in \twopresilt A$. 
\begin{itemize}
\item[(1)]
\cite[Lemma 4.3]{Asai2}
The subset $N_U$ is an open neighborhood of $C^+(U)$.
\item[(2)]
\cite[Lemmas 6.4, 6.5]{AsI}
We have 
\begin{align*}
\overline{N_U}&=\{ \theta \in K_0(\proj A)_\R \mid \calT_U \subset \ovcalT_\theta, \ 
\calF_U \subset \ovcalF_\theta\}, \\
N_U \cap K_0(\proj A)&=
\{ \theta \in K_0(\proj A) \mid \text{$[U]$ is a direct summand of $\theta$} \}, \\
\overline{N_U} \cap K_0(\proj A)&=
\{ \theta \in K_0(\proj A) \mid [U] \oplus \theta \}.
\end{align*}
\item[(3)]
\cite[Lemma 6.6]{AsI}
The complex $U \oplus U'$ is 2-term presilting if and only if
$N_U \cap N_{U'} \ne \emptyset$.
In this case, $N_U \cap N_{U'}=N_V$,
where $V$ is the basic 2-term presilting complex given by $U \oplus U'$.
\end{itemize}
\end{Lem}

\begin{proof}
(1)
This is shown in the same way as \cite[Lemma 4.3]{Asai2};
namely, Proposition \ref{Prop_cone_TF} and 
$N_U=\{ \theta \in K_0(\proj A)_\R \mid \calS_U \subset \calT_\theta, \ 
\calS'_U \subset \calF_\theta \}$ yield the assertion.

(2), (3)
The proofs in \cite{AsI} work also in our setting of complete special biserial algebras.
We remark that Proposition \ref{Prop_direct_sum_TF} is also needed to 
deduce the last equality of (2).
\end{proof}

Therefore, the neighborhood $N_U$ is a nice tool 
when we consider rigid direct summands of a given element $\theta \in K_0(\proj A)$.

By using Bongartz completions in Definition \ref{Def_Bongartz}, 
Jasso \cite{Jasso} gave a very useful method called
\textit{$\tau$-tilting reduction}
to investigate the wide subcategory $\calW_U=\ovcalT_U \cap \ovcalF_U$,
which is equal to $\calW_\theta$ for any $\theta \in C^+(U)$
by Proposition \ref{Prop_cone_TF},
and the subset $\twosilt_U A$ consisting of all $V \in \twosilt A$
which have $U$ as a direct summand.
We define $\twopresilt_U A$ in a similar way.
The point of $\tau$-tilting reduction is the algebra
\begin{align*}
B:=B_U=\End_A(H^0(T))/[H^0(U)],
\end{align*}
where $[H^0(U)]$ is the ideal of $\End_A(H^0(T))$ of all endomorphisms on $H^0(T)$
factoring through some module in $\add H^0(U)$.

For any $M \in \fd A$, recall that
we have a short exact sequence 
$0 \to \mathsf{t}_U M \to M \to \overline{\mathsf{f}}_U M \to 0$
such that $\mathsf{t}_U M \in \calT_U$ and $\overline{\mathsf{f}}_U M \in \ovcalF_U$.

Under this preparation, $\tau$-tilting reductions are given as follows. 

\begin{Prop}\label{Prop_Jasso}
Assume that $A$ is a finite-dimensional algebra.
Let $U \in \twopresilt A$.
\begin{itemize}
\item[(1)]
\cite[Theorem 3.8]{Jasso} \textup{(cf.~\cite[Theorem 4.12]{DIRRT})}
We have an equivalence
\begin{align*}
\Phi:=\Hom_A(H^0(T),?) \colon \calW_U \to \fd B
\end{align*}
of abelian categories.
\item[(2)]
\cite[Theorems 3.16, 4.12]{Jasso}\cite[Proposition 4.2]{Asai2}
There exists a bijection 
\begin{align*}
\red \colon \twopresilt_U A \to \twopresilt B
\end{align*}
such that 
$\Phi(\overline{\mathsf{f}}_U (H^0(V)))=H^0(\red(V))$ for any $V \in \twopresilt_U A$.
\end{itemize}
\end{Prop}

To describe the relationship between 
$[V] \in K_0(\proj A)$ and $[\red(V)] \in K_0(\proj B)$ for each
$V \in \twopresilt_U A$,
we introduced an $\R$-linear projection $\pi \colon K_0(\proj A)_\R \to K_0(\proj B)_\R$
in \cite[Subsection 4.1]{Asai2}.
Let $m:=|U|$, and 
write $T=\bigoplus_{i=1}^n T_i$ with $U=\bigoplus_{i=n-m+1}^n T_i$ 
and $T_i$ indecomposable, and set 
\begin{align*}
X_i:=H^0(T_i)/\sum_{f \in \rad_A(H^0(T),H^0(T_i))} \Im f.
\end{align*}
as in Lemma \ref{Lem_semibrick} for each $i$.
Then, $\{ X_1,X_2,\ldots,X_{n-m} \}$ is the set of simple objects in $\calW_U$
by \cite[Theorem 2.21]{Asai1},
so $S_i^B:=\Phi(X_i)$ are the simple modules in $\fd B$.
We remark that $\{ X_1,X_2,\ldots,X_{n-m} \}$ is contained in the semibrick $\calS_T$.
Moreover, for any $i \in \{1,2,\ldots,n\}$ and $j \in \{1,2,\ldots,n-m\}$,
we get $\dim_K\Hom_A(T_i,X_j)=\delta_{i,j}$ by Lemma \ref{Lem_semibrick}.
Since $X_j \in \calW_U$, we also have $\Hom_A(T_i,X_j[1])=\Hom_A(X_j,\nu T_i[-1])=0$.
Thus, $\langle T_i,X_j \rangle=\delta_{i,j}$ holds.

Define $P_i^B$ as the projective cover of $S_i^B$.
Since $[T_1],[T_2],\ldots,[T_n]$ give a $\Z$-basis of $K_0(\proj A)$
by Proposition \ref{Prop_silt_basis}, 
we can define an $\R$-linear surjection $\pi \colon K_0(\proj A)_\R \to K_0(\proj B)_\R$ by 
\begin{align*}
\pi(\theta):=\sum_{i=1}^m \theta(X_i)[P_i^B].
\end{align*}
Then, since $\langle T_i,X_j \rangle=\delta_{i,j}$, we have
\begin{align*}
\pi([T_i^B])=\begin{cases}
P_i^B & (i \le n-m) \\
0 & (i > n-m)
\end{cases}.
\end{align*}

Under this setting, we have the projection $\pi$ is compatible with
the $\tau$-tilting reduction at $U$.

\begin{Prop}\label{Prop_Jasso_Grothendieck}\cite[Lemma 4.4, Theorem 4.5]{Asai2}
Assume that $A$ is a finite-dimensional algebra. 
Let $U \in \twopresilt A$. 
\begin{itemize}
\item[(1)]
We have $\pi(N_U)=K_0(\proj B)_\R$.
\item[(2)]
For any $\theta \in N_U$, we get $\Phi(\calW_\theta)=\calW_{\pi(\theta)}$.
\item[(3)]
Let $\theta,\theta' \in N_U$.
Then, $\theta$ and $\theta'$ are TF equivalent in $K_0(\proj A)_\R$
if and only if $\pi(\theta)$ and $\pi(\theta')$ are TF equivalent in $K_0(\proj B)_\R$.
\item[(4)]
Let $V \in \twopresilt_U A$.
Then, $\pi([V])=[\red(V)]$ in $K_0(\proj B)_\R$,
and $\pi^{-1}(N_{\red(V)}(B)) \cap N_U=N_V$ in $N_U$.
\end{itemize}
\end{Prop}

\begin{proof}
This follows from \cite[Lemma 4.4, Theorem 4.5]{Asai2} and Proposition \ref{Prop_cone_TF}.
\end{proof}

We remark the following property for later use.

\begin{Lem}\label{Lem_TF_N_U_bij}
Assume that $U,U' \in \twopresilt A$ satisfy $B_U \cong B_{U'}$ as algebras.
Then, there exists an $\R$-linear automorphism 
$\rho \colon K_0(\proj A)_\R \to K_0(\proj A)_\R$
which induces a bijection
\textup{ 
\begin{align*}
\{ \text{TF equivalence classes in $N_U$} \}
&\to \{ \text{TF equivalence classes in $N_{U'}$} \} \\
E &\mapsto \text{(the TF equivalence class containing $\rho(E) \cap N_{U'}$)}.
\end{align*}
}
\end{Lem}

\begin{proof}
Take the Bongartz completions $T,T' \in \twosilt A$ of $U,U' \in \twopresilt A$,
respectively.
Decompose them as $T=\bigoplus_{i=1}^n T_i$ and $T'=\bigoplus_{i=1}^n T'_i$ so that 
\begin{itemize}
\item
$U=\bigoplus_{i=n-m+1}^n T_i$ and $U'=\bigoplus_{i=n-m+1}^n T'_i$;
\item
$\red_U(T_i) \cong \red_{U'}(T'_i)$ as projective $B_U \cong B_{U'}$-modules
for each $i \in \{1,2,\ldots,n-m\}$.
\end{itemize}
Then, we can define an $\R$-linear automorphism 
$\rho \colon K_0(\proj A)_\R \to K_0(\proj A)_\R$ by $[T_i] \mapsto [T'_i]$
for each $i \in \{1,2,\ldots,n\}$.
Now, Proposition \ref{Prop_Jasso_Grothendieck} implies the assertion.
\end{proof}

For each $\theta \in K_0(\proj A)$, 
we can define the maximum rigid direct summand of $\theta$ by Proposition \ref{Prop_canon_decomp}.
An analogue in $K_0(\proj A)_\R$ was considered in \cite{AsI}.

\begin{Def}\cite[Definition 6.9]{AsI}
For any $U \in \twopresilt A$, 
we set
\begin{align*}
R_U=R_U(A):=N_U \setminus
\bigcup_{V \in (\twopresilt_U A) \setminus \{U\} } N_V.
\end{align*}
We call
\begin{align*}
R_0 = K_0(\proj A)_\R \setminus \bigcup_{V \in (\twopresilt A) \setminus \{0\}} N_V
\end{align*}
the \textit{purely non-rigid region} of $K_0(\proj A)_\R$.
\end{Def}

By definition, $K_0(\proj A)_\R$ is the disjoint union of $R_U$.
For each $\theta \in K_0(\proj A)$,
the maximum rigid direct summand of $\theta$ is $[U]$ if and only if $\theta \in R_U$
by Lemma \ref{Lem_neighbor_basic}.
The purely non-rigid region $R_0$ is a closed subset of $K_0(\proj A)_\R$, and
$\theta \in K_0(\proj A)$ has no nonzero rigid direct summand
if and only if $\theta \in R_0$.

We state the following easy property.

\begin{Lem}\label{Lem_R_U_almost_silt}
Let $U \in \twopresilt A$.
\begin{itemize}
\item[(1)]
If $|U|=n$ (that is, $U \in \twosilt A$), then $N_U=R_U=C^+(U)$.
\item[(2)]
If $|U|=n-1$ and $\twosilt_U A=\{T,T'\}$, 
then $N_U=C^+(T) \cup C^+(T') \cup C^+(U)$ and $R_U=C^+(U)$.
\end{itemize}
\end{Lem}

\begin{proof}
Both follow from Proposition \ref{Prop_Jasso_Grothendieck}.
\end{proof}

Since $N_U, R_U$ are unions of TF equivalence classes,
the following property comes from
Propositions \ref{Prop_reduction_brick} and \ref{Prop_TF_brick}.
Recall that $\overline{A}=A/J$ and $\overline{U}=U \otimes_A \overline{A}$ 
for any 2-term complex $U \in \sfK^\rmb(\proj A)$.

\begin{Prop}\label{Prop_R_0_reduction}
Let $A=\widehat{KQ}/I$ be a complete special biserial algebra,
$U \in \twopresilt A$, and $J \subset I_\rmc$ be an ideal of $A$.
Then, we have $N_U(A)=N_{\overline{U}}(\overline{A})$ 
and $R_U(A)=R_{\overline{U}}(\overline{A})$.
In particular, $R_0(A)=R_0(\overline{A})$.
\end{Prop}

We set $\Cone$ as the union of the cones $C(T)$:
\begin{align*}
\Cone:=\bigcup_{T \in \twosilt A}C(T)=\coprod_{U \in \twopresilt A}C^+(U)
\end{align*}
and call its complement $\NR:=K_0(\proj A)_\R \setminus \Cone$ 
the \textit{non-rigid region} in $K_0(\proj A)_\R$.
Proposition \ref{Prop_tau_finite_cone} implies that
$\NR=\emptyset$ if and only if $A$ is $\tau$-tilting finite.
The non-rigid region $\NR$ can be described as follows.

\begin{Prop}\label{Prop_R_U_decompose}\cite[Corollary 6.11]{AsI}
The following assertions hold.
\begin{itemize}
\item[(1)]
For each $U \in \twopresilt A$, we have 
\begin{align*}
R_U=C^+(U)+(\overline{N_U} \cap R_0):=
\{\theta_1+\theta_2 \mid \theta_1 \in C^+(U),\ \theta_2 \in \overline{N_U} \cap R_0\},
\end{align*}
and the choice of such $(\theta_1,\theta_2)$ for each $\theta \in R_U$ is unique.
\item[(2)]
For any $U \in \twopresilt A$, we have
\begin{align*}
R_U \cap \NR = R_U \setminus C^+(U) 
= C^+(U) + ((\overline{N_U} \cap R_0) \setminus \{0\}).
\end{align*}
Therefore,
\begin{align*}
\NR = \coprod_{U \in \twopresilt A} (C^+(U) + ((\overline{N_U} \cap R_0) \setminus \{0\})).
\end{align*}
\end{itemize}
\end{Prop}

\begin{proof}
By Proposition \ref{Prop_reduction_brick}, we may assume that $A$ is finite-dimensional.
Then, \cite[Corollary 6.11]{AsI} works.
\end{proof}

The proposition above says that we know $R_U$ without calculating the algebra $B$
or the linear map $\pi \colon N_U \to K_0(\proj B)_\R$ if $R_0$ is explicitly given.
Thus, we will determine $R_0$ for complete special biserial algebras in the next section.

We also remark that $R_0 \cap \NR=R_0 \setminus \{0\}$,
so any nonzero element in the purely non-rigid region $R_0$
belongs to the non-rigid region $\NR$, but $0 \notin \NR$.

For later use, we state the following property.

\begin{Lem}\label{Lem_bij_to_R_0_B}
Let $U \in \twopresilt A$. 
The linear projection 
$\pi \colon K_0(\proj A)_\R \to K_0(\proj B)_\R$ restricts to a bijection
$\phi \colon \overline{N_U} \cap R_0 \to R_0(B)$.
\end{Lem}

\begin{proof}
This follows from Propositions \ref{Prop_Jasso_Grothendieck} and \ref{Prop_R_U_decompose}.
\end{proof}

We end this subsection by recalling the next property,
which was obtained in our proof of \cite[Theorem 4.7]{Asai2} 
(Proposition \ref{Prop_tau_finite_cone} in this paper).

\begin{Prop}\label{Prop_tau_finite_R_0}\cite[Theorem 4.7]{Asai2}
Let $A$ be a finite-dimensional algebra.
Then, the algebra $A$ is $\tau$-tilting finite
if and only if $R_0=\{0\}$.
\end{Prop}

We remark that it is not known whether 
$R_0 \cap K_0(\proj A)=\{0\}$ implies that $A$ is $\tau$-tilting finite
for general finite-dimensional algebras.
In the case of complete special biserial algebras, 
Schroll-Treffinger-Valdivieso \cite{STV} showed that 
$R_0 \cap K_0(\proj A)=\{0\}$ is equivalent to that $A$ is $\tau$-tilting finite;
see also Corollary \ref{Cor_STV}.

\subsection{$\tau$-Tilting reduction of special biserial algebras}
\label{Subsec_sp_closed_red}

In this subsection, we show the following result, that is, 
the class of finite-dimensional special biserial algebras 
is closed under $\tau$-tilting reduction.

\begin{Thm}\label{Thm_sp_closed_under_Jasso}
Let $A$ be a finite-dimensional special biserial algebra,
and $U \in \twopresilt A$, and $T \in \twosilt A$ be its Bongartz completion.
Then, $B=B_U=\End_A(H^0(T))/[H^0(U)]$ is isomorphic to 
a finite-dimensional special biserial algebra.
\end{Thm}

In proving Theorem \ref{Thm_sp_closed_under_Jasso},
we fully use the notation in the previous subsection.
Set $M:=\Phi^{-1}(B) \in \calW_U$, and 
consider the indecomposable module 
$M_i:=\Phi^{-1}(P_i^B) \in \calW_U$
for each indecomposable $P_i^B \in \proj B$.
Then, $M=\bigoplus_{i=1}^m M_i$, and it is the decomposition to indecomposable modules.
We remark that $M_i=\overline{\mathsf{f}}_U (H^0(T_i))$ for $i \in \{1,2,\ldots,m\}$.
Moreover, $B \cong \End_A(M)$.

We first prepare the following property on the structure of the modules $M_i$ in $\calW_U$.

\begin{Lem}\label{Lem_filt_M_i}
For any $i \in \{1,2,\ldots,m\}$, there exists a sequence
\begin{align*}
0=M_{i,k_0} \subset \cdots \subset M_{i,2} \subset M_{i,1} \subset M_{i,0}=M_i,
\end{align*}
such that, for any $k \in \{0,1,\ldots,k_0-1\}$, 
\begin{itemize}
\item
the module $M_{i,k}$ does not have a band module as a direct summand;
\item
$M_{i,k}/M_{i,k+1}$ is a nonzero direct sum of simple objects in $\calW_U$.
\end{itemize}
\end{Lem}

\begin{proof}
By Proposition \ref{Prop_Jasso} (1),
the short exact sequences in $\calW_U$ are sent to those in $\fd B$.
Therefore, $\Ext_A^1(M_i,M_i)=0$, so $M_i$ is not a band module.

For each $k \in \Z_{\ge 1}$, we have
\begin{align*}
\rad^k P_i^B
=\sum_{g_1,g_2,\ldots,g_k \in \rad(B,B)} \Im (\pi_i^B g_k \cdots g_2 g_1),
\end{align*}
where $\pi_i^B \colon B \to P_i^B$ is the canonical projection.
Set $\pi_i:=\Phi^{-1}(\pi_i^B) \colon M \to M_i$. 
Then, Proposition \ref{Prop_Jasso} (1) implies
\begin{align*}
M_{i,k}:=\Phi^{-1}(\rad^k P_i^B)
=\sum_{f_1,f_2,\ldots,f_k \in \rad(M,M)} \Im (\pi_i f_k \cdots f_2 f_1) \subset M_i.
\end{align*}
By Proposition \ref{Prop_special_biserial_hom},
the standard nonisomorphic endomorphisms on $M$ give a basis of $\rad(M,M)$.
Thus, $M_{i,k}$ cannot have a band module as a direct summand.
Set $k_0$ as the smallest integer satisfying $M_{i,k_0}=0$.
In the sequence
\begin{align*}
0=M_{i,k_0} \subset \cdots \subset M_{i,2} \subset M_{i,1} \subset M_{i,0}=M_i,
\end{align*}
the subfactor $M_{i,k}/M_{i,k+1}$ is sent to 
a nonzero semisimple $B$-module $\rad^k P_i^B/{\rad^{k+1} P_i^B}$ 
for $k \in \{0,1,\ldots,k_0-1\}$,
so $M_{i,k}/M_{i,k+1}$ is a nonzero direct sum of some simple objects in $\calW_U$.
\end{proof}

Under this property, we define a set $\Gamma$ of standard homomorphisms
between indecomposable direct summands of $M$ by
$\Gamma=\coprod_{i=1}^m \Gamma_i$, where $\Gamma_i$ is given as follows.

(a)
If $M_i$ is a string module $M(s)$,
then from Lemma \ref{Lem_filt_M_i}, the string $s$ is of the form
\begin{align*}
s_{-l'} \alpha_{-l'}^{-1} \cdots s_{-2} \alpha_{-2}^{-1} s_{-1} \alpha_{-1}^{-1} \cdot
s_0 \cdot
\alpha_1 s_1 \alpha_2 s_2 \cdots \alpha_l s_l
\end{align*}
such that $l,l' \in \Z_{\ge 0}$, $\alpha_i \in Q_1$ 
and that $M(s_k)$ is a simple object in $\calW_U$.
Note that $M(s_0)=X_i$.

If $l \ge 1$, then take $j \in \{1,2,\ldots,m\}$ such that $M(s_1) \cong X_j$.
It is easy to check that there uniquely exists a standard homomorphism
$\gamma_i \colon M_j \to M_i$ satisfying 
$\Im \gamma_i = M(s_1 \alpha_2 s_2 \cdots \alpha_l s_l)$.
If $l' \ge 1$, then we can similarly take $j' \in \{1,2,\ldots,m\}$ 
so that $M(s_{-1}) \cong X_{j'}$, and define $\gamma'_i \colon M_{j'} \to M_i$
such that 
$\Im \gamma'_i = M(s_{-l'} \alpha_{-l'}^{-1} \cdots s_{-2} \alpha_{-2}^{-1} s_{-1})$.
Thus, we set 
\begin{align*}
\Gamma_i:=\begin{cases}
\{\gamma_i,\gamma'_i\} & (l \ge 1, \ l' \ge 1) \\
\{\gamma_i\}           & (l \ge 1, \ l' = 0) \\
\{\gamma'_i\}          & (l = 0, \ l' \ge 1) \\
\emptyset              & (l = 0, \ l' = 0)
\end{cases}.
\end{align*}

(b)
Otherwise, $M_i$ is a projective-injective module.
In this case, 
take strings $s,t$ such that $M/M_{i,k_0-1} \cong M(s)$ and $M_{i,k_0-1} \cong M(t)$.
Note that $M(t)$ is a simple object in $\calW_U$.
If $M(s)=0$, then $\Gamma_i:=\emptyset$.
Otherwise, $M(s)$ is a string module.
Write 
\begin{align*}
s=
s_{-l'} \alpha_{-l'}^{-1} \cdots s_{-2} \alpha_{-2}^{-1} s_{-1} \alpha_{-1}^{-1} \cdot
s_0 \cdot
\alpha_1 s_1 \alpha_2 s_2 \cdots \alpha_l s_l
\end{align*} 
as above.

If $l \ge 1$, then we take $j$ so that $M(s_1) \cong X_j$ and 
define $\gamma_i \colon M_j \to M_i$ so that
$\Im \gamma_i = M(s_1 \alpha_2 s_2 \cdots \alpha_l s_l \alpha_{l+1} t)$
for some arrow $\alpha_{l+1} \in Q_1$,
where we replace $t$ by $t^{-1}$ if necessary.
If $l' \ge 1$, then we similarly take $j$ so that $M(s_{-1}) \cong X_{j'}$ and
define $\gamma'_i \colon M_{j'} \to M_i$ so that
$\Im \gamma'_i = 
M(t^{-1} \alpha_{-l'-1} s_{-l'} \alpha_{-l'}^{-1} \cdots s_{-2} \alpha_{-2}^{-1} s_{-1})$
for some arrow $\alpha_{-l'-1} \in Q_1$.
If $l=l'=0$, then we take $j$ such that $M(t) \cong X_j$ in $\calW_U$ 
and the unique standard homomorphism 
$\gamma_i'' \colon M_j \to M_i$ such that $\Im \gamma_i'' = M(t)$.
Thus, we set 
\begin{align*}
\Gamma_i:=\begin{cases}
\{\gamma_i,\gamma'_i\} & (l \ge 1, \ l' \ge 1) \\
\{\gamma_i\}           & (l \ge 1, \ l' = 0) \\
\{\gamma'_i\}          & (l = 0, \ l' \ge 1) \\
\{\gamma''_i\}         & (l = 0, \ l' = 0)
\end{cases}.
\end{align*}

Dually, by using the indecomposable injective $B$-modules $I_i^B:=\nu_B P_i^B$,
we define a set $\Delta$ of standard homomorphisms 
between indecomposable direct summands $M'_i:=\Phi^{-1}(I_i^B)$ of $M':=\Phi^{-1}(DB)$.

We define $\widetilde{\Gamma}$ (resp.~$\widetilde{\Delta}$)
as the set of standard homomorphisms between indecomposable direct summands of
$M$ (resp.~$DM$).
Clearly, $\widetilde{\Gamma} \supset \Gamma$ and $\widetilde{\Delta} \supset \Delta$.

\begin{Lem}\label{Lem_arrow_Nakayama}
There exists a bijection $\widetilde{\Gamma} \to \widetilde{\Delta}$
which restricts to a bijection $\Gamma \to \Delta$.
\end{Lem}

\begin{proof}
Let $(\gamma \colon M_j \to M_i) \in \widetilde{\Gamma}$.
By construction, the standard homomorphism $\gamma$ uniquely induces 
a standard homomorphism $(\delta \colon M'_j \to M'_i) \in \widetilde{\Delta}$
which satisfies the commutative diagram
\begin{align*}
\begin{tikzpicture}[baseline=(2.base),->]
\node (2+) at (1.6,1.6) {$\mathstrut M_j$};
\node (3+) at (3.2,1.6) {$\mathstrut M_i$};
\node (4+) at (4.8,1.6) {$\mathstrut M_i$};
\draw (2+) to [edge label=$\scriptstyle \gamma$] (3+);
\draw (3+) to [edge label=$\scriptstyle 1$] (4+);
\node (2) at (1.6,  0) {$\mathstrut X_j$};
\node (3) at (3.2,  0) {$\mathstrut Y$};
\node (4) at (4.8,  0) {$\mathstrut X_i$};
\draw[{Hooks[right]}->] (2) to (3);
\draw[->>] (3) to (4);
\draw[->>] (2+) to (2);
\draw[->>] (3+) to (3);
\draw[->>] (4+) to (4);
\node (2-) at (1.6,-1.6) {$\mathstrut M'_j$};
\node (3-) at (3.2,-1.6) {$\mathstrut M'_j$};
\node (4-) at (4.8,-1.6) {$\mathstrut M'_i$};
\draw (2-) to [edge label=$\scriptstyle 1$] (3-);
\draw (3-) to [edge label=$\scriptstyle \delta$] (4-);
\draw[{Hooks[right]}->] (2) to (2-);
\draw[{Hooks[right]}->] (3) to (3-);
\draw[{Hooks[right]}->] (4) to (4-);
\node (PO) at (2.4,0.8) {P.O.};
\node (PB) at (4.0,-0.8) {P.O.};
\end{tikzpicture}.
\end{align*}
of standard homomorphisms 
(P.O.\ means push outs).
Thus, we obtain a map $\widetilde{\Gamma} \to \widetilde{\Delta}$.
By the same commutative diagram and pullbacks, we can define a map 
$\widetilde{\Delta} \to \widetilde{\Gamma}$.

We can check these maps are mutually inverse,
and restrict to bijections between $\Gamma$ and $\Delta$.
\end{proof}

We remark that the bijection $\widetilde{\Gamma} \to \widetilde{\Delta}$ can be 
extended to an equivalence $\add M \to \add M'$.
By applying $\Phi \colon \calW_U \to \fd B$, we get an equivalence 
$\proj B \to \inj B$, which is isomorphic to the Nakayama functor $\nu_B$.

Now, we can show that $\End_A(M) \cong B$ is a special biserial algebra.

\begin{proof}[Proof of Theorem \ref{Thm_sp_closed_under_Jasso}]
Since each element in $\widetilde{\Gamma}$ is a composite of some elements in $\Gamma$,
the ideal $\rad B$ is generated by $\Phi(\Gamma)$;
here, we regard each $\gamma \in \Gamma$ as an endomorphism on $M$ in an obvious way.
Thus, the quiver $Q_B$ defined by $(Q_B)_0:=\{1,2,\ldots,n-m\}$ and 
$(Q_B)_1:=\{\beta_\gamma \colon i \to j \mid (\gamma \colon M_j \to M_i) \in \Gamma \}$
is the Gabriel quiver of $B$,
and we have the surjective homomorphism $\widehat{KQ_B} \to B$ 
of algebras given by $\beta_\gamma \mapsto \gamma$ for each $\gamma \in \Gamma$.
We check the conditions (a)--(e) in Definition \ref{Def_sp_bi_alg}.

(a) 
For each $i \in (Q_B)_0$, there are at most two elements in $\Gamma$
whose codomains are $M_i$,
which corresponds to the arrows starting at $i$.
Thus, (a) is satisfied.
 
(c)
For any $(\gamma \colon M_j \to M_i) \in \Gamma$ such that $\#\Gamma_j=2$, 
the construction implies that $\gamma \gamma'=0$ for at least one $\gamma' \in \Gamma_j$.
This means (c).

(b) and (d) are similarly shown to (a) and (c)
by using $\Delta$ and Lemma \ref{Lem_arrow_Nakayama}.

Set $J_B \subset \widehat{KQ_B}$ as the ideal generated by 
$\{\beta_\gamma \beta_{\gamma'} \mid \gamma,\gamma' \in \Gamma, \ \gamma\gamma'=0 \}$,
whose elements are paths in $Q_B$ of length 2.

(e) 
Let $i \in (Q_B)_0$.
We write $E_i$ for the set of elements in $\widetilde{\Gamma}$
such that their codomains are $M_i$ and their images are simple objects in $\calW_U$.

We first assume that $M_i$ is isomorphic to a string module.
Then, for each element $\gamma=\gamma_1\gamma_2 \cdots \gamma_l \in E_i$ 
with each $\gamma_k \in \Gamma$,
we define a path $p_\gamma$ in $Q_B$ by 
$p_\gamma:=\beta_{\gamma_1}\beta_{\gamma_2} \cdots \beta_{\gamma_l}\beta_{\gamma_{l+1}}$
if there exists $\gamma_{l+1} \in \Gamma$ such that $\gamma_l\gamma_{l+1}\ne 0$
(such $\gamma_{l+1}$ is unique if exists by (c));
and $p_\gamma=0$ otherwise.

Otherwise, $M_i$ is projective-injective.
In this case, $\#E_i=1$, so take the unique element $\gamma \in E_i$.
We can write
$\gamma=\gamma_1\gamma_2 \cdots \gamma_l
=\gamma'_1\gamma'_2 \cdots \gamma'_{l'}$ with
$\gamma_k,\gamma'_k \in \Gamma$, $\gamma_1 \ne \gamma'_1$.
Set $p_\gamma$ in $\widehat{KQ_B}$ by 
$p_\gamma:=\beta_{\gamma_1}\beta_{\gamma_2} \cdots \beta_{\gamma_l}
-\beta_{\gamma'_1}\beta_{\gamma'_2} \cdots \beta_{\gamma'_{l'}}$.

We can check that the kernel of the surjective homomorphism $\widehat{KQ_B} \to B$ 
of algebras is the ideal generated by $J_B$ and $p_\gamma$ for $\widetilde{\Gamma}_i$'s.
Thus, (e) is satisfied.
\end{proof}

We give a way to calculate the algebra $B$ explicitly in the following example.

\begin{Ex}\label{Ex_Jasso}
Define $A'$ as $A$ in the complete special biserial algebra appearing 
in Example \ref{Ex_rectangle_sp_MP},
and we here set $A:=A'/I_\rmc^{A'}$, 
which is a finite-dimensional special biserial algebra.

(1)
It is straightforward to check that 
$U=U_1 \oplus U_2 \oplus U_3 \oplus U_4 \in \twopresilt A$, where
\begin{align*}
U_1&:=(P_2 \xrightarrow{(\alpha_1 \cdot)} P_1), &
U_2&:=(P_4 \xrightarrow{(\beta_1 \cdot)} P_3), &
U_3&:=(P_6 \xrightarrow{(\gamma_1 \cdot)} P_5), &
U_4&:=(P_8 \xrightarrow{(\delta_4 \cdot)} P_7).
\end{align*}
The simple objects of $\calW_U=\calW_{[U]}$ are 
\begin{align*}
X_1&:= \begin{tikzpicture}[baseline=(0.base),->]
\node (0) at (   0,   0) {$\scriptstyle \mathstrut$};
\node (1) at (   0, 0.3) {$\scriptstyle 1$};
\node (2) at ( 0.6,-0.3) {$\scriptstyle 2$};
\draw (1) to (2);
\end{tikzpicture}, &
X_2&:= \begin{tikzpicture}[baseline=(0.base),->]
\node (0) at (   0,   0) {$\scriptstyle \mathstrut$};
\node (1) at (   0, 0.3) {$\scriptstyle 3$};
\node (2) at ( 0.6,-0.3) {$\scriptstyle 4$};
\draw (1) to (2);
\end{tikzpicture}, &
X_3&:= \begin{tikzpicture}[baseline=(0.base),->]
\node (0) at (   0,   0) {$\scriptstyle \mathstrut$};
\node (1) at (   0, 0.3) {$\scriptstyle 5$};
\node (2) at ( 0.6,-0.3) {$\scriptstyle 6$};
\draw (1) to (2);
\end{tikzpicture}, &
X_4&:= \begin{tikzpicture}[baseline=(0.base),->]
\node (0) at (   0,   0) {$\scriptstyle \mathstrut$};
\node (1) at (   0, 0.3) {$\scriptstyle 7$};
\node (2) at ( 0.6,-0.3) {$\scriptstyle 8$};
\draw (1) to (2);
\end{tikzpicture}.
\end{align*}
These are sent to $S_i^B \in \fd B$ by $\Phi$. 
The module $M_i=\Phi^{-1}(P_i^B)$ is the unique module $M_i$ in $\calW_U$ 
which has $X_i=\Phi^{-1}(S_i^B)$ as a quotient module and $\Ext_A^1(M_i,\calW_U)=0$.
By this characterization, we get
\begin{align*}
M_1&= \begin{tikzpicture}[baseline=(0.base),->]
\node (0) at (   0,   0) {$\scriptstyle \mathstrut$};
\node (1) at (   0, 0.6) {$\scriptstyle 1$};
\node (2) at ( 0.6,   0) {$\scriptstyle 2$};
\node (3) at ( 1.2,-0.6) {$\scriptstyle 4$};
\node (4) at ( 1.8,   0) {$\scriptstyle 3$};
\draw (1) to (2);
\draw (2) to (3);
\draw (4) to (3);
\end{tikzpicture}, &
M_2&= \begin{tikzpicture}[baseline=(0.base),->]
\node (0) at (   0,   0) {$\scriptstyle \mathstrut$};
\node (1) at (   0, 0.9) {$\scriptstyle 3$};
\node (2) at ( 0.6, 0.3) {$\scriptstyle 4$};
\node (3) at ( 1.2,-0.3) {$\scriptstyle 6$};
\node (4) at ( 1.8, 0.3) {$\scriptstyle 5$};
\node (5) at ( 2.4,-0.3) {$\scriptstyle 7$};
\node (6) at ( 3.0,-0.9) {$\scriptstyle 8$};
\draw (1) to (2);
\draw (2) to (3);
\draw (4) to (3);
\draw (4) to (5);
\draw (5) to (6);
\end{tikzpicture}, &
M_3&= \begin{tikzpicture}[baseline=(0.base),->]
\node (0) at (   0,   0) {$\scriptstyle \mathstrut$};
\node (1) at (   0,-0.6) {$\scriptstyle 8$};
\node (2) at ( 0.6,   0) {$\scriptstyle 7$};
\node (3) at ( 1.2, 0.6) {$\scriptstyle 5$};
\node (4) at ( 1.8,   0) {$\scriptstyle 6$};
\node (5) at ( 2.4,-0.6) {$\scriptstyle 8$};
\node (6) at ( 3.0,   0) {$\scriptstyle 7$};
\draw (2) to (1);
\draw (3) to (2);
\draw (3) to (4);
\draw (4) to (5);
\draw (6) to (5);
\end{tikzpicture}, &
M_4&= \begin{tikzpicture}[baseline=(0.base),->]
\node (0) at (   0,   0) {$\scriptstyle \mathstrut$};
\node (1) at (   0, 0.3) {$\scriptstyle 7$};
\node (2) at ( 0.6,-0.3) {$\scriptstyle 8$};
\draw (1) to (2);
\end{tikzpicture}.
\end{align*}
Thus, $B=\End_A(M)$ is isomorphic to $KQ_B/I_B$, where
\begin{align*}
Q_B&= \begin{tikzpicture}[baseline=(1.base),->]
\node (1) at (   0,   0) {1};
\node (2) at ( 1.6,   0) {2};
\node (3) at ( 3.2,   0) {3};
\node (4) at ( 4.8,   0) {4};
\draw (1) to [edge label=$\scriptstyle \alpha$] (2);
\draw (2) to [edge label=$\scriptstyle \beta$] (3);
\draw (3. 30) to [edge label=$\scriptstyle \gamma$] (4.150);
\draw (3.330) to [edge label'=$\scriptstyle \delta$] (4.210);
\end{tikzpicture}, &
I_B&= \langle \alpha\beta, \beta\delta \rangle.
\end{align*}
\end{Ex}

As an application of the proof of Theorem \ref{Thm_sp_closed_under_Jasso},
we can show some kind of ``fractal'' property of the wall-chamber structure on 
$K_0(\proj A)_\R$.

\begin{Cor}\label{Cor_fractal}
Let $A$ be a finite-dimensional special biserial algebra which is not $\tau$-tilting finite,
and $m \in \{1,2,\ldots,n\}$.
\begin{itemize}
\item[(1)]
There exists an infinite subset $\calU$ of $\twopresilt A$ such that
$|U|=|U'|=m$ and $B_U \cong B_{U'}$ as algebras for any $U,U' \in \calU$.
\item[(2)]
Let $\calU$ as in (1).
Then, for any $U,U' \in \calU$, there exists an $\R$-linear automorphism 
$\rho_{U,U'} \colon K_0(\proj A)_\R \to K_0(\proj A)_\R$
which induces a bijection
\textup{ 
\begin{align*}
\{ \text{TF equivalence classes in $N_U$} \}
&\to \{ \text{TF equivalence classes in $N_{U'}$} \} \\
E &\mapsto \text{(the TF equivalence class containing $\rho(E) \cap N_{U'}$)}.
\end{align*}
}
\end{itemize}
\end{Cor}

\begin{proof}
(1)
Take $l \in \Z_{\ge 2}$ so that the admissible ideal $I$ in $A=\widehat{KQ}/I$
satisfies $R^l \subset I$, where $R$ is the arrow ideal of $\widehat{KQ}$.
Then, for any $U \in \twopresilt A$, 
the proof of Theorem \ref{Thm_sp_closed_under_Jasso} implies that
$B_U$ is isomorphic to a finite-dimensional special biserial algebra $\widehat{KQ_B}/I_B$
with $I_B$ an ideal satisfying $(R')^l \subset I_B$
for the arrow ideal $R'$ of $\widehat{KQ_B}$.
Thus, the set $\{B_U \mid U \in \twopresilt A\}$ contains only 
finitely many isoclasses of finite-dimensional special biserial algebras.

On the other hand, there exist infinitely many $U \in \twopresilt A$ with $|U|=m$,
since $A$ is not $\tau$-tilting finite.

Thus, we have the assertion.

(2) follows from Lemma \ref{Lem_TF_N_U_bij}.
\end{proof}

We end this section with
showing that complete gentle algebras are closed under $\tau$-tilting reduction
in the following sense.
This will be crucial in Section \ref{Sec_app}.

\begin{Cor}\label{Cor_gentle_closed_under_Jasso}
Let $A$ be a complete gentle algebra, and $U \in \twopresilt A$.
For each $k \in \Z_{\ge 1}$, we set $A_k:=A/I_\rmc^k$ and 
$B_k:=\End_A(H^0(T_k))/[H^0(U_k)]$,
where $T_k \in \twosilt A_k$ is the Bongartz completion of $U_k:=U \otimes_A A_k$.
Then, there exists a complete gentle algebra $B$ such that
we have epimorphisms $B \to B_k \to B/I_\rmc^B$ of algebras for sufficiently large $k$.
\end{Cor}

\begin{proof}
We have seen that $B_k$ is a string algebra for any $k$ in
Theorem \ref{Thm_sp_closed_under_Jasso}.
Moreover, the simple objects of $\calW_{U_k} \subset \fd A_k$ 
do not depend on $k$ by Proposition \ref{Prop_reduction_brick}.

If $k$ is sufficiently large, then the quiver $Q_{B_k}$ of $B_k$ is constant,
so we call it $Q'$.
We define an admissible ideal $I'$ of the complete algebra $\widehat{KQ'}$ so
that $I'$ is generated by the paths $p$ of length $\ge 2$ in $Q'$ which is not      
admitted in $B_k$ for any $k$.
Then, since $A$ is a complete gentle algebra, the generators of $I'$ are paths of exactly length 2, so $B:=\widehat{KQ'}/I'$ is a complete gentle algebra.

It remains to show the existence of $k_0$ such that 
any path in $Q'$ admitted in $B/I_\rmc^B$ is admitted also in $B_{k_0}$.
We can take $l_0 \in \Z_{\ge 1}$ such that
any path admitted in $B/I_\rmc^B$ is of length $\le l_0$,
and $k \in \Z_{\ge 1}$ such that any path of length 2 admitted in $B$ is admitted in $B_k$.

Let $p'=\beta_1\beta_2\cdots\beta_l$ be a path admitted in $B/I_\rmc^B$
with $\beta_i \in Q'_1$ and $l \le l_0$.
Then, $\beta_i\beta_{i+1}$ is admitted in $B_k$ for each $i \in \{1,2,\ldots,l-1\}$,
and we can consider the modules $M(\beta_i\beta_{i+1}) \in \fd B_k$ and 
$\Phi_k^{-1}(M(\beta_i\beta_{i+1})) 
\in \calW_{U_k} \subset \fd A_k$,
where $\Phi_k \colon \calW_{U_k} \to \fd B_k$ is the equivalence.
In the proof of Theorem \ref{Thm_sp_closed_under_Jasso},
each arrow in $Q_{B_k}=Q'$ corresponds to some standard homomorphism in $\fd A_k$
with $A_k$ a string algebra.
Therefore, there exist strings $s_0,s_1,\ldots,s_l$ admitted in $A_k$
and arrows $\alpha_1,\alpha_2,\ldots,\alpha_l \in Q_1$ such that,
for each $i \in \{1,2,\ldots,l-1\}$,
we have $\Phi_k^{-1}(M(\beta_i\beta_{i+1})) \cong 
M(s_{i-1} \alpha_i s_i \alpha_{i+1} s_{i+1})$ in $\fd A_k$.
By definition of $I_\rmc$, we can check that $s_0 \alpha_1 s_1 \cdots \alpha_l s_l$ is an
admitted string in $A_{kl_0}$.
This means that $\beta_1\beta_2\cdots\beta_l$ is a path admitted in $B_{kl_0}$.
By setting $k_0:=kl_0$, we have proved that 
any path in $Q'$ admitted in $B/I_\rmc^B$ is admitted also in $B_{k_0}$ as desired.
\end{proof}

\section{Descriptions of the non-rigid regions}
\label{Sec_R_0}

Our aim in this section is to describe the non-rigid region $\NR$ 
for complete special biserial algebras $A$.
As we have seen in Proposition \ref{Prop_R_U_decompose}, 
the purely non-rigid region $R_0$ plays an important role,
so we explicitly determine it for all complete special biserial algebras.

In Subsection \ref{Subsec_result}, we give our main results and some examples on $R_0$.
Roughly speaking, our description of $R_0$ 
claims that $R_0$ is expressed by finitely many stability conditions 
for string modules associated to certain paths admitted in $A$.
Subsection \ref{Subsec_key} is devoted to the key arguments,
and we complete the proof of our main results in Subsection \ref{Subsec_proof}.
The argument in these subsections is a modification of 
that for tame hereditary algebras in \cite[Subsection 7.1]{AsI}.
We also describe $\overline{N_U} \cap R_0$ for $U \in \twopresilt A$ 
in terms of $\tau$-rigid and $\tau^{-1}$-rigid modules 
in Subsection \ref{Subsec_complement}. 

\subsection{Results on the purely non-rigid regions}
\label{Subsec_result}

We first give the purely non-rigid region $R_0$ for complete gentle algebras.

For any $c \in \Cyc(A)$, there uniquely exist a path $p$ in $Q$ and $\beta \in Q_1$
such that $c=p\beta$.
We write $p_c$ for this $p$.

\begin{Thm}\label{Thm_R_0_gentle}
Let $A$ be a complete gentle algebra.
For each $\theta \in K_0(\proj A)_\R$, the following conditions are equivalent.
\begin{itemize}
\item[(a)]
The element $\theta$ belongs to $R_0$.
\item[(b)]
For any $c \in \Cyc(A)$, there exists some cyclic permutation
$d \in \Cyc(A)$ of $c$ such that $\theta(M(p_d)) \in \calW_\theta$,
and for any $p \in \ovMP(A)$, we have $M(p) \in \calW_\theta$.
\item[(b$'$)]
For any $c \in \Cyc(A)$, we have $\theta(M(p_c))=0$,
and for any $p \in \ovMP(A)$, we get $M(p) \in \calW_\theta$.
\end{itemize}
In particular, $R_0$ is a rational polyhedral cone of $K_0(\proj A)_\R$,
that is, there exist finitely many elements 
$\theta_1,\theta_2,\ldots,\theta_m \in K_0(\proj A)$ 
such that $R_0=\sum_{i=1}^m \R \theta_i$.
\end{Thm}

Note that, if $c' \in \Cyc(A)$ is a cyclic permutation of $c \in \Cyc(A)$,
then $\theta(M(p_c))=0$ is equivalent to $\theta(M(p_{c'}))=0$,
since these two modules have the same dimension vector.

Set $h:=\sum_{i \in Q_0}[S_i] \in K_0(\fd A)$,
then we obtain the following nice property from the theorem.

\begin{Cor}\label{Cor_R_0_h}
Let $A$ be a complete gentle algebra.
Then, $R_0 \subset \Ker \langle ?,h \rangle$. 
\end{Cor}

\begin{proof}
We define an equivalence relation $\sim$ on $\Cyc(A)$ so that
$c \sim c'$ if and only if $c'$ is a cyclic permutation of $c$, and
take a complete system $X \subset \Cyc(A)$ of representatives with respect to $\sim$.
By Proposition \ref{Prop_gentle_explicit_axiom}, we can check that 
\begin{align*}
2h=\sum_{d \in X}[M(p_d)]+\sum_{p \in \ovMP(A)}a_p[M(p)],
\end{align*}
where $a_p$ is given as follows:
\begin{itemize}
\item 
$a_p=0$ if $p=e_i$ and $i \in Q_0$ has exactly one arrow $\alpha$ ending at $i$ and
exactly one arrow $\beta$ starting at $i$ such that $\alpha \beta=0$ in $A$;
\item
$a_p=2$ if $p=e_i$ with no arrow starting or ending at $i \in Q_0$; and
\item 
$a_p=1$ otherwise.
\end{itemize}
For any $\theta \in R_0$, we have $\theta(M(p))=0$ by Theorem \ref{Thm_R_0_gentle},
so we obtain $\theta(h)=0$.
\end{proof}

We remark that if $\theta \in R_0 \cap K_0(\proj A)$,
then $\theta$ has no nonzero rigid elements,
so $\theta(h)=0$ is clear.
This also implies the corollary above, since $R_0$ is a rational polyhedral cone.

We give some examples. 

\begin{Ex}\label{Ex_rectangle}
Let $Q$ be the quiver
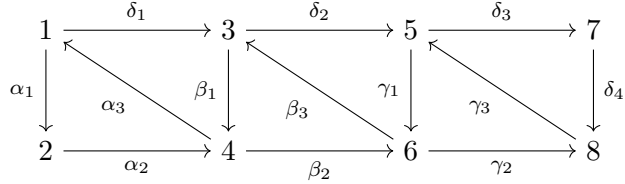
\begin{align*}
\begin{tikzpicture}[scale=0.1, baseline=0pt,->]
\node (1) at ( 0,  8) {$1$};
\node (2) at ( 0, -8) {$2$};
\node (3) at (24,  8) {$3$};
\node (4) at (24, -8) {$4$};
\node (5) at (48,  8) {$5$};
\node (6) at (48, -8) {$6$};
\node (7) at (72,  8) {$7$};
\node (8) at (72, -8) {$8$};
\draw (1) to [edge label'=$\scriptstyle \alpha_1$] (2);
\draw (2) to [edge label'=$\scriptstyle \alpha_2$] (4);
\draw (4) to [edge label =$\scriptstyle \alpha_3$] (1);
\draw (3) to [edge label'=$\scriptstyle \beta_1$] (4);
\draw (4) to [edge label'=$\scriptstyle \beta_2$] (6);
\draw (6) to [edge label =$\scriptstyle \beta_3$] (3);
\draw (5) to [edge label'=$\scriptstyle \gamma_1$] (6);
\draw (6) to [edge label'=$\scriptstyle \gamma_2$] (8);
\draw (8) to [edge label =$\scriptstyle \gamma_3$] (5);
\draw (1) to [edge label =$\scriptstyle \delta_1$] (3);
\draw (3) to [edge label =$\scriptstyle \delta_2$] (5);
\draw (5) to [edge label =$\scriptstyle \delta_3$] (7);
\draw (7) to [edge label =$\scriptstyle \delta_4$] (8);
\end{tikzpicture}
\end{align*}
(which is the same as Example \ref{Ex_rectangle_sp_MP})
and $I \subset KQ$ be the ideal generated by the paths
\begin{align*}
&
\alpha_3\delta_1, 
\delta_1\beta_1,  \beta_3 \delta_2,
\alpha_2\beta_2,  \beta_1 \alpha_3,
\delta_2\gamma_1, \gamma_3\delta_3,
\beta_2 \gamma_2, \gamma_1\beta_3,
\delta_4\gamma_3.
\end{align*}
We can check that $A:=\widehat{KQ}/I$ is a complete gentle algebra.
In this case, 
\begin{align*}
\ovMP(A)&=\{\delta_1\delta_2\delta_3\delta_4,e_2,e_7\}, \\
\Cyc(A)&=\{\alpha_i\alpha_{i+1}\alpha_{i+2}, \beta_i\beta_{i+1}\beta_{i+2}, 
\gamma_i\gamma_{i+1}\gamma_{i+2} \mid i \in \{1,2,3\} \},
\end{align*}
where $\alpha_{i+3}=\alpha_i$, and so on.
Theorem \ref{Thm_R_0_gentle} tells us that $R_0$ consists of 
all $\theta=\sum_{i=1}^8 a_i [P_i]$ satisfying
\begin{itemize}
\item
$a_1+a_3+a_5+a_7+a_8=0$, $a_1+a_3+a_5+a_7 \ge 0$, $a_1+a_3+a_5 \ge 0$, $a_1+a_3 \ge 0$,
$a_1 \ge 0$,
\item
$a_1+a_2+a_4=0$, $a_3+a_4+a_6=0$, $a_5+a_6+a_8=0$,
\item
$a_2=0$, $a_7=0$.
\end{itemize}
Calculating these, we have
\begin{align*}
&
R_0=\{ x\eta_1+y\eta_2 \mid x,y \in \R_{\ge 0} \}, \\
&
\eta_1:=[P_1]-[P_4]-[P_5]+[P_6], \quad \eta_2:=[P_5]-[P_8].
\end{align*}
The element $\eta_1$ is in $\INR(A)$;
indeed if $\eta_1=\theta \oplus \theta'$, 
then since $\eta_1$ has no rigid direct summand by Lemma \ref{Lem_neighbor_basic},
we have $\theta,\theta' \in K_0(\proj A) \cap R_0=
\{ x\eta_1+y\eta_2 \mid x,y \in \Z_{\ge 0} \}$,
so $\theta$ or $\theta'$ must be $0$.
Similarly, $\eta_2 \in \INR(A)$.
The bands $b_{\eta_1}$ and $b_{\eta_2}$ are
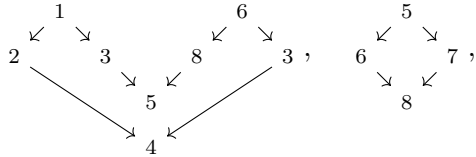
\begin{align*}
\begin{tikzpicture}[scale=0.1, baseline=0pt,->]
\node (1) at ( 0,  0) {$\scriptstyle 2$};
\node (2) at ( 6,  6) {$\scriptstyle 1$};
\node (3) at (12,  0) {$\scriptstyle 3$};
\node (4) at (18, -6) {$\scriptstyle 5$};
\node (5) at (24,  0) {$\scriptstyle 8$};
\node (6) at (30,  6) {$\scriptstyle 6$};
\node (7) at (36,  0) {$\scriptstyle 3$};
\node (8) at (18,-12) {$\scriptstyle 4$};
\draw (2) to (1);
\draw (1) to (8);
\draw (2) to (3);
\draw (3) to (4);
\draw (6) to (5);
\draw (5) to (4);
\draw (6) to (7);
\draw (7) to (8);
\end{tikzpicture}, \quad
\begin{tikzpicture}[scale=0.1, baseline=0pt,->]
\node (1) at ( 0,  0) {$\scriptstyle 6$};
\node (2) at ( 6,  6) {$\scriptstyle 5$};
\node (3) at (12,  0) {$\scriptstyle 7$};
\node (4) at ( 6, -6) {$\scriptstyle 8$};
\draw (2) to (1);
\draw (1) to (4);
\draw (2) to (3);
\draw (3) to (4);
\end{tikzpicture},
\end{align*}
respectively. 
On the other hand, $\eta_1$ and $\eta_2$ are not sign coherent,
so $\eta_1 \oplus \eta_2$ does not hold in $K_0(\proj A)$.
Actually, $\eta_3:=\eta_1+\eta_2=[P_1]-[P_4]+[P_6]-[P_8]$ 
also belongs to $\INR(A)$,
and the band $b_{\eta_3}$ is
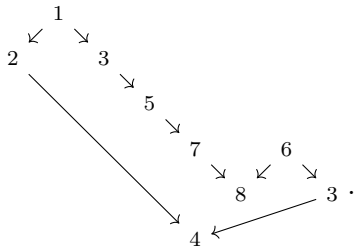
\begin{align*}
\begin{tikzpicture}[scale=0.1, baseline=0pt,->]
\node (1) at ( 0, 18) {$\scriptstyle 2$};
\node (2) at ( 6, 24) {$\scriptstyle 1$};
\node (3) at (12, 18) {$\scriptstyle 3$};
\node (4) at (18, 12) {$\scriptstyle 5$};
\node (5) at (24,  6) {$\scriptstyle 7$};
\node (6) at (30,  0) {$\scriptstyle 8$};
\node (7) at (36,  6) {$\scriptstyle 6$};
\node (8) at (42,  0) {$\scriptstyle 3$};
\node (9) at (24, -6) {$\scriptstyle 4$};
\draw (2) to (1);
\draw (1) to (9);
\draw (2) to (3);
\draw (3) to (4);
\draw (4) to (5);
\draw (5) to (6);
\draw (7) to (6);
\draw (7) to (8);
\draw (8) to (9);
\end{tikzpicture}.
\end{align*}
We can check $\eta_1+\eta_3=\eta_1 \oplus \eta_3$ and 
$\eta_3+\eta_2=\eta_3 \oplus \eta_2$ by Lemma \ref{Lem_tau_reduced_TF}.
Since 
\begin{align*}
R_0=\{ x\eta_1+z\eta_3 \mid x,z \in \R_{\ge 0} \} 
\cup \{ z\eta_3+y\eta_2 \mid z,y \in \R_{\ge 0} \},
\end{align*}
we get that $\INR(A)=\{\eta_1,\eta_2,\eta_3\}$.
\end{Ex}

We obtain the following observation if we apply our result to hereditary algebras of type 
$\widetilde{\mathbb{A}}_{n-1}$.

\begin{Ex}
Let $Q$ be quiver of type $\widetilde{\mathbb{A}}_{n-1}$.

If $Q$ is acyclic, then $A=\widehat{KQ}$ is a finite-dimensional gentle algebra.
Set $(Q_0)_+,(Q_0)_-$ as the set of sources and the set of sinks, respectively.
Then, Theorem \ref{Thm_R_0_gentle} gives
\begin{align*}
R_0=\R_{\ge 0} \left( \sum_{i \in (Q_0)_+}[P_i] - \sum_{i \in (Q_0)_-}[P_i] \right).
\end{align*}
We remark that the modules $M(p)$ for $p \in \ovMP(A)$ are
nothing but the string modules in the mouths of the regular tubes.

For example, let $Q$ be the quiver
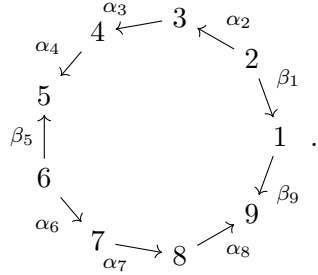
\begin{align*}
\begin{tikzpicture}[baseline=(1.base),->]
\node (1) at (  0:1.6cm) {$1$};
\node (2) at ( 40:1.6cm) {$2$};
\node (3) at ( 80:1.6cm) {$3$};
\node (4) at (120:1.6cm) {$4$};
\node (5) at (160:1.6cm) {$5$};
\node (6) at (200:1.6cm) {$6$};
\node (7) at (240:1.6cm) {$7$};
\node (8) at (280:1.6cm) {$8$};
\node (9) at (320:1.6cm) {$9$};
\draw (2) to [edge label=$\scriptstyle \beta_1$] (1);
\draw (2) to [edge label'=$\scriptstyle \alpha_2$] (3);
\draw (3) to [edge label'=$\scriptstyle \alpha_3$] (4);
\draw (4) to [edge label'=$\scriptstyle \alpha_4$] (5);
\draw (6) to [edge label=$\scriptstyle \beta_5$] (5);
\draw (6) to [edge label'=$\scriptstyle \alpha_6$] (7);
\draw (7) to [edge label'=$\scriptstyle \alpha_7$] (8);
\draw (8) to [edge label'=$\scriptstyle \alpha_8$] (9);
\draw (1) to [edge label=$\scriptstyle \beta_9$] (9);
\end{tikzpicture}.
\end{align*}
Then, $\ovMP(A)$ consists of the paths
\begin{align*}
e_1, \alpha_2\alpha_3\alpha_4, \alpha_6\alpha_7\alpha_8; \ 
\beta_1\beta_9, e_8, e_7, \beta_5, e_4, e_3.
\end{align*}
Thus, $R_0=\R_{\ge 0}([P_2]+[P_6]-[P_5]-[P_9])$ by Theorem \ref{Thm_R_0_gentle}.
The band associated to the unique element $[P_2]+[P_6]-[P_5]-[P_9]$ in $\INR(A)$ is
the quiver $Q$ itself;
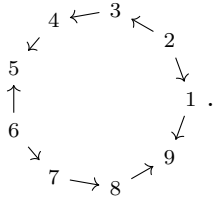
\begin{align*}
\begin{tikzpicture}[baseline=(1.base),->]
\node (1) at (  0:1.2cm) {$\scriptstyle 1$};
\node (2) at ( 40:1.2cm) {$\scriptstyle 2$};
\node (3) at ( 80:1.2cm) {$\scriptstyle 3$};
\node (4) at (120:1.2cm) {$\scriptstyle 4$};
\node (5) at (160:1.2cm) {$\scriptstyle 5$};
\node (6) at (200:1.2cm) {$\scriptstyle 6$};
\node (7) at (240:1.2cm) {$\scriptstyle 7$};
\node (8) at (280:1.2cm) {$\scriptstyle 8$};
\node (9) at (320:1.2cm) {$\scriptstyle 9$};
\draw (2) to (1);
\draw (2) to (3);
\draw (3) to (4);
\draw (4) to (5);
\draw (6) to (5);
\draw (6) to (7);
\draw (7) to (8);
\draw (8) to (9);
\draw (1) to (9);
\end{tikzpicture}.
\end{align*}

On the other hand, if $Q$ is cyclic, then $A=\widehat{KQ}$ is a complete gentle algebra.
In this case, any $e_i$ belongs to $\ovMP(A)$, 
so Theorem \ref{Thm_R_0_gentle} implies $R_0=\{0\}$.
Thus, any element in $K_0(\proj A)$ is rigid.

This can be explained also in the following way.
Clearly, the algebra $A/I_\rmc$ is a Nakayama algebra, 
so $A/I_\rmc$ is representation-finite and $\tau$-tilting finite.
Therefore, $R_0=R_0(A/I_\rmc)=\{0\}$ by Proposition \ref{Prop_R_0_reduction}.
\end{Ex}

The following complete algebras from \cite[Subsections 7.2, 7.3]{AsI} suggest
that $R_0$ itself is not enough to determine $\INR(A)$.

\begin{Ex}\label{Ex_triangle}
We consider the following two gentle algebras.

(1)
Let $Q$ be the quiver
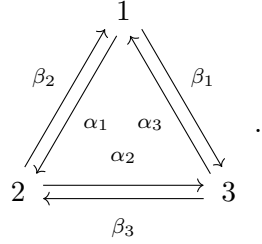
\begin{align*}
\begin{tikzpicture}[every node/.style={circle},baseline=0pt,->]
\node (1) at ( 90:1.6cm) {$1$};
\node (2) at (210:1.6cm) {$2$};
\node (3) at (330:1.6cm) {$3$};
\draw (1.255) to [edge label=$\scriptstyle \alpha_1$] (2. 45);
\draw (2. 15) to [edge label=$\scriptstyle \alpha_2$] (3.165);
\draw (3.135) to [edge label=$\scriptstyle \alpha_3$] (1.285);
\draw (1.315) to [edge label=$\scriptstyle \beta_1$] (3.105);
\draw (3.195) to [edge label=$\scriptstyle \beta_3$] (2.345);
\draw (2. 75) to [edge label=$\scriptstyle \beta_2$] (1.225);
\end{tikzpicture}.
\end{align*}
and $I \subset \widehat{KQ}$ be the ideal generated by the paths
\begin{align*}
\alpha_1\beta_2,\alpha_2\beta_3,\alpha_3\beta_1,
\beta_1\alpha_3,\beta_2\alpha_1,\beta_3\alpha_2
\end{align*}
Then, $A:=\widehat{KQ}/I$ is a complete gentle algebra.
In this case, 
\begin{align*}
\ovMP(A)&=\emptyset, \\
\Cyc(A)&=\{\alpha_i\alpha_{i+1}\alpha_{i+2},\beta_{i+2}\beta_{i+1}\beta_i\}.
\end{align*}
Therefore, $R_0=\Ker \langle ?,h \rangle$.
By considering the quotient algebra $A/I_\rmc$,
we have 
\begin{align*}
\INR(A)=\{\theta_{i,j} \mid i,j \in \{1,2,3\}, \ i \ne j\}
\end{align*}
from \cite[Theorem 7.6]{AsI}.

(2)
Let $Q$ be the quiver
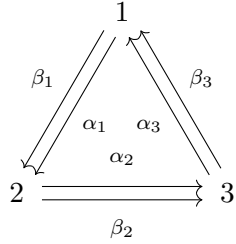
\begin{align*}
\begin{tikzpicture}[every node/.style={circle},baseline=0pt,->]
\node (1) at ( 90:1.6cm) {$1$};
\node (2) at (210:1.6cm) {$2$};
\node (3) at (330:1.6cm) {$3$};
\draw (1.255) to [edge label=$\scriptstyle \alpha_1$] (2. 45);
\draw (2. 15) to [edge label=$\scriptstyle \alpha_2$] (3.165);
\draw (3.135) to [edge label=$\scriptstyle \alpha_3$] (1.285);
\draw (1.225) to [edge label'=$\scriptstyle \beta_1$] (2. 75);
\draw (2.345) to [edge label'=$\scriptstyle \beta_2$] (3.195);
\draw (3.105) to [edge label'=$\scriptstyle \beta_3$] (1.315);
\end{tikzpicture}
\end{align*}
and $I \subset \widehat{KQ}$ be the ideal generated by the paths
\begin{align*}
\alpha_1\beta_2,\alpha_2\beta_3,\alpha_3\beta_1,
\beta_1\alpha_2,\beta_2\alpha_3,\beta_3\alpha_1.
\end{align*}
Then, $A:=\widehat{KQ}/I$ is a complete gentle algebra.
In this case, 
\begin{align*}
\ovMP(A)&=\emptyset, \\
\Cyc(A)&=\{\alpha_i\alpha_{i+1}\alpha_{i+2},\beta_i\beta_{i+1}\beta_{i+2}\}.
\end{align*}
We have $R_0=\Ker \langle ?,h \rangle$.
This is the same as in (1).
However, $\INR(A)$ has infinitely many elements by \cite[Theorem 7.9]{AsI}:
\begin{align*}
\INR(A)=\{ a_1 [P_1] + a_2 [P_2] + a_3 [P_3] \mid a_1,a_2,a_3 \in \Z, \
a_1+a_2+a_3=0, \ \gcd(a_1,a_2,a_3)=1 \}.
\end{align*}
\end{Ex}

Next, we consider $R_0$ for any complete special biserial algebra $A$.
We fix an ideal $\widetilde{I} \subset I$ such that 
$\widetilde{A}$ is a complete gentle algebra by Proposition \ref{Prop_sp_quot_gentle}.

For each $\theta \in K_0(\proj A)_\R$, we set 
$\widetilde{\calW}_\theta$ 
($\subset \calW_\theta^{\widetilde{A}} \subset \fd \widetilde{A}$) as the subcategory of
all $M \in \fd \widetilde{A}$ such that
$M$ admit flitrations $0=M_0 \subset M_1 \subset \cdots \subset M_l=M$
with $M_i/M_{i-1} \cong \calW_\theta \cap \fd A$.

For our purpose, we do not need to fully understand the category $\widetilde{\calW}_\theta$;
we are interested only in 
whether the string module $M(p)$ associted to a path $p$ 
belongs to $\widetilde{\calW}_\theta$ or not.
In this context, the following properties are useful.

\begin{Lem}\label{Lem_path_tilde_W}
Let $p$ be a path in $Q$ with $p \ne 0$ in $A$.
Then, $M(p) \in \widetilde{\calW}_\theta$ if and only if
there exist paths $q_1,q_2,\ldots,q_k$ in $Q$ and
arrows $\alpha_1,\alpha_2,\ldots,\alpha_{k-1} \in Q_1$
such that 
\begin{itemize}
\item
$p = q_1 \alpha_1 q_2 \alpha_2 \cdots q_{l-1} \alpha_{l-1} q_l$ as paths; 
\item
$M(q_i)$ is a simple object in $\calW_\theta \cap \fd A$;
\end{itemize}
In particular,
$\{\theta \in K_0(\proj A)_\R \mid M(p) \in \widetilde{\calW}_\theta\}$
is a union of finitely many rational polyhedral cones.
\end{Lem}

\begin{proof}
This directly follows from the definition.
\end{proof}

\begin{Lem}\label{Lem_path_W_MP_not_T}
Let $\widetilde{p}$ be a path which is a string in $\widetilde{A}$.
If $M(\widetilde{p}) \in \calW_\theta^{\widetilde{A}}$ and
any $p \in \ovMP_*(A)$ satisfies $M(p) \notin \calT_\theta$,  
then $M(\widetilde{p}) \in \widetilde{\calW}_\theta$.
\end{Lem}

\begin{proof}
If there exists $p \in \ovMP_*(A)$ 
such that $M(\widetilde{p})$ is a quotient module of $M(p)$,
then it is done. 

Otherwise, 
we can take subpaths $p_1,\widetilde{s}_1$ of length $\ge 1$ such that 
$p_1 \in \MP_*(A)$ and $\widetilde{p}=p_1\widetilde{s}_1$.
By assumption, 
there exist some subpaths $q_1,r_1$ of $p_1$ such that 
$M(q_1) \in \ovcalF_\theta \cap \fd A$ and that $p_1=q_1r_1$.
Since $M(q_1)$ is a quotient module of $M(p) \in \calW_\theta$,
we have $M(q_1) \in \calW_\theta \cap \fd A$.
We define $\alpha_1 \in Q_1$ and a path $\widetilde{p}_1$ in $Q$ 
so that $r_1\widetilde{s}_1=\alpha_1\widetilde{p}_1$.
Then, we get $M(\widetilde{p}_1)=M(\widetilde{p})/M(p_1) \in \calW_\theta$.
Applying the argument above to $\widetilde{p}_1$,
we have some paths $q_2,r_2,\widetilde{s}_2$ such that 
$M(q_2) \in \calW_\theta \cap \fd A$ and $\widetilde{p}_1=q_2r_2\widetilde{s}_2$.
We set $\alpha_2 \in Q_1$ and a path $\widetilde{p}_2$ in $Q$ 
so that $r_2\widetilde{s}_2=\alpha_2\widetilde{p}_2$.

Repeating this procedure finitely many times,
we finally get paths $q_1,q_2,\ldots,q_k$ in $Q$ and
arrows $\alpha_1,\alpha_2,\ldots,\alpha_{k-1} \in Q_1$ satisfying the conditions of 
Lemma \ref{Lem_path_tilde_W}.
Thus, $M(\widetilde{p}) \in \widetilde{\calW}_\theta$.
\end{proof}

We aim to prove the following result.
For any $\widetilde{c} \in \Cyc(\widetilde{A})$, 
there uniquely exist a path $\widetilde{p}$ in $Q$ and $\beta \in Q_1$
such that $\widetilde{c}=\widetilde{p}\beta$.
We write $\widetilde{p}_{\widetilde{c}}$ for this $\widetilde{p}$ 
as in Theorem \ref{Thm_R_0_gentle}.

\begin{Thm}\label{Thm_R_0_sp}
Let $A=\widehat{KQ}/I$ be a complete special biserial algebra,
and take an ideal $\widetilde{I} \subset I$ such that 
$\widetilde{A}=\widehat{KQ}/\widetilde{I}$ is a complete gentle algebra.
Then, for each $\theta \in K_0(\proj A)_\R$, the following conditions are equivalent.
\begin{itemize}
\item[(a)]
The element $\theta$ belongs to $R_0$.
\item[(b)]
The element $\theta$ belongs to $R_0(\widetilde{A})$, and 
if $p \in \ovMP_*(A)$, then $M(p) \notin \calT_\theta$.
\item[(b$'$)]
The element $\theta$ belongs to $R_0(\widetilde{A})$, and 
if $p \in \ovMP^*(A)$, then $M(p) \notin \calF_\theta$.
\item[(c)]
For any $\widetilde{c} \in \Cyc(\widetilde{A})$,
there exists a cyclic permutation $\widetilde{d}$ of $\widetilde{c}$ such that
$M(\widetilde{p}_{\widetilde{d}}) \in \widetilde{\calW}_\theta$,
and for any $\widetilde{p} \in \ovMP(\widetilde{A})$, 
we have $M(\widetilde{p}) \in \widetilde{\calW}_\theta$.
\end{itemize}
In particular, 
$R_0$ is a union of finitely many rational polyhedral cones in $K_0(\proj A)_\R$.
Moreover, $R_0 \subset \Ker \langle ?,h \rangle$.
\end{Thm}

We remark that $R_0$ is not necessarily a rational polyhedral cone;
see Example \ref{Ex_rectangle_sp}.
Note also that 
this theorem implies that projective-injective modules do not matter on $R_0$;
more precisely, if two different paths $p \ne q$ in $Q$ satisfy
$p-q \in I$ and $p,q \notin I$, then $R_0=R_0(A/\langle p-q \rangle)$.

We write the following property for later use.

\begin{Cor}\label{Cor_2h_simple_obj}
Let $A=\widehat{KQ}/I$ be a complete special biserial algebra.
Then, there exist a set $Z$ of paths admitted in $A$ 
and $k_p \in \{0,1,2\}$ for all $p \in Z$ satisfying the following conditions:
\begin{itemize} 
\item
$M(p)$ is a simple object in $\calW_\theta$ for each $p \in Z$;
\item
$\sum_{p \in Z}k_p[M(p)]=2h \in K_0(\fd A)$;
\item
$k_p=1$ if $p$ is of length $\ge 1$.
\end{itemize}
\end{Cor}

\begin{proof}
Take the complete gentle algebra $\widetilde{A}$ as in Theorem \ref{Thm_R_0_sp}.
By Theorem \ref{Thm_R_0_sp},
we can take a complete system $X \subset \Cyc(\widetilde{A})$ 
of representatives with respect to cyclic permutations so that 
$M(\widetilde{p}_{\widetilde{d}}) \in \widetilde{\calW}_\theta$ 
for any $\widetilde{d} \in X$.

For each $\widetilde{d} \in X$,
we take the paths $q_i$ for $i \in \{1,2,\ldots,l\}$ admitted in $A$ 
applying Lemma \ref{Lem_path_tilde_W} to $M(\widetilde{p}_{\widetilde{d}})$,
and rewrite these paths as $q(\widetilde{d},i)$ 
for $i \in \{1,2,\ldots,l(\widetilde{d})\}$.
In a similar way, 
we apply Lemma \ref{Lem_path_tilde_W} to $M(\widetilde{p})$
for each $\widetilde{p} \in \ovMP(\widetilde{A})$,
and define $q(\widetilde{p},i)$ ($i \in \{1,2,\ldots,l(\widetilde{p})\}$).

We set $a_{\widetilde{p}}$ as in the proof of Corollary \ref{Cor_R_0_h}, then
\begin{align*}
2h=\sum_{\widetilde{d} \in X}[M(\widetilde{p}_{\widetilde{d}})]+\sum_{\widetilde{p} \in \ovMP(\widetilde{A})}a_{\widetilde{p}}[M(\widetilde{p})].
\end{align*}
For each $\widetilde{d} \in X$, we define $a_{\widetilde{d}}=1$.
Then, we get 
\begin{align*}
2h=\sum_{\widetilde{p} \in \ovMP(\widetilde{A}) \cup X}
\sum_{i=1}^{l(\widetilde{p})}a_{\widetilde{p}}[M(q(\widetilde{p},i))].
\end{align*}

We set $Y$ and $Z$ as the set of all 
pairs $(\widetilde{p},i)$ and paths $q(\widetilde{p},i)$ appearing above
for $\widetilde{p} \in \ovMP(\widetilde{A}) \cup X$, respectively.
Namely,
\begin{align*}
Y&:=\{(\widetilde{p},i) \mid \widetilde{p} \in \ovMP(\widetilde{A}) \cup X, \
i \in \{1,2,\ldots,l(\widetilde{p})\}, \\
Z&:=\{q(\widetilde{p},i) \mid (\widetilde{p},i) \in Y \}.
\end{align*}
Then, for any $p \in Z$, the path $p$ is admitted in $A$,
and $M(p)$ is a simple object in $\calW_\theta$ for any $p \in Z$.
Note that the map $Y \ni (\widetilde{p},i) \mapsto q(\widetilde{p},i) \in Z$
is not necessarily injective.
Thus, for each $p \in Z$, we define 
$Y_p:=\{ (\widetilde{p},i) \in Y \mid p=q(\widetilde{p},i) \}$ and 
$k_p:=\sum_{(\widetilde{p},i) \in Y_p} a_{\widetilde{p}}$.
Then, $\sum_{p \in Z}k_p[M(p)]=2h \in K_0(\fd A)$.

To finish the proof, let $p \in Z$.
By definition, we have
\begin{itemize}
\item
if $\# Y_p \ge 2$, then $p$ is of length 0, $\# Y_p = 2$, and 
$a_{\widetilde{p}}=1$ for any $(\widetilde{p},i) \in Y_p$;
\item
if there exists $(\widetilde{p},i) \in Y_p$ such that $a_{\widetilde{p}}=2$,
then $\# Y_p = 1$;
\item
if $p$ is of length $\ge 1$, then $\# Y_p=1$, and $a_{\widetilde{p}}=1$
for the unique element $(\widetilde{p},i) \in Y_p$.
\end{itemize}
These allow us to check that $k_p \in \{0,1,2\}$ for any $p \in Z$,
and that $k_p=1$ if $p$ is of length $\ge 1$.
\end{proof}

We give some examples.

\begin{Ex}\label{Ex_rectangle_sp}
We consider the algebra $A$ in Example \ref{Ex_rectangle_sp_MP}.
Define $\widetilde{A}$ as the complete gentle algebra in Example \ref{Ex_rectangle}.
Then, $A$ is the quotient algebra of $\widetilde{A}$
by the additional relation $\delta_1\delta_2\delta_3\delta_4=0$.

We have seen 
\begin{align*}
& R_0(\widetilde{A})=\{ x\eta_1+y\eta_2 \mid x,y \in \R_{\ge 0} \}, \\
& \eta_1=[P_1]-[P_4]-[P_5]+[P_6], \quad \eta_2=[P_5]-[P_8],
\end{align*}
and Theorem \ref{Thm_R_0_sp} tells us that 
$R_0(A)=\{\theta \in R_0(\widetilde{A}) \mid
M(\delta_1\delta_2\delta_3\delta_4) \in \widetilde{\calW}_\theta \}$.
The condition 
$M(\delta_1\delta_2\delta_3\delta_4) \in \widetilde{\calW}_\theta$ implies that
$M(e_8)$, $M(\delta_4)$, $M(\delta_3\delta_4)$ or $M(\delta_2\delta_3\delta_4)$
is $\theta$-semistable.
If $\theta \in R_0(\widetilde{A})$ satisfies this condition, 
then $\theta \in \R_{\ge 0}\eta_1$ or $\theta \in \R_{\ge 0}\eta_2$.
Conversely, if $\theta \in \R_{\ge 0}\eta_1 \cup \R_{\ge 0}\eta_2$,
then we can check $M(\delta_1\delta_2\delta_3\delta_4) \in \widetilde{\calW}_\theta$.
Therefore,
$R_0=\R_{\ge 0}\eta_1 \cup \R_{\ge 0}\eta_2$.
In particular, $R_0$ itself is not a rational polyhedral cone.

The band $\widetilde{A}$-modules corresponding to $\eta_1$ and $\eta_2$ are $A$-modules,
so $\eta_1$ and $\eta_2$ are in $\INR(A)$.
On the other hand, $\sigma=\eta_1+\eta_2 \notin R_0$,
so $\sigma \notin \INR(A)$.
Actually, $\sigma \in \IR(A)$,
and the corresponding $\tau$-rigid $(A/I_\rmc)$-module
and $\tau^{-1}$-rigid $(A/I_\rmc)$-module are
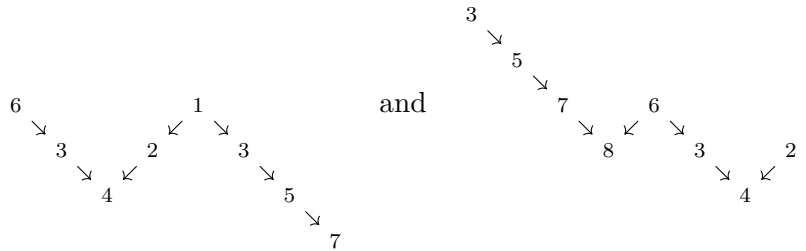
\begin{align*}
\begin{tikzpicture}[scale=0.1, baseline=(5.base),->]
\node (1) at (-18, 12) {$\scriptstyle 6$};
\node (2) at (-12,  6) {$\scriptstyle 3$};
\node (3) at ( -6,  0) {$\scriptstyle 4$};
\node (4) at (  0,  6) {$\scriptstyle 2$};
\node (5) at (  6, 12) {$\scriptstyle 1$};
\node (6) at ( 12,  6) {$\scriptstyle 3$};
\node (7) at ( 18,  0) {$\scriptstyle 5$};
\node (8) at ( 24, -6) {$\scriptstyle 7$};
\draw (1) to (2);
\draw (2) to (3);
\draw (4) to (3);
\draw (5) to (4);
\draw (5) to (6);
\draw (6) to (7);
\draw (7) to (8);
\end{tikzpicture} \quad \text{and} \quad
\begin{tikzpicture}[scale=0.1, baseline=(5.base),->]
\node (1) at (-42, 24) {$\scriptstyle 3$};
\node (2) at (-36, 18) {$\scriptstyle 5$};
\node (3) at (-30, 12) {$\scriptstyle 7$};
\node (4) at (-24,  6) {$\scriptstyle 8$};
\node (5) at (-18, 12) {$\scriptstyle 6$};
\node (6) at (-12,  6) {$\scriptstyle 3$};
\node (7) at ( -6,  0) {$\scriptstyle 4$};
\node (8) at (  0,  6) {$\scriptstyle 2$};
\draw (1) to (2);
\draw (2) to (3);
\draw (3) to (4);
\draw (5) to (4);
\draw (5) to (6);
\draw (6) to (7);
\draw (8) to (7);
\end{tikzpicture}.
\end{align*}
\end{Ex}

\begin{Ex}\label{Ex_triangle_sp}
We consider the quiver $Q$ in Example \ref{Ex_triangle} (1),
and set the admissible ideal $\widetilde{I}$ of $\widehat{KQ}$ 
as $I$ in Example \ref{Ex_triangle} (1).
We also define 
$I:=\widetilde{I}+\langle \alpha_i\alpha_{i+1} \rangle \supset \widetilde{I}$.
For the complete gentle algebra $\widetilde{A}:=\widehat{KQ}/\widetilde{I}$,
we have $\MP(\widetilde{A})=\emptyset$ and 
$\Cyc(\widetilde{A})=\langle \alpha_i\alpha_{i+1}\alpha_{i+2}, 
\beta_i\beta_{i+1}\beta_{i+2} \mid i \in \{1,2,3\} \rangle$.

Thus, by Theorem \ref{Thm_R_0_sp}, $\theta \in R_0$ holds if and only if
$M(\alpha_i\alpha_{i+1}) \in \widetilde{\calW}_\theta \subset \fd \widetilde{A}$ 
for some $i$ and 
$M(\beta_i\beta_{i+1}) \in \widetilde{\calW}_\theta \subset \fd \widetilde{A}$ 
for some $i$.

Let $\theta=a_1[P_1]+a_2[P_2]+a_3[P_3] \in K_0(\proj A)_\R$.
The second condition is equivalent to $M(\beta_i\beta_{i+1}) \in \calW_\theta$,
since $M(\beta_i\beta_{i+1}) \in \fd A$.
This holds if and only if $a_1+a_2+a_3=0$.
On the other hand, the first condition is equivalent to that 
$M(\alpha_i) \in \calW_\theta$ and $S_{i+2} \in \calW_\theta$ for some $i$.
This precisely means that $a_i+a_{i+1}=0$, $a_i \ge 0$ and $a_{i+2}=0$.
Therefore,
\begin{align*}
R_0=\R_{\ge 0}([P_1]-[P_2])+\R_{\ge 0}([P_2]-[P_3])+\R_{\ge 0}([P_3]-[P_1]).
\end{align*}
\end{Ex}

\subsection{Key arguments}
\label{Subsec_key}

In this subsection, we show the following property,
which is crucial to prove our theorems.

\begin{Prop}\label{Prop_limit_sp}
Let $\theta \in R_0$.
Then, for any $p \in \ovMP_*(A)$, we have $M(p) \notin \calT_\theta$.
\end{Prop}

First, we define a techniqual symbol.

\begin{Def}\label{Def_critical_path}
For any $p \in \ovMP_*(A)$ and $\theta \in K_0(\proj A)_\R$, 
we set a subpath $s_{p,\theta}$ of $p$ so that $p=s_{p,\theta}t$ for some path $t$
and that $M(s_{p,\theta}) \cong M(p)/M$,
where $M$ is the maximum submodule of $\rad M(p)$ belonging to $\calT_\theta$.
\end{Def}

Then, we can prove the next property by using the modules 
in Definition \ref{Def_M_sigma_M_eta}.

\begin{Lem}\label{Lem_band_oplus_sp}
Assume that $A$ is a finite-dimensional special biserial algebra.
Let $p \in \ovMP_*(A)$.
Then, for any $\eta \in \INR(A)$ such that $M(p) \in \ovcalT_\eta$,
we have $M(s_{p,\eta}) \in \calW_\eta$.
\end{Lem}

\begin{proof}
Clearly, $M(s_{p,\eta}) \in \ovcalT_\eta$.
By Proposition \ref{Prop_TF_perp},
it suffices to show $\Hom_A(M_\eta(\lambda),M(s_{p,\eta}))=0$ 
for some $\lambda \in K^\times$.
Take $\lambda \in K^\times$.
Assume that there exists nonzero $\phi \in \Hom_A(M_\eta(\lambda),M(s_{p,\eta}))$.

By definition, any proper module $L \subset M(s_{p,\eta})$ belongs to $\ovcalF_\eta$.
Since $L$ is a string module, by using Proposition \ref{Prop_TF_perp}, 
we get $\Hom_A(M_\eta(\lambda),L)=0$ for any $\lambda \in K^\times$.
This implies that any $\phi \in \Hom_A(M_\eta(\lambda),M(s_{p,\eta}))$ is
zero or surjective.

If a surjection $\phi \in \Hom_A(M_\eta(\lambda),M(s_{p,\eta}))$ exists,
then there exist two arrows $\alpha,\alpha' \in Q_1$ and two paths $r,r'$ in $Q$
such that $M((\alpha'r')^{-1} s_{p,\eta} \alpha r)$ is isomorphic 
to a submodule of $M_\eta(\lambda)$.
This and $p \in \ovMP_*(A)$ imply $s_{p,\eta} \ne p$,
and we have $M(r) \in \calT_\eta$ by definition.

Thus, there exists a nonzero homomorphism $M_\eta(\lambda) \to M(r)$.
On the other hand, $M(r)$ is a proper submodule of $M_\eta(\lambda)$.
Thus, we have a nonzero non-isomorphism $M_\eta(\lambda) \to M_\eta(\lambda)$.
It contradicts Lemma \ref{Lem_tau_reduced_simple}.

Now, we have obtained $\Hom_A(M_\eta(\lambda),M(s_{p,\eta}))=0$.
Therefore, $M(s_{p,\eta}) \in \calW_\eta$.
\end{proof}

We also need an analogous result for indecomposable rigid elements.
The next notions are useful.

\begin{Def}
Assume that $A$ is a finite-dimensional special biserial algebra.
Let $\sigma \in \IR(A)$.
We define a module $L_\sigma$ as follows:
\begin{itemize}
\item
if $\sigma(h)=1$, then $L_\sigma:=M'_\sigma$;
\item
if $\sigma(h)=0$, then take the unique $p \in \ovMP_*(A)$
which admits the canonical surjection $\pi \colon M_\sigma \to M(p)$,
and set $L_\sigma:=\Ker \pi$;
\item
if $\sigma(h)=-1$, then $L_\sigma:=M_\sigma$.
\end{itemize}
\end{Def}

We can check that $L_\sigma$ is zero or indecomposable.
Moreover, $L_\sigma$ is a submodule of $M_\sigma$.
Then, we can show the following property.

\begin{Lem}\label{Lem_string_oplus_sp}
Assume that $A$ is a finite-dimensional special biserial algebra.
Let $p \in \ovMP_*(A)$.
Then, there exists $E \in \Z_{\ge 1}$ satisfying the following condition:
for any $\theta \in \IR(A)$ such that $M(p) \in \calT_\sigma$,
we have $0 \le \sigma(M(s_{p,\sigma})) \le E$.
\end{Lem}

\begin{proof}
Clearly, $M(s_{p,\sigma}) \in \ovcalT_\sigma$.
We have $\sigma(M(s_{p,\sigma}))=\dim_K \Hom_A(M_\sigma,M(s_{p,\sigma}))$
by Proposition \ref{Prop_TF_perp}.

To evaluate this, 
we show that $\phi(L_\sigma)=0$ for any $\phi \colon M_\sigma \to M(s_{p,\sigma})$.
If $\phi(L_\sigma) \ne 0$ for some $\phi \colon M_\sigma \to M(s_{p,\sigma})$,
then we can show that $\Hom_A(M(s_{p,\sigma}),M'_\sigma) \ne 0$ 
as in Lemma \ref{Lem_band_oplus_sp},
and this implies $M(s_{p,\sigma}) \notin \ovcalT_\sigma$ 
by Proposition \ref{Prop_TF_perp} again.
This contradicts the assumption, so $\phi(L_\sigma)=0$.
Thus,
\begin{align*}
\dim_K \Hom_A(M_\sigma,M(s_{p,\sigma}))=\dim_K \Hom_A(M_\sigma/L_\sigma,M(s_{p,\sigma})).
\end{align*}

Since $M_\sigma/L_\sigma$ is an indecomposable projective module or 
of the form $\bigoplus_{j=1}^{1+\sigma(h)} M(p_j)$ with each $p_j \in \ovMP_*(A)$,
there exists $E \in \Z_{\ge 1}$ such that $0 \le \sigma(M(s_{p,\sigma})) \le E$,
where $E$ does not depend on $\sigma$.
\end{proof}

If $\theta'$ is a direct summand of $\theta$ in $K_0(\proj A)$,
the paths $s_{p,\theta}$ and $s_{p,\theta'}$ do not coincide in general.
However, this is not a problem by the following property.

\begin{Lem}\label{Lem_critical_path_eq}
Assume that $A$ is a finite-dimensional special biserial algebra.
Let $p \in \ovMP_*(A)$ with $M(p) \in \ovcalT_\theta$,
and $\theta=\bigoplus_{i=1}^m \theta_i$ in $K_0(\proj A)$. 
Then, for any $i$, we have $\theta_i(M(s_{p,\theta}))=\theta_i(M(s_{p,\theta_i}))$.
\end{Lem}

\begin{proof}
Let $i \in \{1,2,\ldots,m\}$.
We consider the canonical surjections 
$\pi \colon M(p) \to M(s_{p,\theta})$ and $\pi_i \colon M(p) \to M(s_{p,\theta_i})$.
By Proposition \ref{Prop_direct_sum_TF}, 
$s_{p,\theta}$ is shorter than or equal to $s_{p,\theta_i}$, 
so we have a surjection $M(s_{p,\theta_i}) \to M(s_{p,\theta})$.
Define $X_i$ as its kernel, then $X_i$ is a proper submodule of $M(s_{p,\theta_i})$,
so $X_i \in \ovcalF_{\theta_i}$.
On the other hand, $X_i \cong \Ker \pi/\Ker \pi_i$ and 
$\Ker \pi \in \calT_\theta$ hold, so $X_i \in \calT_\theta$.
By Proposition \ref{Prop_direct_sum_TF}, we get $X_i \in \ovcalT_{\theta_i}$.
Therefore, $X_i \in \calW_{\theta_i}$, and 
$\theta_i(M(s_{p,\theta}))=\theta_i(M(s_{p,\theta_i}))-\theta_i(X_i)
=\theta_i(M(s_{p,\theta_i}))$.
\end{proof}

For $\theta=\sum_{i \in Q_0}a_i[P_i] \in K_0(\proj A)_\R$, we set
the 1-norm of $\theta$ by
\begin{align*}
\| \theta \|_1:=\sum_{i \in Q_0}|a_i|
\end{align*}
as usual.
By using this norm, we have the following crucial lemma.

\begin{Lem}\label{Lem_finite_cover_sp}
Let $p \in \ovMP_*(A)$ and $\epsilon \in \R_{>0}$, set
\begin{align*}
Y:=\{ \theta \in K_0(\proj A)_\R \mid \textup{
$\theta(X) > \epsilon \| \theta \|_1$ for any nonzero quotient $X$ of $M(p)$} \}.
\end{align*}
Then, there exists a finite subset $F \subset \IR(A)$ such that
\begin{align*}
Y \subset \bigcup_{\sigma \in F}N_\sigma.
\end{align*}
\end{Lem}

\begin{proof}
We may assume that $A$ is a finite-dimensional special biserial algebra 
by Proposition \ref{Prop_R_0_reduction}.
We set $F:=\IR(A) \cap Y$.
Then, $\sigma \in F$ implies $\| \sigma \|_1 < E/\epsilon$ 
for $E$ in Lemma \ref{Lem_string_oplus_sp}, so $F$ is a finite set.

We show that this finite set $F$ satisfies the assertion.
Since each $N_\sigma$ and $Y$ are both defined by
linear inequalities whose coefficients are integers,
it suffices to prove
\begin{align*}
Y \cap K_0(\proj A) \subset \bigcup_{\sigma \in F}N_\sigma.
\end{align*}

Let $\theta \in Y \cap K_0(\proj A)$,
and $\theta=\bigoplus_{i=1}^m \theta_i$ be the canonical decomposition.
By Proposition \ref{Prop_sign_coherent}, 
$\theta_1,\theta_2,\ldots,\theta_m$ are sign-coherent.
This and $\theta \in Y$ imply that there must exist $i \in \{1,2,\ldots,m\}$ such that
$\theta_i(M(s_{p,\theta})) > \epsilon \| \theta_i \|_1$.
By Lemma \ref{Lem_critical_path_eq}, 
$\theta_i(M(s_{p,\theta_i})) > \epsilon \| \theta_i \|_1>0$.
Since $M(p) \in \calT_\theta$ by $\theta \in Y$, we have $M(p) \in \ovcalT_{\theta_i}$
by Proposition \ref{Prop_direct_sum_TF}.
Thus, Lemma \ref{Lem_band_oplus_sp} yields 
$\theta_i \notin \INR(A)$; hence, $\theta_i \in \IR(A)$.
We show $\theta_i \in F$.
Let $X$ be a nonzero quotient module of $M(p)$.
By the definition of $M(s_{p,\theta_i})$, 
we have $\theta_i(X) \ge \theta_i(M(s_{p,\theta_i})) > \epsilon \| \theta_i \|$,
so $\theta_i \in F$.
By Lemma \ref{Lem_neighbor_basic}, we have $\theta \in N_{\theta_i}$.

Therefore, $Y \cap K_0(\proj A) \subset \bigcup_{\sigma \in F}N_\sigma$,
and we have the asserion.
\end{proof}

Now, we are ready to prove Proposition \ref{Prop_limit_sp}.

\begin{proof}[Proof of Proposition \ref{Prop_limit_sp}]
We assume $\theta \in R_0$ and $M(p) \in \calT_\theta$, and deduce a contradiction.
Since $M(p) \in \calT_\theta$,
there exists $\epsilon \in \R_{>0}$ such that $\theta(M(p))>\epsilon\|\theta\|_1$.
Then, Lemma \ref{Lem_finite_cover_sp} implies $\theta \in N_\sigma$ 
for some $\sigma \in \IR(A)$.
This contradicts $\theta \in R_0$.
\end{proof}

\subsection{Proof of Theorems \ref{Thm_R_0_gentle} and \ref{Thm_R_0_sp}}
\label{Subsec_proof}

In this section, we prove Theorems \ref{Thm_R_0_gentle} and \ref{Thm_R_0_sp}. 
We first introduce the following notion.

\begin{Def}
Let $A=\widehat{KQ}/I$ be a complete special biserial algebra.
Then, $A$ is called a \textit{truncated gentle algebra}
if there exist a complete gentle algebra $\widetilde{A}$ and 
$m \ge 1$ such that 
$A \cong \widetilde{A}/(I_\rmc^{\widetilde{A}})^m$
and $m l_c \ge 3$ for all $c \in \Cyc(\widetilde{A})$,
where $l_c$ is the length of $c$.
\end{Def}

By Proposition \ref{Prop_R_0_reduction}, 
it suffices to determine $R_0$ for truncated gentle algebras.
The assumption $m l_c \ge 3$ is needed to make the following property, 
which gentle algebras also satisfy, hold true for
truncated gentle algebras.

\begin{Lem}\label{Lem_trunc_length_2}
Let $A$ be a truncated gentle algebra.
\begin{itemize}
\item[(1)]
If $\alpha \in Q_1$ is an arrow ending at $i \in Q_0$ and 
$\beta \ne \gamma \in Q_1$ are arrows starting at $i$,
then $\alpha\beta \notin I$ or $\alpha\gamma \notin I$.
\item[(2)]
If $\alpha \in Q_1$ is an arrow starting at $i \in Q_0$ and 
$\beta \ne \gamma \in Q_1$ are arrows ending at $i$,
then $\beta\alpha \notin I$ or $\gamma\alpha \notin I$.
\end{itemize}
\end{Lem}

In the truncation process of gentle algebras, the sets $\MP$ and $\Cyc$ change as follows.

\begin{Lem}\label{Lem_MP_amalg_cycle}
Let $\widetilde{A}$ be a complete gentle algebra and $m \ge 1$. 
If $A:=\widetilde{A}/(I_\rmc^{\widetilde{A}})^m$ is a truncated gentle algebra,
then $A$ satisfies
\begin{align*}
\MP_*(A)&=\MP_*(\widetilde{A}) \amalg C,&
\MP^*(A)&=\MP^*(\widetilde{A}) \amalg C,&
\MP(A)&=\MP(\widetilde{A}) \amalg C,\\
\ovMP_*(A)&=\ovMP_*(\widetilde{A}) \amalg C,& 
\ovMP^*(A)&=\ovMP^*(\widetilde{A}) \amalg C,&
\ovMP(A)&=\ovMP(\widetilde{A}) \amalg C,&\\
\Cyc(A)&=\emptyset,
\end{align*}
where $C:=\{ c^{m-1}p_c \mid c \in \Cyc(\widetilde{A}) \}$.
\end{Lem}

The following properties directly follow 
from Lemmas \ref{Lem_MP_cycle_explicit} and \ref{Lem_MP_amalg_cycle}.
Note that the paths in (1)(ii) below for truncated gentle algebras $A$
are precisely the elements in $C$ of Lemma \ref{Lem_MP_amalg_cycle}.

\begin{Prop}\label{Prop_MP_bar_axiom}
Let $A$ be a truncated gentle algebra.
In \textup{(i)} and \textup{(ii)}, 
$\alpha$ and $\beta$ denote the first and the last arrows of $p$.
\begin{itemize}
\item[(1)]
The set $\ovMP_*(A)$ consists of the following paths in $Q$.
\begin{itemize}
\item[(i)]
The paths $p$ of length $\ge 1$ satisfying
$p \ne 0$ and $\beta \gamma=0$ in $A$ for any arrow $\gamma \in Q_1$.
\item[(ii)]
The paths $p$ of length $\ge 2$ satisfying $p \ne 0$ in $A$ and that 
there exists an arrow $\gamma \in Q_1$ such that
\begin{itemize}
\item
$c_p:=p \gamma$ and $c^p:=\gamma p$ are cycles in $Q$,
\item
$\beta \gamma \ne 0$, $\gamma \alpha \ne 0$, $p \gamma=0$ and $\gamma p=0$ in $A$.
\end{itemize}
\item[(iii)]
The paths $e_i$ for $i \in Q_0$ satisfying that 
there exists at most one arrow starting at $i$.
\end{itemize}
\item[(2)]
The set $\ovMP^*(A)$ consists of the following paths in $Q$.
\begin{itemize}
\item[(i)]
The paths $p$ of length $\ge 1$ satisfying
$p \ne 0$ and $\gamma \alpha=0$ in $A$ for any arrow $\gamma \in Q_1$.
\item[(ii)]
The same paths as \textup{(1)(ii)}.
\item[(iii)]
The paths $e_i$ for $i \in Q_0$ satisfying that 
there exists at most one arrow ending at $i$.
\end{itemize}
\end{itemize}
\end{Prop}

The next property is easily deduced, but will be important later.

\begin{Lem}\label{Lem_if_MP_non_cycle}
Let $A$ be a truncated gentle algebra, $p,q,q'$ be paths in $Q$
such that $p=q'q$ is of length $\ge 1$,
and $\alpha,\beta \in Q_1$ be the first and the last arrows of $p$.
\begin{itemize}
\item[(1)]
Let $p \in \MP_*(A)$. 
If $\beta \gamma = 0$ in $A$ for any $\gamma \in Q_1$, then $q \in \ovMP_*(A)$.
Otherwise, $p \in \MP(A)$ and is of type \textup{(ii)} 
in Proposition \ref{Prop_MP_bar_axiom}.
\item[(2)]
Let $p \in \MP^*(A)$.
If $\gamma \alpha = 0$ in $A$ for any $\gamma \in Q_1$, then $q' \in \ovMP^*(A)$.
Otherwise, $p \in \MP(A)$ and is of type \textup{(ii)} 
in Proposition \ref{Prop_MP_bar_axiom}.
\end{itemize}
\end{Lem}

\begin{proof}
This follows from Proposition \ref{Prop_MP_bar_axiom}.
\end{proof}

We have the following result.

\begin{Thm}\label{Thm_R_0_tg}
Let $A$ be a truncated gentle algebra.
For each $\theta \in K_0(\proj A)_\R$, the following conditions are equivalent.
\begin{itemize}
\item[(a)]
The element $\theta$ belongs to $R_0$.
\item[(b)]
Any $p \in \ovMP_*(A)$ satisfies $M(p) \notin \calT_\theta$,
and any $p \in \ovMP^*(A)$ satisfies $M(p) \notin \calF_\theta$.
\item[(c)]
For any $p \in \ovMP(A)$,
if $p$ is of type \textup{(ii)}, then there exists a path $q \in \ovMP(A)$ such that
$c_q$ is a cyclic permutation of $c_p$ and that $M(q) \in \calW_\theta$;
otherwise, $M(p) \in \calW_\theta$.
\item[(c$'$)]
For any $p \in \ovMP(A)$,
if $p$ is of type \textup{(ii)}, then $\theta(M(p))=0$;
otherwise, $M(p) \in \calW_\theta$.
\end{itemize}
In particular, $R_0$ is a rational polyhedral cone of $K_0(\proj A)_\R$.
\end{Thm}

\begin{proof}
$\text{(a)} \Rightarrow \text{(b)}$:
This is Proposition \ref{Prop_limit_sp} and its dual.

$\text{(b)} \Rightarrow \text{(c)}$:
Consider the case that $p$ is of type (ii) first.
Take $\alpha_1,\alpha_2,\ldots,\alpha_l \in Q_1$ such that
the cycle $c_p$ is $\alpha_1 \alpha_2 \cdots \alpha_l$.
Thus, $p=\alpha_1 \alpha_2 \cdots \alpha_{l-1}$.
Let $i_k \in Q_0$ denote the source of $\alpha_k$ for each $k \in \{1,2,\ldots,l\}$,
and for $u,v \in \{1,2,\ldots,l\}$, we set 
\begin{align*}
[[u,v]]:=\begin{cases}
\{u,u+1,\ldots,v\} & (u \le v) \\
\{u,u+1,\ldots,l\} \cup \{1,2,\ldots,v\} & (u > v) 
\end{cases}
\end{align*}
and 
\begin{align*}
f(u,v):=\sum_{k \in [[u,v]]} \theta(S_{i_k})
\end{align*}

If $\theta(M(p)) \ge 0$,
then we have a pair $(u,v)$ 
such that $[[u,v]]$ has the least elements among the pairs giving the maximum value of $f$,
and a path $q$ such that 
$c_q=(\alpha_u \alpha_{u+1} \cdots \alpha_l) \cdot (\alpha_1 \alpha_2 \cdots \alpha_{u-1})$.
Then, we can check that $M(q) \in \ovcalT_\theta$ by $\theta(M(p)) \ge 0$.
Moreover, if $\theta(M(p))>0$, then $M(q) \in \calT_\theta$,
but it contradicts (b).

Similarly, we can deduce a contradiction from $\theta(M(p))<0$.

Therefore, $\theta(M(p))=0$ and the path $q$ above satisfies $M(q) \in \calW_\eta$.
Clearly, $c_q$ is a cyclic permutation of $c_p$.

If $p$ is of type (i) or (iii), 
(b) and Lemma \ref{Lem_if_MP_non_cycle} immediately give the assertion.

$\text{(c)} \Rightarrow \text{(c$'$)}$: 
It is obvious.

$\text{(c$'$)} \Rightarrow \text{(a)}$:
We suppose that $\theta \notin R_0$, and show that (c$'$) is not satisfied.
Then, we can take some indecomposable rigid element $\sigma \in \IR(A)$ such that
$\theta \in N_\sigma$.

If $\sigma(h) \ge 0$, then we can find $q \in \ovMP_*(A)$ such that
$M(q)$ is a quotient module of $M_\sigma$.
Since $\theta \in N_\sigma$, we have $M(q) \in \calT_\theta$.
If $q$ is of type (ii), then $q \in \MP(A)$ and $M(q)>0$, which denies (c$'$).
Otherwise, there exists some path $r$ such that $p:=rq \in \ovMP(A)$ 
is of type (i) or (iii).
Then, $M(q) \in \calT_\theta$ implies $M(p) \notin \calW_\theta$, which contradicts (c$'$).
We can check that (c$'$) is not satisfied in the case $\sigma(h) < 0$, either.

Thus, (c) implies that $\theta \in R_0$.
\end{proof}

Now, Theorem \ref{Thm_R_0_gentle} follows almost immediately.

\begin{proof}[Proof of Theorem \ref{Thm_R_0_gentle}]
Proposition \ref{Prop_R_0_reduction} allows us to consider 
a truncated gentle algebra $A/I_\rmc^m$ instead of the complete gentle algebra $A$.
We apply Theorem \ref{Thm_R_0_tg} to the truncated gentle algebra $A/I_\rmc^m$.
Then, by Lemma \ref{Lem_MP_amalg_cycle} and Proposition \ref{Prop_MP_bar_axiom},
the conditions (c) and (c$'$) in Theorem \ref{Thm_R_0_tg} are equivalent to
(b) and (b$'$) in Theorem \ref{Thm_R_0_gentle}, respectively.
\end{proof}

We can also prove Theorem \ref{Thm_R_0_sp}.

\begin{proof}[Proof of Theorem \ref{Thm_R_0_sp}]
By Proposition \ref{Prop_R_0_reduction}, we may assume that
$A$ is a finite-dimensional special biserial algebra.

$\text{(b)} \Rightarrow \text{(c)}$:
It follows from Theorem \ref{Thm_R_0_gentle}
and Lemma \ref{Lem_path_W_MP_not_T}.

$\text{(c)} \Rightarrow \text{(b)}$:
First, $\theta \in R_0(\widetilde{A})$ follows from Theorem \ref{Thm_R_0_gentle} and (c).

Let $p \in \ovMP_*(A)$.
If $p$ is of length 0,
then we can take $\widetilde{p} \in \ovMP(\widetilde{A})$ ending with $p$.
Then, $M(\widetilde{p}) \in \widetilde{\calW}_\theta$ by (c),
and since $M(p)=\soc M(\widetilde{p})$, we get $M(p) \notin \calT_\theta$.

In the rest, we assume that $p \in \MP_*(A)$ of length $\ge 1$.
We first show that 
some path $\widetilde{p}$ admitted in $\widetilde{A}$ satisfies that
$M(\widetilde{p}) \in \widetilde{\calW}_\theta$ and that
$p$ is a subpath of $\widetilde{p}$.

If there exists 
some $\widetilde{p} \in \ovMP(\widetilde{A})$ ending with $p$,
then $M(\widetilde{p}) \in \widetilde{\calW}_\theta$ by (c).

Otherwise, there exists some $\widetilde{c} \in \Cyc(\widetilde{A})$
such that $(\widetilde{c})^m$ ends with $p$ for some $m \ge 1$.
Then, take a cyclic permutation $\widetilde{d}$ of $\widetilde{c}$ such that 
that $M(\widetilde{p}_{\widetilde{d}}) \in \widetilde{\calW}_\theta$ by (c).
We can check that $p$ is a subpath of 
$\widetilde{p}:=(\widetilde{d})^m\widetilde{q}$,
where $\widetilde{d}=\widetilde{q}\beta$ for some $\beta \in Q_1$.
Moreover, $M(\widetilde{p}) \in \widetilde{\calW}_\theta$.

In any case, we have found $\widetilde{p}$ admitted in $\widetilde{A}$ such that
$M(\widetilde{p}) \in \widetilde{\calW}_\theta$ and that
$p$ is a subpath of $\widetilde{p}$.
Write $\widetilde{p}=\widetilde{s}p\widetilde{t}$.

By applying Lemma \ref{Lem_path_tilde_W} to $\widetilde{p}$,
we can take paths $q_1,q_2,\ldots,q_k$ that are strings in $A$ and
arrows $\alpha_1,\alpha_2,\ldots,\alpha_{k-1} \in Q_1$ such that 
$M(q_j) \in \calW_\theta \cap \fd A$
and $\widetilde{p}=q_1\alpha_1q_2\alpha_2 \cdots q_k$.
Set $j \in \{1,2,\ldots,k\}$ as the minimum $j$
such that $q_1\alpha_1q_2\alpha_2 \cdots q_j$ is strictly longer than 
or equal to $\widetilde{s}$.
For this $j$, 
we can check that there exists 
a nonzero quotient module of $M(p)$ that is a submodule of $M(q_j)$
by using $p \in \MP_*(A)$.
Since $M(q_j) \in \calW_\theta$, we get $M(p) \notin \calT_\theta$.

Dually, we can prove $\text{(b$'$)} \Leftrightarrow \text{(c)}$.
It remains to show $\text{(a)} \Leftrightarrow \text{((b) and (b$'$))}$.

$\text{(a)} \Rightarrow \text{((b) and (b$'$))}$:
It is clear that $R_0 \subset R_0(\widetilde{A})$.
The other part follows from Proposition \ref{Prop_limit_sp} and its dual.

$\text{((b) and (b$'$))} \Rightarrow \text{(a)}$:
Assume that $\theta \notin R_0$, and show that (b) or (b$'$) is not satisfied.
By definition, there exists some indecomposable rigid $\sigma \in \IR(A)$ 
such that $\theta \in N_\sigma$.

If $\sigma(h) \ge 0$, then we can take $p \in \ovMP_*(A)$ such that 
$M(p)$ is a quotient module of $M_\sigma$.
Since $\theta \in N_\sigma$, we get $M(p) \in \calT_\theta$, which contradicts (b).
Similarly, in the case $\sigma(h) < 0$, we find $p \in \ovMP^*(A)$ such that 
$M(p) \in \calF_\theta$, which means that (b$'$) is not satisfied.

Thus, $\theta \in R_0$.

The last statement follows from Corollary \ref{Cor_R_0_h}.
\end{proof}

\subsection{Results on the non-rigid regions}
\label{Subsec_complement}

In this subsection, we determine $R_U$ for each $U \in \twopresilt A$ in terms of $R_0$.
We recall $R_U=C^+(U)+(\overline{N_U} \cap R_0)$
from Proposition \ref{Prop_R_U_decompose}, 
so it suffices to describe $\overline{N_U} \cap R_0$ explicitly.

In this subsection,
we use the $\tau$-rigid module $H^0(U)$ and the $\tau^{-1}$-rigid module $H^{-1}(\nu U)$
corresponding to $U$, so unless otherwise stated,
we assume that $A$ is a finite-dimensional special biserial algebra.
However, we can apply our results to all complete special biserial algebras
by Proposition \ref{Prop_R_0_reduction}.

For a techniqual reason, we first define a certain quotient algebra of $A$.

\begin{Def}\label{Def_A_theta}
For each $\theta \in R_0$, we define a quotient algebra $A_\theta:=A/I_\theta$,
where the ideal $I_\theta$ of $A$ is generated 
by the following paths $s_p,t_p$ and the paths $p,q$ such that $p=q \ne 0$ in $A$:
\begin{itemize}
\item
For each $p \in \ovMP_*(A)$ such that $p \notin \ovcalF_\theta$,
take subpaths $q,r$ of $p$ and $\alpha \in Q_1$ so that $p=q \alpha r$,
$M(r) \in \calT_\theta$ and that $M(q) \in \ovcalF_\theta$.
We set $s_p:=q \alpha$.
\item
For each $p \in \ovMP^*(A)$ such that $p \notin \ovcalT_\theta$,
take subpaths $q,r$ of $p$ and $\alpha \in Q_1$ so that $p=q \alpha r$,
$M(r) \in \ovcalT_\theta$ and that $M(q) \in \calF_\theta$.
We set $t_p:=\alpha r$.
\end{itemize}
\end{Def}

Clearly, $A_\theta$ is also a finite-dimensional string algebra.
Note that $A_\theta=A$ if $A$ is a truncated gentle algebra
by Theorem \ref{Thm_R_0_tg}, since any $s_p$ or $t_p$ is not defined.
The next properties follow from the definition of $A_\theta$.

\begin{Lem}\label{Lem_path_W_admit_A_theta}
Let $\theta \in R_0$. 
If $p$ is a path admitted in $A$ and $M(p) \in \calW_\theta$,
then $p$ is admitted also in $A_\theta$.
\end{Lem}

\begin{Lem}\label{Lem_A_theta_MP}
Let $\theta \in R_0$. 
Then, the following statements hold.
\begin{itemize}
\item[(1)]
For any $p \in \ovMP_*(A_\theta)$, we have $M(p) \in \ovcalF_\theta$.
\item[(2)]
For any $p \in \ovMP^*(A_\theta)$, we have $M(p) \in \ovcalT_\theta$.
\end{itemize}
\end{Lem}

We give an example.

\begin{Ex}\label{Ex_A_theta}
We use the complete special biserial algebra $A$ in Example \ref{Ex_rectangle_sp}.
The elements of $\ovMP_*(A/I_\rmc)$ are
\begin{align*}
\alpha_i\alpha_{i+1},\beta_i\beta_{i+1},\gamma_i\gamma_{i+1} \ (i \in \{1,2,3\}), 
\delta_1\delta_2\delta_3,\delta_2\delta_3\delta_4,\delta_3\delta_4,\delta_4,e_8,
\end{align*}
and 
$\ovMP^*(A/I_\rmc)$ are
\begin{align*}
\alpha_i\alpha_{i+1},\beta_i\beta_{i+1},\gamma_i\gamma_{i+1} \ (i \in \{1,2,3\}), 
e_1,\delta_1,\delta_1\delta_2,\delta_1\delta_2\delta_3,\delta_2\delta_3\delta_4.
\end{align*}

We have obtained $R_0=\R_{\ge 0}\eta_1 \cup \R_{\ge 0}\eta_2$,
where $\eta_1:=[P_1]-[P_4]-[P_5]+[P_6]$ and $\eta_2:=[P_5]-[P_8]$.
By direct calculation,
\begin{align*}
(A/I_\rmc)_{\eta_1}&=(A/I_\rmc)/
\langle \alpha_3, \beta_2, \gamma_1, \delta_3 \delta_4 \rangle, & 
(A/I_\rmc)_{\eta_2}&=(A/I_\rmc)/
\langle \gamma_3, \delta_1 \delta_2 \rangle.
\end{align*}
\end{Ex}

By using these, we get useful information on $\overline{N_U} \cap R_0$.

\begin{Lem}\label{Lem_A_theta_tau_rigid_admitted}
Let $\sigma \in \IR(A) \setminus \{\pm[P_i]\}$ and 
$\theta \in \overline{N_\sigma} \cap R_0$.
Take a string $s$ admitted in $A$ such that $M(s) \cong M_\sigma$,
and write $s=p_1^{-1}p_2p_3^{-1}p_4 \cdots p_{2k-1}^{-1}p_{2k}$
with $p_i$ paths in $Q$ admitted in $A$ and $p_i$ is of length $\ge 1$ if $i \ne 1,2k$.
\begin{itemize}
\item[(1)]
For any $i \in \{2,3,\ldots,2k-1\}$, the path $p_i$ is admitted in $A_\theta$.
\item[(2)]
For any $i \in \{1,2k\}$, 
the path $p_i$ belongs to $\ovMP_*(A)$ or $p_i$ is admitted in $A_\theta$.
\end{itemize}
\end{Lem}

\begin{proof}
(1)
Let $i \in \{2,3,\ldots,2k-1\}$, and assume that $p_i$ is not admitted in $A_\theta$.
We take the longest path $q$ admitted in $A_\theta$ such that 
$p_i=q \alpha r$ for some arrow $\alpha \in Q_1$ and some path $r$.
By the definition of $A_\theta$, we get $M(r) \in \calT_\theta$.
We can see that $M(r)$ is a submodule of $\tau M(s)=M'_\sigma$ 
by Proposition \ref{Prop_AR}.
Then, $M(r) \in \calT_\theta$ implies $M'_\sigma \notin \ovcalF_\theta$,
but it contradicts $\theta \in \overline{N_\sigma}$ by Lemma \ref{Lem_neighbor_basic}.
Therefore, $p_i=q$, and $p_2,p_3,\ldots,p_{2k-1}$ are admitted in $A_\theta$.

(2)
By symmetry, we only prove the case $i=1$.

Assume that $p_1 \notin \ovMP_*(A)$.
Then, we can take $\alpha \in Q_1$ and $p_0 \in \ovMP^*(A)$
such that $p_0 \alpha^{-1} p_1^{-1}$ is admitted in $A$.
By Proposition \ref{Prop_AR}, $M(p_0)$ is a submodule of $\tau M_\sigma=M'_\sigma$, so
$\theta \in \overline{N_U}$ implies $M(p_0) \in \ovcalF_\theta$ 
by Lemma \ref{Lem_neighbor_basic}.
On the other hand, $\theta \in R_0$ and Theorem \ref{Thm_R_0_sp} yield that
$M(p_0) \notin \calF_\theta$.
Thus, we can find the shortest path $r_0$ 
such that $p_0$ ends with $r_0$ and that $M(r_0) \in \calW_\theta$.
Since $M(p_1)$ is a quotient module of $M_\sigma$, we get $M(p_1) \in \ovcalT_\sigma$.

If $p_1$ is not admitted in $A_\theta$, then the same argument for $p_i$ with 
$i \in \{2,3,\ldots,2k-1\}$ gives a decomposition $p_1=q_1 \alpha_1 r_1$
such that $M(r_1) \in \calT_\theta$.
Then, by the minimality of $r_0$, we have $M(r_0 \alpha^{-1} r_1^{-1}) \in \calT_\theta$. 
However, $M(r_0 \alpha^{-1} r_1^{-1})$ is a submodule of $\tau M(s)=M'_\sigma$ 
by Proposition \ref{Prop_AR}.
Thus, $M(r_0 \alpha^{-1} r_1^{-1}) \in \calT_\theta$ implies 
$\theta \notin \overline{N_\sigma}$,
which contradicts our assumption.

Therefore, $p_1 \in \ovMP_*(A)$ or $p_i$ is admitted in $A_\theta$.
\end{proof}

We can obtain $\theta$-semistable modules 
from 2-term presilting objects $U \in \twopresilt A$
such that $\theta \in \overline{N_U} \cap R_0$ as follows.

\begin{Prop}\label{Prop_A_theta_tau_rigid_semistable}
Let $U \in \twopresilt A$ and $\theta \in \overline{N_U} \cap R_0$.
Set $U_\theta:=U \otimes_A A_\theta$.
\begin{itemize}
\item[(1)]
We have $U_\theta \in \twopresilt A_\theta$ and $C^+(U_\theta)=C^+(U)$.
\item[(2)]
The modules $H^0(U_\theta)$ and $H^{-1}(\nu_\theta U_\theta)$ belong to $\calW_\theta$,
where $\nu_\theta$ is the Nakayama functor 
$\sfK^\rmb(\proj A_\theta) \to \sfK^\rmb(\inj A_\theta)$.
\item[(3)]
The semibricks $\calS_{U_\theta}$ and $\calS'_{U_\theta}$ are contained in $\calW_\theta$.
\item[(4)]
Let $U=\bigoplus_{i=1}^m U_i$ with $U_i$ indecomposable and $\sigma_i:=[U_i]$.
Then, $L_{\sigma_i} \in \calW_\theta \cap \fd A_\theta$ holds for all $i$.
\end{itemize}
\end{Prop}

\begin{proof}
We may assume that $U$ is indecomposable in the proof of (1) and (2).
Set $\sigma:=[U]$, then $M_\sigma=H^0(U)$ and $M'_\sigma=H^{-1}(\nu U)$.

(1)
The first statement is clear.

For the second statement,
it suffices to show that $U_\theta$ is still indecomposable.
In the case that $\sigma=\pm [P_i]$ for some $i \in Q_0$, this is clear.
Otherwise, $M_\sigma$ is a string module $M(s)$ in $\fd A$.
By Lemma \ref{Lem_A_theta_tau_rigid_admitted}, $M(s) \otimes_A A_\theta$ is indecomposable,
and it is $H^0(U_\theta)$.
Since $U_\theta$ is not of the form $P_i[1]$,
we get that $U_\theta$ is indecomposable as desired.

(2)
We set $(M_\sigma)_\theta:=M_\sigma \otimes_A A_\theta=H^0(U_\theta)$
and $(M'_\sigma)_\theta:=\Hom_A(A_\theta,M'_\sigma)=H^{-1}(\nu_\theta U_\theta)$.

By Lemma \ref{Lem_A_theta_tau_rigid_admitted}, we have an injection 
$\psi \colon L_\sigma \to (M_\sigma)_\theta$
and a surjection 
$\psi' \colon (M'_\sigma)_\theta \to L_\sigma$.
Since $\theta \in \overline{N_U}$, 
we get $(M_\sigma)_\theta \in \ovcalT_\theta$ and $(M'_\sigma)_\theta \in \ovcalF_\theta$,
so $L_\sigma \in \calW_\theta$.
If $\Coker \psi, \Ker \psi' \in \calW_\theta$, then we obtain the assertion.
We only prove $\Coker \psi \in \calW_\theta$, since the other one can be shown similarly.

First, we assume that $\sigma=[P_i]$ for some $i \in Q_0$, then 
$L_\sigma=0$ and $\Coker \psi_1 \cong (P_i)_\theta$.
There exists a short exact sequence 
\begin{align*}
0 \to (P_i)_\theta \to M(p_1) \oplus M(p_2) \to S_i \to 0,
\end{align*}
with $p_1, p_2 \in \ovMP_i(A_\theta)$.
Since $M(p_1),M(p_2) \in \ovcalF_\theta$ by Lemma \ref{Lem_A_theta_MP}, 
we have $(P_i)_\theta \in \ovcalF_\theta$.
On the other hand, $(P_i)_\theta=(M_\sigma)_\theta \in \ovcalT_\theta$.
Thus, $\Coker \psi \cong (P_i)_\theta \in \calW_\theta$.

Otherwise, we can see that
$\Coker \psi$ is of the form $\bigoplus_{j=1}^{1+\sigma(h)} M(p_j)$
(including the case $\sigma=-[P_i]$),
where each $p_j$ is in $\ovMP_*(A)$.
Then, $\Coker \psi \in \ovcalF_\theta$ by Lemma \ref{Lem_A_theta_MP}.
Since $\Coker \psi$ is a quotient module of $(M_\sigma)_\theta \in \ovcalT_\theta$,
we get $\Coker \psi \in \ovcalT_\theta$.
Thus, $\Coker \psi \in \calW_\theta$.

By the argument above, we have $(M_\sigma)_\theta, (M'_\sigma)_\theta \in \calW_\theta$ 
as desired.

(3) follows from the definition of the two semibricks and (2).
 
(4) has been shown in the proof of (2).
\end{proof}

We can finally describe $\overline{N_U} \cap R_0$ explicitly
by using stability conditions.

\begin{Thm}\label{Thm_N_U_R_0}
Let $U=\bigoplus_{i=1}^m U_i \in \twopresilt A$ with $U_i$ indecomposable and 
$\sigma_i:=[U_i]$,
and consider the injection
$\psi_i \colon L_{\sigma_i} \to M_{\sigma_i}$ and 
the surjection $\psi'_i \colon M'_{\sigma_i} \to L_{\sigma_i}$ for each $i$.
Then, for any $\theta \in R_0$, the following conditions are equivalent.
\begin{itemize}
\item[(a)]
The element $\theta$ belongs to $\overline{N_U} \cap R_0$.
\item[(b)]
For any $i$, the conditions 
$L_{\sigma_i} \in \calW_\theta$, $\Coker \psi_i \in \ovcalT_\theta$ and
$\Ker \psi'_i \in \ovcalF_\theta$ hold.
\item[(c)]
Both $H^0(U) \in \ovcalT_\theta$ and $H^{-1}(\nu U) \in \ovcalF_\theta$ hold.
\item[(d)]
Both $\calS_U \in \ovcalT_\theta$ and $\calS'_U \in \ovcalF_\theta$ hold.
\end{itemize}
Moreover, if $A$ is a truncated gentle algebra, 
then the conditions above are also equivalent to the following one.
\begin{itemize}
\item[(b$'$)]
For any $i$, the conditions 
$L_{\sigma_i}, \Coker \psi_i, \Ker \psi'_i \in \calW_\theta$ hold.
\item[(c$'$)]
The modules $H^0(U)$ and $H^{-1}(\nu U)$ belong to $\calW_\theta$.
\item[(d$'$)]
The semibricks $\calS_U$ and $\calS'_U$ are contained in $\calW_\theta$.
\end{itemize}
\end{Thm}

\begin{proof}
$\text{(a)} \Rightarrow \text{(b)}$: 
In Proposition \ref{Prop_A_theta_tau_rigid_semistable},
$L_{\sigma_i} \in \calW_\theta$ is shown.
The remaining conditions are clear,
since $\Coker \psi_i$ is a quotient module of $M_{\sigma_i} \in \ovcalT_\theta$
and $\Ker \psi'_i$ is a submodule of $M'_{\sigma_i} \in \ovcalF_\theta$.

$\text{(b)} \Rightarrow \text{(a)}$ and 
$\text{(a)} \Leftrightarrow \text{(c)}$ and $\text{(a)} \Leftrightarrow \text{(d)}$
follow from Lemma \ref{Lem_neighbor_basic}.

If $A$ is a truncated gentle algebra, then 
(a) implies (b$'$), (c$'$) and (d$'$) by
Proposition \ref{Prop_A_theta_tau_rigid_semistable}, 
since $A_\theta=A$ for any $\theta \in R_0$.

$\text{(b$'$)} \Rightarrow \text{(b)}$, $\text{(c$'$)} \Rightarrow \text{(c)}$ and 
$\text{(d$'$)} \Rightarrow \text{(d)}$ are obvious.
\end{proof}

We give an example that (b$'$) is not equivalent to the other conditions.

\begin{Ex}
We continue Example \ref{Ex_A_theta}.
Recall that $\eta_1=[P_1]-[P_4]-[P_5]+[P_6]$ and $\eta_2=[P_5]-[P_8]$.

Let $\sigma:=[P_1]-[P_4] \in K_0(\proj A)$, which also belongs to $\IR(A)$.
Take $U \in \twopresilt A$ satisfying $[U]=\sigma$, and set 
$\overline{U}:=U \otimes_A (A/I_\rmc)$.
Then,
\begin{align*}
M_\sigma&=H^0(\overline{U})=\begin{tikzpicture}[baseline=(0.base),->]
\node (0) at (   0,   0) {$\scriptstyle \mathstrut$};
\node (1) at (   0,   0) {$\scriptstyle 2$};
\node (2) at ( 0.6, 0.6) {$\scriptstyle 1$};
\node (3) at ( 1.2,   0) {$\scriptstyle 3$};
\node (4) at ( 1.8,-0.6) {$\scriptstyle 5$};
\node (5) at ( 2.4,-1.2) {$\scriptstyle 7$};
\draw (2) to (1);
\draw (2) to (3);
\draw (3) to (4);
\draw (4) to (5);
\end{tikzpicture}, &
M'_\sigma&=H^{-1}(\nu \overline{U})=\begin{tikzpicture}[baseline=(0.base),->]
\node (0) at (   0,   0) {$\scriptstyle \mathstrut$};
\node (1) at (-1.8, 0.6) {$\scriptstyle 6$};
\node (2) at (-1.2,   0) {$\scriptstyle 3$};
\node (3) at (-0.6,-0.6) {$\scriptstyle 4$};
\node (4) at (   0,   0) {$\scriptstyle 2$};
\draw (1) to (2);
\draw (2) to (3);
\draw (4) to (3);
\end{tikzpicture}.
\end{align*}
In this case, we have $\overline{N_U} \cap R_0=\R_{\ge 0} \eta_1 \cup \R_{\ge 0} \eta_2$.

Moreover,
\begin{align*}
L_\sigma&=S_2, &
\Coker \psi&=\begin{tikzpicture}[baseline=(0.base),->]
\node (0) at ( 0.6,   0) {$\scriptstyle \mathstrut$};
\node (2) at ( 0.6, 0.6) {$\scriptstyle 1$};
\node (3) at ( 1.2,   0) {$\scriptstyle 3$};
\node (4) at ( 1.8,-0.6) {$\scriptstyle 5$};
\node (5) at ( 2.4,-1.2) {$\scriptstyle 7$};
\draw (2) to (3);
\draw (3) to (4);
\draw (4) to (5);
\end{tikzpicture}, &
\Ker \psi'&=\begin{tikzpicture}[baseline=(0.base),->]
\node (0) at (-0.6,   0) {$\scriptstyle \mathstrut$};
\node (1) at (-1.8, 0.6) {$\scriptstyle 6$};
\node (2) at (-1.2,   0) {$\scriptstyle 3$};
\node (3) at (-0.6,-0.6) {$\scriptstyle 4$};
\draw (1) to (2);
\draw (2) to (3);
\end{tikzpicture}.
\end{align*}
We can see that all elements in $\R_{\ge 0} \eta_1 \cup \R_{\ge 0} \eta_2$ satisfy 
$L_\sigma \in \calW_\theta$, $\Coker \psi \in \ovcalT_\theta$ and 
$\Ker \psi' \in \ovcalF_\theta$.
However, we also have $\Coker \psi \notin \calW_{\eta_2}$.
Thus, (b$'$) is not equivalent to (a) and (b) in Theorem \ref{Thm_N_U_R_0}.

Proposition \ref{Prop_A_theta_tau_rigid_semistable} is still valid; we can check that both
because 
\begin{align*}
(M_\sigma)_{\eta_2}&=\begin{tikzpicture}[baseline=(0.base),->]
\node (0) at (   0,   0) {$\scriptstyle \mathstrut$};
\node (1) at (   0,   0) {$\scriptstyle 2$};
\node (2) at ( 0.6, 0.6) {$\scriptstyle 1$};
\node (3) at ( 1.2,   0) {$\scriptstyle 3$};
\draw (2) to (1);
\draw (2) to (3);
\end{tikzpicture}, &
(M'_\sigma)_{\eta_2}&=\begin{tikzpicture}[baseline=(0.base),->]
\node (0) at (   0,   0) {$\scriptstyle \mathstrut$};
\node (1) at (-1.8, 0.6) {$\scriptstyle 6$};
\node (2) at (-1.2,   0) {$\scriptstyle 3$};
\node (3) at (-0.6,-0.6) {$\scriptstyle 4$};
\node (4) at (   0,   0) {$\scriptstyle 2$};
\draw (1) to (2);
\draw (2) to (3);
\draw (4) to (3);
\end{tikzpicture}
\end{align*}
both belong to $\calW_{\eta_2}$.
\end{Ex}

In Corollary \ref{Cor_conn_comp_sp}, 
we will show that any $T \in \twosilt A$ can be obtained 
by iterated mutations from $A$ or $A[1]$ in $\twosilt A$
if $A$ is a complete special biserial algebra.
Therefore, we can obtain each $U \in \twopresilt A$ by mutations.
By applying Theorems \ref{Thm_R_0_sp} and \ref{Thm_N_U_R_0} for each $U \in \twopresilt A$,
we can ``eventually'' determine the non-rigid region $\NR \subset K_0(\proj A)_\R$ by
Proposition \ref{Prop_R_U_decompose}.
The following property for truncated gentle algebras is important to prove
Corollary \ref{Cor_conn_comp_sp}.

\begin{Cor}\label{Cor_Jasso_inv_lin}
Let $A$ be a truncated gentle algebra, and $U \in \twopresilt A$. 
We define $B:=\End_A(H^0(T))/[H^0(U)]$, 
where $T \in \twosilt A$ is the Bongartz completion of $U$.
Then, the inverse $\phi^{-1} \colon R_0(B) \to \overline{N_U} \cap R_0$ of the bijection 
$\phi \colon \overline{N_U} \cap R_0 \to R_0(B)$ in Lemma \ref{Lem_bij_to_R_0_B}
is a restriction of an $\R$-linear map $K_0(\proj B)_\R \to K_0(\proj A)_\R$.
\end{Cor}

\begin{proof}
Decompose $U=\bigoplus_{i=1}^m U_i$ and 
$T=U \oplus \left( \bigoplus_{j=1}^{n-m} T_j \right)$ 
so that $U_i$ and $T_j$ are indecomposable.
We set $\sigma_i:=[U_i]$ for $i \in \{1,2,\ldots,m\}$
and $\sigma'_j:=[T_j]$ for $j \in \{1,2,\ldots,n-m\}$.

For each $i \in \{1,2,\ldots,m\}$,
we define the modules $X_i$ and $X'_i$ as in Lemma \ref{Lem_semibrick} for $U \in \twopresilt A$; namely, 
\begin{align*}
X_i&:=H^0(\overline{U_i})/\sum_{f \in \rad_A(H^0(\overline{U}),H^0(\overline{U_i}))} \Im f,&
X'_i&:=\bigcap_{f \in \rad_A(H^{-1}(\nu \overline{U_i}),H^{-1}(\nu \overline{U}))} \Ker f.
\end{align*}
Then, at least $X_i$ or $X'_i$ is nonzero for each $i \in \{1,2,\ldots,m\}$.
Thus, we set 
\begin{align*}
x_i:=\begin{cases}
[X_i] & (X_i \ne 0) \\
-[X'_i] & (X_i=0)
\end{cases} \in K_0(\fd A) \setminus \{0\}.
\end{align*}
Then, for any $i,k \in \{1,2,\ldots,m\}$,
Lemma \ref{Lem_semibrick} tells us $\sigma_k(x_i)=\delta_{k,i}$.
We also have $\theta(x_i)=0$ for all $\theta \in \overline{N_U} \cap R_0$
by Theorem \ref{Thm_N_U_R_0}.

Recall that $\phi$ is a restriction of the linear projection 
$K_0(\proj A)_\R \to K_0(\proj B)_\R$ such that
$\sigma_i \mapsto 0$ for each $i \in \{1,2,\ldots,m\}$ and
$\sigma'_j \mapsto [P_j^B]$ for each $j \in \{1,2,\ldots,n-m\}$,
where $P_j^B$ is the indecomposable projective $B$-module corresponding to $T_j$.
Thus, if $\theta^B=\sum_{j=1}^{n-m}b_j[P_j^B] \in R_0(B)$, then 
\begin{align*}
\theta:=\phi^{-1}(\theta^B)=
\sum_{j=1}^{n-m}b_j \sigma'_j - \sum_{i=1}^m c_{\theta,i} \sigma_i
\in \overline{N_U} \cap R_0
\end{align*}
for some $c_{\theta,i} \in \R$.
By $\theta(x_i)=0$ and $\sigma_k(x_i)=\delta_{k,i}$ for $i,k \in \{1,2,\ldots,m\}$, 
we have $c_{\theta,i}=\sum_{j=1}^{n-m}b_j \sigma'_j(x_i)$.
Therefore,
\begin{align*}
\phi^{-1} \left(\sum_{j=1}^{n-m}b_j[P_j^B]\right)
=\sum_{j=1}^{n-m}\sum_{i=1}^m b_j (\sigma'_j-\sigma'_j(x_i) \cdot \sigma_i).
\end{align*}
In particular, it is a restriction of 
an $\R$-linear map $K_0(\proj B)_\R \to K_0(\proj A)_\R$.
\end{proof}

\section{Applications}
\label{Sec_app}

In this section, we show some applications of our results in the previous section.
For any $U \in \twopresilt A$ and any finite-dimenisonal algebra $B$, we set
\begin{align*}
h_U&:=\sum_{X \in \simple \calW_U}[X], & h^B&:=\sum_{S^B \in \simple (\fd B)}[S^B].
\end{align*}

\subsection{g-tameness}
\label{Subsec_g-tame}

For any complete special biserial algebra,
we have shown that $R_0 \subset \Ker \langle ?,h \rangle$ in Theorem \ref{Thm_R_0_sp}.
Moreover, the class of finite-dimensional special biserial algebras are closed under 
$\tau$-tilting reduction by Theorem \ref{Thm_sp_closed_under_Jasso}.
From these results, we can show that the non-rigid region 
is contained in a union of countably many hyperplanes of codimension one.

\begin{Cor}\label{Cor_cone_dense}
Let $A$ be a complete special biserial algebra.
Then, we have
\begin{align*}
\NR=K_0(\proj A)_\R \setminus \Cone \subset 
\bigcup_{U \in \twopresilt A, \ |U| \le n-2} \Ker \langle ?,h_U \rangle.
\end{align*}
In particular, $\Cone$ is dense in $K_0(\proj A)_\R$.
\end{Cor}

\begin{proof}
We may assume that $A$ is finite-dimensional by Proposition \ref{Prop_reduction_brick}.
If $U \in \twopresilt A$ with $|U| \ge n-1$, then $R_U=C^+(U) \subset \Cone$
by Lemma \ref{Lem_R_U_almost_silt}.
Thus, it suffices to show that $R_U \subset \Ker \langle ?,h_U \rangle$
for all $U \in \twopresilt A$ with $|U| \le n-2$
by Proposition \ref{Prop_R_U_decompose}.

Fix $U \in \twopresilt A$  with $|U| \le n-2$, 
take its Bongartz completion $T \in \twosilt A$,
and set $B$ as the algebra $\End_A(H^0(T))/[H^0(U)]$.
The linear map $\pi \colon K_0(\proj A)_\R \to K_0(\proj B)_\R$ in Proposition \ref{Prop_Jasso_Grothendieck} satisfies $R_U=\pi^{-1}(R_0(B)) \cap N_U$.
By Theorem \ref{Thm_sp_closed_under_Jasso}, $B$ is also a special biserial algebra,
so Theorem \ref{Thm_R_0_sp} gives 
$R_0(B) \subset \Ker \langle ?, h^B \rangle$.
Proposition \ref{Prop_Jasso} implies that $\Phi(\simple \calW_U)=\simple (\fd B)$,
so we get 
\begin{align*}
R_U \subset \pi^{-1}(R_0(B)) \subset \pi^{-1}(\Ker \langle ?, h^B \rangle) \subset 
\Ker \langle ?,h_U \rangle,
\end{align*} 
where we use Proposition \ref{Prop_Jasso_Grothendieck} (1) for the last inclusion.

Note that $h^B$ and $h_U$ are nonzero, since $|B|=n-|U| \ge 2$.
Since $\twopresilt A$ is a countable set,
$\NR=K_0(\proj A)_\R \setminus \Cone$ is contained in a union of 
countably many hyperplanes of codimension one.
Therefore, $\Cone$ is dense in $K_0(\proj A)_\R$.
\end{proof}

We remark that a finite-dimensional algebra $A$ is said to be \textit{g-tame} 
if $\Cone$ is a dense subset of $K_0(\proj A)_\R$.
There are other proofs of the g-tameness of special biserial algebras.
For example, 
Aoki-Yurikusa \cite{AY} obtained the g-tameness 
by using Dehn twists in the marked surfaces associated to complete gentle algebras.
Also, this was one of the motivations of the paper \cite{PY}, which showed that 
any representation-tame (or finite) algebra is 
g-tame in the same context as Proposition \ref{Prop_E-tame}.

We also have the following axiomatical description of the subset $\Cone$.

\begin{Prop}
Assume that $A$ is a complete special biserial algebra.
Let $\theta \in K_0(\proj A)_\R$.
Then, the following conditions are equivalent.
\begin{itemize}
\item[(a)]
The element $\theta$ belongs to $\Cone$.
\item[(b)]
The family $([Y])_{Y \in \simple \calW_\theta}$ of elements in $K_0(\fd A)$
can be extended to a $\Z$-basis of $K_0(\fd A)$.
\end{itemize}
\end{Prop}

\begin{proof}
We may assume that $A$ is finite-dimensional by Proposition \ref{Prop_reduction_silt}.

$\text{(a)} \Rightarrow \text{(b)}$:
This follows from the argument after Proposition \ref{Prop_Jasso};
more precisely, in the notation there,
$[T_1],[T_2],\ldots,[T_n] \in K_0(\proj A)$ 
for the Bongartz completion $T \in \twosilt A$ give a $\Z$-basis of $K_0(\proj A)$,
and $\simple \calW_\theta=\{ X_1,X_2,\ldots,X_{n-m}\}$ satisfy
$\langle T_i,X_j \rangle=\delta_{i,j}$ 
for any $i \in \{1,2,\ldots,n\}$ and $j \in \{1,2,\ldots,n-m\}$,
so (b) holds.

$\text{(b)} \Rightarrow \text{(a)}$:
Let $\theta$ satisfy the condition (b).
Take the unique $U \in \twopresilt A$ such that $\theta \in R_U$.
We use the notation of Propositions \ref{Prop_Jasso} and \ref{Prop_Jasso_Grothendieck}.
By Theorem \ref{Thm_sp_closed_under_Jasso},
$B$ is isomorphic to a finite-dimensional special biserial algebra.

We apply Corollary \ref{Cor_2h_simple_obj} to $B$, 
and take $Z$ and $k_p$ for all $p \in Z$ there. 
Then,
\begin{align}\label{Eq_2h_B}
\sum_{p \in Z}k_p[M(p)^B]=2h^B \in K_0(\fd B).
\end{align}
Since $\Phi^{-1}(\calW_{\pi(\theta)})=\calW_\theta$ 
by Proposition \ref{Prop_Jasso_Grothendieck}, 
$Y_p:=\Phi^{-1}(M(p)^B)$ is a simple object of $\calW_{\theta}$ for each $p \in Z$,
and we get  
\begin{align*}
\sum_{p \in Z}k_p[Y_p]=2h_U \in K_0(\fd A).
\end{align*}

Now, by the assumption (b), we get that 
$([Y_p])_{p \in Z}$ can be extended to some $\Z$-basis of $K_0(\fd A)$.
Thus, the equation above implies that $k_p \ne 1$ for every $p \in Z$.
By the definition of $k_p$, every $p \in Z$ must be of length 0.
Thus, the equation (\ref{Eq_2h_B}) implies $\{M(p)^B \mid p \in Z\}=\simple (\fd B)$.
Since $\simple (\fd B)=\{M(p)^B \mid p \in Z\} \subset \calW_{\pi(\theta)}$, 
we get $\pi(\theta)=0$, which yields $\theta \in C^+(U)$ as desired.
\end{proof}

We remark that $\text{(a)} \Rightarrow \text{(b)}$ 
holds for all finite-dimensional algebras.
On the other hand, $\text{(b)} \Rightarrow \text{(a)}$ does not hold in general;
for example, if $A$ is the 3-Kronecker algebra 
$K(\begin{tikzpicture}[every node/.style={circle},baseline=(1.base),->]
\node (1) at (  0, 0) {$1$};
\node (2) at (1.2, 0) {$2$};
\draw (1) to (2);
\draw ([yshift=0.2cm] 1.east) to ([yshift=0.2cm] 2.west);
\draw ([yshift=-0.2cm] 1.east) to ([yshift=-0.2cm] 2.west);
\end{tikzpicture})$
and $\theta=[P_1]-r[P_2]$ with $r \in \R \setminus \Q$ and 
$(3-\sqrt{5})/2<r<(3+\sqrt{5})/2$,
then \cite[Section 5]{Asai2} implies that
$\calW_\theta=\{0\}$ (so, (b) is satisfied) and that $\theta \notin \Cone$.

\subsection{Connected components of the exchange quiver}
\label{Subsec_conn}

In this subsection, 
we consider mutations and the exchange quiver of 2-term silting complexes
in $\sfK^\rmb(\proj A)$ for complete special biserial algebras.
We first recall necessary properties.

Let $T \not \cong T' \in \twosilt A$ be two non-isomorphic basic 2-term silting complexes.
Then, we say that $T'$ is a \textit{mutation} of $T$
if there exists $U \in \twopresilt A$ such that $|U|=n-1$
and that $U$ is a common direct summand of $T$ and $T'$.
In this case, we can uniquely take an indecomposable direct summand 
$V$ of $T$ which is not a direct summand of $T'$.
Then, $T'$ is called a mutation of $T$ at $V$.
For any indecomposable direct summand $V$ of $T$,
there uniquely exists a mutation $T'$ of $T$ at $V$ by Proposition \ref{Prop_silt_fund}.

Recall that we associated the torsion class $\ovcalT_T=\calT_T$ 
for $T \in \twosilt A$ in Definition \ref{Def_silt_tors}.
If $T'$ is a mutation of $T$, then
$\ovcalT_{T'} \subsetneq \ovcalT_T$ or $\ovcalT_{T'} \supsetneq \ovcalT_T$
holds,
since we can extend the property \cite[Definition-Proposition 2.28, Theorem 3.2]{AIR} 
for finite-dimensional algebras to complete special biserial algebras
by Proposition \ref{Prop_reduction_brick}.
If the first condition holds, then we say that $T'$ is a \textit{left mutation} of $T$;
otherwise a \textit{right mutation} of $T$.

Under this preparation, we can define the exchange quiver of $\twosilt A$.

\begin{Def}
We define the \textit{exchange quiver} of $\twosilt A$
as the quiver such that its vertices set is $\twosilt A$
and that there exists an arrow $T \to T'$ if and only if $T'$ is a left mutation of $T$.
\end{Def}

By Proposition \ref{Prop_max_common}, 
it is easy to see that the number of connected components of the exchange quiver coincides
that of  
\begin{align*}
\Cone_{\ge n-1}:=\coprod_{U \in \twopresilt A, \ |U| \ge n-1}C^+(U)
=\bigcup_{U \in \twopresilt A, \ |U|=n-1}N_U.
\end{align*}
The main result of this section is the following one.
We set $H^+:=\{ \theta \in K_0(\proj A)_\R \mid \theta(h)>0 \}$
and $H^-:=\{ \theta \in K_0(\proj A)_\R \mid \theta(h)<0 \}$.

\begin{Thm}\label{Thm_conn_comp_gentle}
Let $A$ be a connected complete gentle algebra.
Then, the following assertions hold.
\begin{itemize}
\item[(1)]
Let $T \in \twosilt A$.
If $C^+(T) \cap H^+ \ne \emptyset$, 
then $T$ belongs to the same connected component as $A$ 
in the exchange quiver of $\twosilt A$, 
and if $C^+(T) \cap H^- \ne \emptyset$, 
then $T$ belongs to the same connected component as $A[1]$.
\item[(2)]
The number of connected components of the exchange quiver of $\twosilt A$ is 
\begin{align*}
\begin{cases}
1 & (R_0 \ne \Ker \langle ?,h \rangle) \\
2 & (R_0 = \Ker \langle ?,h \rangle)
\end{cases}.
\end{align*}
\end{itemize}
\end{Thm}

To prove this theorem, we will use $\tau$-tilting reduction,
so we need the following property.
Here, if a complete gentle algebra $B$ is isomorphic to a direct product
$\prod_{i=1}^m B_i$ with $B_i$ a connected algebra, 
we call each $B_i$ a \textit{block} of $B$.

\begin{Prop}\label{Prop_R_0_B}
Let $A$ be a connected complete gentle algebra, 
and $U \in \twopresilt A \setminus \{0\}$.
Consider the complete gentle algebra $B$ given 
in Corollary \ref{Cor_gentle_closed_under_Jasso}.
Then, for any block $B'$ of $B$, we have $R_0(B') \ne \Ker \langle ?,h^{B'} \rangle$.
\end{Prop}

\begin{proof}
In the notation of Corollary \ref{Cor_gentle_closed_under_Jasso},
this corollary implies that 
we can take $k \in \Z_{\ge 1}$ such that 
there exist epimorphisms $B \to B_k \to B/I_\rmc^B$ of algebras.
By Propositions \ref{Prop_reduction_silt} and \ref{Prop_R_0_reduction},
we may reset $A:=A_k$, $U:=U \otimes_A A_k$ and $B:=B_k$
(we do not use the original $A,U,B$ in the rest of the proof).

Take the Bongartz completion $T$ of $U$.

(i)
We first consider the case that $B$ is connected.
We assume that $R_0(B)=\Ker \langle ?,h^B \rangle$ and deduce a contradiction.

The assumption $R_0(B)=\Ker \langle ?,h^B \rangle$ 
implies that $R_0(B)$ is a $(|B|-1)$-dimensional $\R$-vector subspace of 
$K_0(\proj B)_\R$.
Since the bijection $\phi^{-1} \colon R_0(B) \to \overline{N_U} \cap R_0$ is 
a restriction of an $\R$-linear map by Corollary \ref{Cor_Jasso_inv_lin},
we get $\overline{N_U} \cap R_0$ is a $(|B|-1)$-dimensional $\R$-vector subspace of 
$K_0(\proj A)_\R$.

Thus, for any $\theta \in \overline{N_U} \cap R_0$, 
we have $\theta,-\theta \in \overline{N_U}$,
so $H^0(U) \in \ovcalT_\theta \cap \ovcalT_{-\theta}$ and 
$H^{-1}(\nu U)  \in \ovcalF_\theta \cap \ovcalF_{-\theta}$
by the definition of $\overline{N_U}$.
Therefore, $H^0(U),H^{-1}(\nu U) \in \calW_\theta \cap \calW_{-\theta}$.
Set 
\begin{align*}
J:=\left\{i \in Q_0 \mid \text{there exists 
$\sum_{i \in Q_0}a_i[P_i] \in \overline{N_U} \cap R_0$ such that $a_i \ne 0$} \right\}.
\end{align*}
By \cite[Lemma 2.5]{Asai2}, for any $i \in J$,
then $S_i$ is not a composition factor of $H^0(U)$ or $H^{-1}(\nu U)$.
Thus, the simple modules $S_i$ for all $i \in J$ are in $\calW_U$, 
and $\{ S_i \mid i \in J \} \subset \simple \calW_U$.
Since $\overline{N_U} \cap R_0$ is a $(|B|-1)$-dimensional $\R$-vector subspace 
of $\Ker \langle ?,h \rangle$ by Theorem \ref{Thm_R_0_sp}, 
we get $\#J>|B|-1=\#\simple(\fd B)-1=\#\simple \calW_U-1$.
Therefore, $\{ S_i \mid i \in J \} = \simple \calW_U$ and $|B|=\#J$,
so $\fd B \cong \calW_U \cong \fd A'$, where
$A':=A/ \langle 1-\sum_{i \in J} e_i \rangle$.
Thus, the basic algebra $B$ is isomorphic to $A'$, 
and $R_0(A')=\Ker \langle ?,\sum_{i \in J} [S_i] \rangle$ follows 
from $R_0(B)=\Ker \langle ?,h^B \rangle$.

Let $Q'$ be the full subquiver of $Q$ whose vertices set is $J$.
Then, by Theorem \ref{Thm_R_0_gentle},
$R_0(A')=\Ker \langle ?,\sum_{i \in J} [S_i] \rangle$ implies that 
every $j \in Q'_0=J$ has two arrows starting at $j$ and two arrows ending at $j$ in $Q'$.
Since $\#J=|B|=|A|-|U|<|A|$ and $A=KQ/I$ is a connected special biserial algebra,
we have $J=\emptyset$. 
This clearly yields that $B=0$, but it contradicts that $B$ is connected.

Therefore, $R_0(B) \ne \Ker \langle ?,h^B \rangle$.

(ii)
We proceed to the general case.
Decompose $B=B' \times B''$ as algebras.

Since $\red(T)=B$ in Proposition \ref{Prop_Jasso},
there uniquely exists $V \in \twopresilt_U A$ such that $\red(V)=B''$.
Then, the Bongartz completion of $V$ is also $T$, and we have
$\End_A(H^0(T))/[H^0(V)] \cong B'$.
By (i), $R_0(B') \ne \Ker \langle ?,h \rangle$.
\end{proof}

Now, we are able to prove Theorem \ref{Thm_conn_comp_gentle}.

\begin{proof}[Proof of Theorem \ref{Thm_conn_comp_gentle}]
We prove both (1) and (2) at once by induction on $n=|A|$.
If $n=0$, then the assertions are clear.
Assume that $n \ge 1$ in the rest.

(1)
We only consider the first case $C^+(T) \cap H^+ \ne \emptyset$, 
since similar arguments work in the other case.
Let $X$ be the connected component of $\Cone_{\ge n-1}$ containing $C^+(T)$.
It suffices to show that $C^+(A) \subset X$.

Consider the continuous map $f \colon K_0(\proj A)_\R \setminus \{0\} \to [-1,1]$ defined by
\begin{align*}
f(\theta):=\theta(h)/\|\theta\|_1.
\end{align*}
Set $s:=\sup f(X)$ and $l:=\max\{m \in \Z_{\ge 1} \mid s \le 1/m\}$.
Then, we can take $\theta \in X$ such that $f(x)>1/(l+1)$.
By definition, $\theta$ belongs to $C^+(U)$ for some $U \in \twopresilt A$ 
with $|U| \ge n-1$.
Decompose $U=\bigoplus_{i=1}^m U_i$ into the indecomposable direct summands.
By Lemma \ref{Lem_direct_sum_rigid} and Proposition \ref{Prop_sign_coherent},
$[U_1],[U_2],\ldots,[U_m]$ are sign-coherent.
Thus, $f(\theta)>1/(l+1)$ implies that there must exist $i \in \{1,2,\ldots,m\}$ with
$f([U_i])>1/(l+1)$.
Set $V:=U_i$.
Since $[V](h) \in \{-1,0,1\}$,
we get $f([V]) \in \pm (1/\Z_{\ge 1}) \cup \{0\}$,
so $f([V])>1/(l+1)$ yields $f([V]) \ge 1/l \ge s$.

We show that $\Cone_{\ge n-1} \cap N_V$ is connected.
This is the union of $C^+(W)$ for all $W \in \twopresilt_V A$ such that $|W| \ge n-1$.
Apply Corollary \ref{Cor_gentle_closed_under_Jasso} to $V$,
and take the complete gentle algebra $B$ there.
By Proposition \ref{Prop_R_0_B}, 
$R_0(B') \ne \Ker \langle ?,h^{B'} \rangle \subset K_0(\proj B')_\R$ 
for all blocks $B'$ of $B$, 
so the exchange quiver of $\twosilt B$ is connected by the induction hypothesis for (2).
Thus, the exchange quiver of $\twosilt_V A$ is also connected, and 
$\Cone_{\ge n-1} \cap N_V$ is connected.

Since $\Cone_{\ge n-1} \cap N_V$ is connected and 
$C^+(U) \subset X \cap (\Cone_{\ge n-1} \cap N_V)$,
we have $\Cone_{\ge n-1} \cap N_V$ is contained in the connected component $X$.
We can take $\epsilon>0$ such that $[V]+\epsilon[A] \in N_V$,
and by the definition of Bongartz completions,
$[V]+\epsilon[A] \in C^+(T_V)$ for the Bongartz completion $T_V$ of $V$.
Thus, $[V]+\epsilon[A] \in X$.
Since $[V]$ and $[V]+\epsilon[A] \in N_V$ are sign-coherent,
we get $f([V]) \le f([V]+\epsilon[A]) \le \sup f(X) = s \le f([V])$.
Therefore, we have $0<f([V])=f([V]+\epsilon[A])$, 
which implies that $V=P_i$ for some $i \in Q_0$.
The Bongartz completion $T_V$ is obviously $A$,
so $C^+(A) \subset \Cone_{\ge n-1} \cap N_V \subset X$.

Thus, if $T \in \twosilt A$ satisfies $C^+(T) \cap H^+ \ne \emptyset$
and $X$ is a connected component of $\Cone_{\ge n-1}$,
then $C^+(T)$ and $C^+(A)$ are contained in $X$,
so $T$ and $A$ belong to the same connected component of the exchange quiver.

(2)
By (1), both $\Cone_{\ge n-1} \cap H^+$ and $\Cone_{\ge n-1} \cap H^-$
are connected.

If $R_0 = \Ker \langle ?,h \rangle$, then $\Cone_{\ge n-1} \subset H^+ \amalg H^-$,
so we get the assertion.

Otherwise, $R_0 \ne \Ker \langle ?,h \rangle$.
Since $R_0$ is a rational polyhedral cone by Theorem \ref{Thm_R_0_gentle}, 
there exists $U \in \twopresilt A \setminus \{0\}$ such that 
$C^+(U) \cap H \ne \emptyset$. 
Then, $C^+(T) \subset \Cone_{\ge n-1} \cap N_U \cap H^+$ and 
$C^+(T') \subset \Cone_{\ge n-1} \cap N_U \cap H^-$ hold 
for the Bongartz completion $T$ and the Bongartz co-completion $T'$ of $U$.
Thus, it suffices to show that $T$ and $T'$ belong to the same connected component of 
the exchange quiver of $\twosilt_U A$.
By Corollary \ref{Cor_gentle_closed_under_Jasso}, 
there exists a complete gentle algebra $B$ such that $|B|<|A|=n$ and
the exchange quivers of $\twosilt_U A$ and $\twosilt B$ are isomorphic.
Then, Proposition \ref{Prop_R_0_B} and the induction hypothesis for (2) imply 
that the exchange quiver of $\twosilt_U A$ is connected.
Therefore, $\twosilt A$ is connected.

Now, the induction process is complete.
\end{proof}

We have the following result for complete special biserial algebras.

\begin{Cor}\label{Cor_conn_comp_sp}
Let $A$ be a connected complete special biserial algebra.
Then, the following assertions hold.
\begin{itemize}
\item[(1)]
Let $T \in \twosilt A$.
If $C^+(T) \cap H^+ \ne \emptyset$, 
then $T$ belongs to the same connected component as $A$, 
and if $C^+(T) \cap H^- \ne \emptyset$, 
then $T$ belongs to the same connected component as $A[1]$ in the exchange quiver.
\item[(2)]
The number of connected components of the exchange quiver of $\twosilt A$ is 
\begin{align*}
\begin{cases}
1 & (\Cone_{\ge n-1} \cap \Ker \langle ?,h \rangle \ne \emptyset) \\
2 & (\Cone_{\ge n-1} \cap \Ker \langle ?,h \rangle = \emptyset)
\end{cases}.
\end{align*}
\end{itemize}
\end{Cor}

\begin{proof}
(1)
Let $T \in \twosilt A$.
We only consider the first case $C^+(T) \cap H^+ \ne \emptyset$.

We take a complete gentle algebra $\widetilde{A}$ such that
$A$ is a quotient algebra of $\widetilde{A}$ and that $|A|=|\widetilde{A}|$.
Since $\Cone_{\ge n-1}(\widetilde{A})$ is dense in $\Cone(\widetilde{A})$,
Corollary \ref{Cor_cone_dense} tells us that $\Cone_{\ge n-1}(\widetilde{A})$ is dense in 
$K_0(\proj A)_\R=K_0(\proj \widetilde{A})_\R$.
In particular, $C^+(T) \cap H^+ \cap \Cone_{\ge n-1}(\widetilde{A}) \ne \emptyset$.
Take $\theta \in C^+(T) \cap H^+ \cap \Cone_{\ge n-1}(\widetilde{A})$.
By Theorem \ref{Thm_conn_comp_gentle} (1),
$\theta$ belongs to the connected component of $\Cone_{\ge n-1}(\widetilde{A})$
containing $C^+(\widetilde{A})$,
so it is in the connected component of $\Cone_{\ge n-1}$ containing $C^+(A)$.
Since $\theta \in C^+(T)$, this means that $T$ and $A$ are in the same connected component
of the exchange quiver.

(2)
This is obvious from (1).
\end{proof}

We remark that the following example shows that we cannot replace the condition 
$\Cone_{\ge n-1} \cap \Ker \langle ?,h \rangle \ne \emptyset$ to 
$R_0 \subset \Ker \langle ?,h \rangle$ in Corollary \ref{Cor_conn_comp_sp}.

\begin{Ex}
We use the setting of Example \ref{Ex_triangle_sp}.
Recall that we obtained 
\begin{align*}
R_0=\R_{\ge 0}([P_1]-[P_2]) \cup \R_{\ge 0}([P_2]-[P_3]) \cup \R_{\ge 0}([P_3]-[P_1]).
\end{align*}
Thus, $R_0 \ne \Ker \langle ?,h \rangle$.
Moreover, $\eta_{1,2}:=[P_1]-[P_2]$, $\eta_{2,3}:=[P_2]-[P_3]$, $\eta_{3,1}:=[P_3]-[P_1]$
belong to $\INR(A)$ by Lemma \ref{Lem_neighbor_basic}.

On the other hand, we show $\Cone_{\ge 2} \cap \Ker \langle ?,h \rangle = \emptyset$.
It is easy to see that $\sigma_{1,3}:=[P_1]-[P_3]$, $\sigma_{3,2}:=[P_3]-[P_2]$,
$\sigma_{2,1}:=[P_2]-[P_1]$ are in $\IR(A)$.
From direct calculation, we have $\eta_{1,2},\eta_{2,3} \in \overline{N_{\sigma_{1,3}}}$.
Thus, by Proposition \ref{Prop_R_U_decompose},
\begin{align*}
(\R_{\ge 0}\sigma_{1,3} \oplus \R_{\ge 0}\eta_{1,2}) \cap \Cone_{\ge 2}&=\emptyset, &
(\R_{\ge 0}\sigma_{1,3} \oplus \R_{\ge 0}\eta_{2,3}) \cap \Cone_{\ge 2}&=\emptyset.
\end{align*}

By applying this argument also to $\sigma_{3,2}$ and $\sigma_{2,1}$,
we can conclude that $\Cone_{\ge 2} \cap \Ker \langle ?,h \rangle = \emptyset$,
Therefore, there exist two connected components of the exchange quiver of $\twosilt A$
by Corollary \ref{Cor_conn_comp_sp}.
\end{Ex}

\subsection{$\tau$-tilting finiteness}
\label{Subsec_fin}

In this section, we give other proofs to some known results on $\tau$-tilting finiteness
of special biserial algebras.

First, as an application of Theorem \ref{Thm_R_0_sp},
we have the following result, where the equivalence of (a) and (b) has already been
proved in \cite[Theorem 5.1]{STV}.
More explicitly, our results give a proof of $\text{(b)} \Rightarrow \text{(a)}$
different from theirs.

\begin{Cor}\label{Cor_STV}
For any complete special biserial algebra $A$,
the following conditions are equivalent.
\begin{itemize}
\item[(a)]
The algebra $A$ is $\tau$-tilting finite.
\item[(b)]
There exists no band $A$-module which is a brick.
\item[(c)]
Any indecomposable $\theta \in K_0(\proj A)$ is rigid.
\item[(d)]
The subset $R_0 \cap K_0(\proj A)$ is $\{0\}$.
\item[(e)]
The subset $R_0$ is $\{0\}$.
\end{itemize}
\end{Cor}

\begin{proof}
$\text{(a)} \Rightarrow \text{(b)}$ follows from \cite[Proposition 3.1]{STV}.

$\text{(b)} \Rightarrow \text{(c)}$ follows from Proposition \ref{Prop_tau_reduced}.

$\text{(c)} \Rightarrow \text{(d)}$ is immediate by Lemma \ref{Lem_neighbor_basic}.

$\text{(d)} \Rightarrow \text{(e)}$ follows from Theorem \ref{Thm_R_0_sp}.

$\text{(e)} \Rightarrow \text{(a)}$ is Proposition \ref{Prop_tau_finite_R_0}.
\end{proof}

Our main result also gives 
another proof for the characterization of 
$\tau$-tilting Brauer graph algebras by \cite{AAC}.
We briefly recall the definition of Brauer graphs and Brauer graph algebras.

Let $G=(V,E)$ be a finite unoriented graph with $V$ the vertices set and $E$ the edges set.
We call $G$ a \textit{Brauer graph} if the following information is additionally given:
\begin{itemize} 
\item
a \textit{cyclic ordering} of the edges involving each vertex $v \in V$,
that is, a cyclic permutation 
$e_{v,1} \mapsto e_{v,2} \mapsto \cdots \mapsto e_{v,l_v} \mapsto e_{v,l_v+1}=e_{v,1}$ 
satisfying
\begin{align*}
\#\{i \in \{1,2,\ldots,l_v\} \mid e_{v,i}=e\}=
\begin{cases}
2 & \text{($e$ is a loop)} \\
1 & \text{(otherwise)}
\end{cases}
\end{align*}
for each edge $e$ around $v$;
\item
the \textit{multiplcity} $m(v) \in \Z_{\ge 1}$ of the vertex $v \in V$.
\end{itemize}

Let $G=(V,E)$ be a Brauer graph.
Then, the \textit{Brauer graph algebra} is 
a finite-dimensional special biserial algebra $\widehat{KQ}/I$ defined by
\begin{itemize}
\item
$Q$ is the quiver whose vertices set is $E$ and whose arrows set is
\begin{align*}
\{\alpha_{v,i} \colon e_{v,i} \to e_{v,i+1} \mid v \in V, \ i \in \{1,2,\ldots,l_v\} \};
\end{align*}
\item
$I:=I_1+I_2$ is the ideal of $\widehat{KQ}$ defined by 
\begin{itemize}
\item
$I_1$ is the ideal generated by $\alpha\beta$ for all pairs $(\alpha,\beta)$ of arrows 
not admitting $v \in V$ and $i \in \{1,2,\ldots,l_v\}$
such that $\alpha=\alpha_{v,i}$ and $\beta=\alpha_{v,i+1}$;
\item
$I_2$ is the ideal by the elements
\begin{align*}
(\alpha_{v,i}\alpha_{v,i+1}\cdots\alpha_{v,i+l_v-1})^{m(v)}-
(\alpha_{w,j}\alpha_{w,j+1}\cdots\alpha_{w,j+l_w-1})^{m(w)}
\end{align*}
for all $e \in E$, where $\alpha_{v,i}$ and 
$\alpha_{w,j}$ are the two distinct arrows starting at each $e$.
\end{itemize}
\end{itemize}
Here, we set $\alpha_{v,i+l_v}:=\alpha_{v,i}$.

Note that $I$ is not necessarily admissible;
if there exists a vertex $v \in V$ such that $l_v=1$ and $m(v)=1$,
then $I_2 \not \subset \rad^2 \widehat{KQ}$.

It is clear that $\widetilde{A}:=\widehat{KQ}/I_1$ is a complete gentle algebra.
Consider the quotient map $\pi \colon \widehat{KQ} \to \widetilde{A}$.
Then, the image $\pi(I_2)$ of the ideal $I_2$ 
is contained in the ideal $I_\rmc^{\widetilde{A}}$ of $\widetilde{A}$.
Therefore, $R_0(\widetilde{A})=R_0$ follows from Proposition \ref{Prop_R_0_reduction}.
In particular, the multiplicities $m(v)$ of the vertices do not matter to $R_0$.

Below, a \textit{cycle} in a Brauer graph $G=(V,E)$ means an $l$-gon in $G$
as in \cite[Remark 2.13]{AAC};
thus, a cycle consists of $l$ distinct vertices $v_1,v_2,\ldots,v_l$
and $l$ edges $e_1,e_2,\ldots,e_l$ such that $e_i$ is an edge between $v_i$ and $v_{i+1}$
(we set $v_{l+1}:=v_1$), where $l \ge 1$.
If $l$ is odd, then the cycle is called an \textit{odd cycle};
otherwise, an \textit{even cycle}.

Under this preparation, we can give another proof of the following result.

\begin{Cor}\label{Cor_AAC}
Let $A$ be the Brauer graph algebra of a connected Brauer graph $G=(V,E)$.
\begin{itemize}
\item[(1)]
Set $x_v:=[S_{e_{v,1}}]+[S_{e_{v,2}}]+\cdots+[S_{e_{v,l_v}}] \in K_0(\fd A)$ 
for each $v \in V$.
Then, we have
\begin{align*}
R_0=\bigcap_{v \in V} \Ker \langle ?, x_v \rangle.
\end{align*} 
In particular, $R_0$ is an $\R$-vector subspace of $K_0(\proj A)_\R$.
\item[(2)]\cite[Theorem 6.7]{AAC}
Then, $A$ is $\tau$-tilting finite if and only if 
$G$ contains no even cycle and at most one odd cycle.
\end{itemize}
\end{Cor}

\begin{proof}
(1)
We can take a complete gentle algebra $\widetilde{A}:=\widehat{KQ}/I_1$ above.
For any $v \in V$ and $i \in \{1,2,\ldots,l_v\}$,
we set $c(v,i):=\alpha_{v,i}\alpha_{v,i+1}\cdots\alpha_{v,i+l_v-1}$.
Then, $\Cyc(\widetilde{A})=\{c(v,i) \mid v \in V,\ i \in \{1,2,\ldots,l_v\}\}$,
and $\ovMP(\widetilde{A})=\emptyset$.

Now, Theorem \ref{Thm_R_0_gentle} implies that 
\begin{align*}
R_0=R_0(\widetilde{A})=\bigcap_{v \in V} \Ker \langle ?, x_v \rangle.
\end{align*}

(2)
By (1) and Corollary \ref{Cor_STV},
$A$ is $\tau$-tilting finite if and only if 
the $\R$-vector subspace $X:=\sum_{v \in V}\R x_v$ is $K_0(\fd A)_\R$.

We have $\dim_\R X \le \# V$ and $\dim_\R K_0(\fd A)_\R=\# E$.
Thus, if $A$ is $\tau$-tilting finite, then $\#V \ge \#E$, so $G$ has at most one cycle.

If $v_0$ is a vertex in $G$ and there exists only one edge $e_0$ involving $v_0$,
then we have a new Brauer graph $G'$ by deleting $v_0$ and $e_0$.
Consider its Brauer graph algebra $A'$, 
then we can check that $X=K_0(\fd A)_\R$ holds
if $\sum_{v \in V \setminus \{v_0\}}\R x'_v = K_0(\fd A')_\R$,
where $x'_v:=x_v-[S_e]$ if $e_0$ involves $v$, and $x'_v:=x_v$ otherwise.

If $G$ contains no cycle, then by repeating this process, 
we get that $X=K_0(\fd A)_\R$ holds if $\sum_{v \in \emptyset}\R x_v=0$, 
which is obviously true, 
so we obtain $X=K_0(\fd A)_\R$.
Thus, $A$ is $\tau$-tilting finite.

Otherwise, $G$ contains exactly one cycle. 
By induction, we may assume that $G$ itself is a cycle 
to determine whether $A$ is $\tau$-tilting finite or not.
If $G$ consists of $l$ distinct vertices $v_1,v_2,\ldots,v_l$
and edges $e_1,e_2,\ldots,e_l$ with $e_i$ connects $v_i$ and $v_{i+1}$
($v_{l+1}:=v_1$),  
then we can observe that $X$ is generated by the elements 
\begin{align*}
[S_{e_i}]+[S_{e_{i+1}}] \in K_0(\fd A)
\end{align*}
for all $i \in \{1,2,\ldots,l\}$ ($e_{l+1}:=e_1$).
We can check that $X=K_0(\fd A)_\R$ holds if and only if $l$ is odd.
Therefore, 
if $G$ contains no even cycle and a unique odd cycle, then $A$ is $\tau$-tilting finite;
on the other hand, if $G$ has an even cycle and no odd cycle, 
then $A$ is not $\tau$-tilting finite.

Now, the proof is complete.
\end{proof}

Finally, we apply our results to special biserial radical square zero algebras.
Assume that a finite quiver $Q$ satisfies (b) and (c) in Definition \ref{Def_sp_bi_alg},
that is, the number of arrows starting (resp.~ending) at each vertex $i \in Q_0$ 
is at most two,
and set $I$ as the ideal generated by all the paths of length $2$.
Then, $A:=KQ/I$ is a finite-dimensional special biserial algebra such that $\rad^2 A=0$.
Conversely, any special biserial radical square zero algebra is obtained in this way.

As in \cite{Adachi,Aoki}, 
we can define the \textit{separated quiver} $Q^\mathrm{s}$ of $Q$;
namely, the vertices set is $(Q^\mathrm{s})_0:=\{i^+,i^- \mid i \in Q_0\}$
and the arrows set is $(Q^\mathrm{s})_1:=\{ i^+ \to j^- \mid (i \to j) \in Q_1\}$.
We say that a full subquiver $Q'$ of $Q^\mathrm{s}$ is a \textit{single subquiver}
if no $i \in Q_0$ satisfies $i^+,i^- \in Q'_0$.
Then, we can recover \cite[Theorem 3.1]{Adachi} in the case that 
$A$ is a special biserial radical square zero algebra.

\begin{Prop}
Let $A=\widehat{KQ}/I$ be a special biserial algebra with $\rad^2 A=0$.
Then, $A$ is $\tau$-tilting finite if and only if 
any single subquiver of $Q^\mathrm{s}$ is a disjoint union of quivers of type $\bbA$.
\end{Prop}

\begin{proof}
We first prove the ``only if'' part.
We assume that there exists a single subquiver 
of $Q^\mathrm{s}$ is not a disjoint union of $\bbA_n$.
It suffices to show that $A$ is not $\tau$-tilting finite.
Since $Q$ satisfies (b) and (c) in Definition \ref{Def_sp_bi_alg},
$Q^\mathrm{s}$ has a single subquiver $Q'$ of type $\widetilde{\bbA}_n$,
and we can consider $Q'$ as a band $b$.
Since $Q'$ is a single subquiver, the band module $M(b,\lambda)$ is a brick.
Thus, $A$ is not $\tau$-tilting finite by Corollary \ref{Cor_STV}.

We next prove the ``if'' part.
Suppose that $A$ is not $\tau$-tilting finite,
then by Corollary \ref{Cor_STV}, we can take some $\eta \in \INR(A)$.
Take a band $b_\eta$ in Proposition \ref{Prop_tau_reduced}.
Write $\eta=\sum_{i=1}^n a_i[P_i]$.
By using $M(b_\eta,\lambda) \cdot \rad^2 A=0$,
each simple module $S$ appears in the composition factors of $M(b_\eta,\lambda)$
at most once,
because the band $b_\eta$ does not admit a shorter string $s$ 
such that $b_\eta=s^m$ with $m \ge 2$.
The arrows and the inverse arrows appearing in the band $b_\eta$ 
define a subquiver $Q'$ of $Q^\mathrm{s}$. 
In particular, $a_i \in \{-1,0,1\}$ holds for each $i$.

The vertices set of $Q'$ is $\{i_+ \mid i \in Q_0, \ a_i=1\} \amalg 
\{i_- \mid i \in Q_0, \ a_i=-1\}$, so $Q'$ is a single subquiver of $Q^\mathrm{s}$ 
of type $\widetilde{\bbA}$,
so it is not a disjoint union of quivers of type $\bbA$.
\end{proof}

\section*{Funding}

The author is a Research Fellow of Japan Society for the Promotion of Science, and was supported by Japan Society for the Promotion of Science 
KAKENHI JP16J02249, JP19K14500 and JP20J00088.

\section*{Acknowledgment}

The author thanks Toshitaka Aoki, Aaron Chan, Laurent Demonet, Osamu Iyama and Toshiya Yurikusa for kind instructions and discussions.


\begin{thebibliography}{DIRRT}
\bibitem[Ada]{Adachi}
T. Adachi,
\textit{Characterizing $\tau$-tilting finite algebras with radical square zero}, 
Proc. Amer. Math. Soc. \textbf{144} (2016), no. 11, 4673--4685.
\bibitem[AAC]{AAC}
T. Adachi, T. Aihara, and A. Chan, 
\textit{Classification of two-term tilting complexes over Brauer graph algebras},
Math. Z. \textbf{290} (2018), no. 1--2, 1--36.
\bibitem[AIR]{AIR}
T. Adachi, O. Iyama, I. Reiten,
\textit{$\tau$-tilting theory},
Compos. Math. \textbf{150} (2014), no. 3, 415--452.
\bibitem[Aih]{Aihara}
T. Aihara,
\textit{Tilting-connected symmetric algebras},
Algebr. Represent. Theor. \textbf{16} (2013), Issue 3, 873--894.
\bibitem[AiI]{AiI}
T. Aihara, O. Iyama,
\textit{Silting mutation in triangulated categories},
J. Lond. Math. Soc. (2) \textbf{85} (2012), 633--668.
\bibitem[APS]{APS}
C. Amiot, P.-G. Plamondon, S. Schroll,
\textit{A complete derived invariant for gentle algebras via winding numbers and Arf invariants},
arXiv:1904.02555v3. 
\bibitem[Aok]{Aoki}
T. Aoki,
\textit{Classifying torsion classes for algebras with radical square zero via sign decomposition},
arXiv:1803.03795v2.
\bibitem[AY]{AY}
T. Aoki, T. Yurikusa,
\textit{Complete special biserial algebras are g-tame},
arXiv:2003.09797v2.
\bibitem[Asa1]{Asai1}
S. Asai, 
\textit{Semibricks},
Int. Math. Res. Not. IMRN 2020, Issue 16, 4993--5054.
\bibitem[Asa2]{Asai2}
S. Asai, 
\textit{The wall-chamber structures of the real Grothendieck groups},
Adv. Math. \textbf{381} (2021), 107615.
\bibitem[AsI]{AsI}
S. Asai, O. Iyama,
\textit{Semistable torsion classes and canonical decompositions},
arXiv:2112.14908v1.
\bibitem[AS]{AS}
M. Auslander, S. O. Smal\o, 
\textit{Almost split sequences in subcategories},
J. Algebra \textbf{69} (1981), no. 2, 426--454. 
\bibitem[BKT]{BKT}
P. Baumann, J. Kamnitzer, P. Tingley, 
\textit{Affine Mirkovi\'{c}-Vilonen polytopes},
Publ. Math. Inst. Hautes \'{E}tudes Sci. \textbf{120} (2014), 113--205.
\bibitem[Bri]{Bridgeland}
T. Bridgeland,
\textit{Scattering diagrams, Hall algebras and stability conditions}, 
Algebr. Geom. \textbf{4} (2017), no. 5, 523--561.
\bibitem[BST]{BST}
T. Br\"{u}stle, D. Smith, H. Treffinger,
\textit{Wall and Chamber Structure for finite-dimensional Algebras},
Adv. Math. \textbf{354} (2019), 106746.
\bibitem[BY]{BY}
T. Br\"{u}stle, D. Yang,
\textit{Ordered exchange graphs},
Advances in representation theory of algebras, 135--193,
EMS Ser. Congr. Rep., Eur. Math. Soc., Z\"{u}rich, 2013.
\bibitem[BR]{BR}
M. C. R. Butler, C. M. Ringel,
\textit{Auslander--Reiten sequences with few middle terms and applications to string
algebras}, 
Comm. Alg. \textbf{15}, Issue 1--2 (1987), 145--179. 
\bibitem[CB1]{CB1}
W. Crawley-Boevey, 
\textit{On tame algebras and bocses}, 
Proc. London Math. Soc. (3) \textbf{56}, Issue 3 (1988), 451--483.
\bibitem[CB2]{CB2}
W. Crawley-Boevey, 
\textit{Maps between representations of zero-relation algebras}, 
J. Algebra \textbf{126}, no. 2 (1989), 259--263.
\bibitem[DIJ]{DIJ}
L. Demonet, O. Iyama, G. Jasso, 
\textit{$\tau$-tilting finite algebras, bricks, and g-vectors},
Int. Math. Res. Not. 
IMRN 2019, Issue 3, 852--892.
\bibitem[DIRRT]{DIRRT}
L. Demonet, O. Iyama, N. Reading, I. Reiten, H. Thomas,
\textit{Lattice theory of torsion classes},
arXiv:1711.01785v2.
\bibitem[DF]{DF}
H. Derksen, J. Fei,
\textit{General presentations of algebras},
Adv. Math. \textbf{278} (2015), 210--237.
\bibitem[EJR]{EJR}
F. Eisele. G. Janssens, T. Raedschelders,
\textit{A reduction theorem for $\tau$-rigid modules},
Math. Z. \textbf{290} (2018), 1377--1413.
\bibitem[Fei]{Fei}
J. Fei, 
\textit{Tropical $F$-polynomials and general presentations}, 
arXiv:1911.10513v2.
\bibitem[GLFS]{GLFS}
C. Geiss, D. Labardini-Fragoso, J. Schr\"{o}er,
\textit{Schemes of modules over gentle algebras and laminations of surfaces},
arXiv:2005.01073v2.
\bibitem[HKK]{HKK}
F. Haiden, L. Katzarkov, M. Kontsevich,
\textit{Flat surfaces and stability structures},
Publ. Math. Inst. Hautes \'{E}tudes Sci. \textbf{126} (2017), 247--318.
\bibitem[IT]{IT}
C. Ingalls, H. Thomas,
\textit{Noncrossing partitions and representations of quivers}, 
Compos. Math. \textbf{145}, no. 6 (2009), 1533--62.
\bibitem[IK]{IK}
O. Iyama, Y. Kimura,
\textit{Classifying torsion pairs of Noetherian algebras},
arXiv:2106.00469v1.
\bibitem[IY]{IY}
O. Iyama, D. Yang,
\textit{Silting reduction and Calabi--Yau reduction of triangulated categories},
Trans. Amer. Math. Soc. \textbf{370}, no. 11 (2018), 7861--7898.
\bibitem[Jas]{Jasso}
G. Jasso,
\textit{Reduction of $\tau$-tilting modules and torsion pairs},
Int. Math. Res. Not. IMRN 2015, no. 16,  7190--7237. 
\bibitem[Kac]{Kac}
V. G. Kac,
\textit{Infinite root systems, representations of graphs and invariant theory},
Invent. Math. \textbf{56} (1980), 57--92.
\bibitem[KV]{KV}
B. Keller, D. Vossieck, 
\textit{Aisles in derived categories}, 
Deuxi\`{e}me Contact Franco-Belge en Alg\`{e}bre (Faulx-les-Tombes, 1987),
Bull. Soc. Math. Belg. S\'{e}r. A \textbf{40} (1988), no. 2, 239--253.
\bibitem[Kim]{Kimura}
Y. Kimura,
\textit{Tilting theory of noetherian algebras},
arXiv:2006.01677v1.
\bibitem[KM]{KM}
Y. Kimura, Y. Mizuno,
\textit{Two-term tilting complexes for preprojective algebras of non-Dynkin type}, 
Comm. Alg., \url{https://doi.org/10.1080/00927872.2021.1962337}.
\bibitem[Kin]{King}
A. D. King, 
\textit{Moduli of representations of finite dimensional algebras},
Quart. J. Math. Oxford Ser. (2) \textbf{45} (1994), no. 180, 515--530.
\bibitem[KY]{KY}
S. Koenig, D. Yang,
\textit{Silting objects, simple-minded collections, 
t-structures and co-t-structures for finite-dimensional algebras},
Doc. Math. \textbf{19} (2014), 403--438.
\bibitem[Kra]{Krause}
H. Krause, 
\textit{Maps between tree and band modules}, 
J. Algebra \textbf{137}, no. 1 (1991), 186--194.
\bibitem[LP]{LP}
Y. Lekili, A. Polishchuk, 
\textit{Derived equivalences of gentle algebras via Fukaya categories}, 
Math. Ann. \textbf{376}, no. 1--2 (2020), 187--225.
\bibitem[M\v{S}]{MS}
F. Marks, J. \v{S}\v{t}ov\'{i}\v{c}ek,
\textit{Torsion classes, wide subcategories and localisations}, 
Bull. London Math. Soc. \textbf{49} (2017), Issue 3, 405--416.
\bibitem[Mou]{Mousavand}
K. Mousavand,
\textit{$\tau$-tilting finiteness of biserial algebras},
arXiv:1904.11514v1.
\bibitem[OPS]{OPS}
S. Opper, P.-G. Plamondon. S. Schroll,
\textit{A geometric model for the derived category of gentle algebras},
arXiv:1801.09659v5.
\bibitem[PPP1]{PPP1}
Y. Palu, V. Pilaud, P.-G. Plamondon,
\textit{Non-kissing complexes and $\tau$-tilting for gentle algebras}, 
arXiv:1707.07574v3.
\bibitem[PPP2]{PPP2}
Y. Palu, V. Pilaud, P.-G. Plamondon,
\textit{Non-kissing and non-crossing complexes for locally gentle algebras},
J. Comb. Algebra \textbf{3}, no. 4 (2019), 401--438.
\bibitem[Pla]{Plamondon}
P.-G. Plamondon, 
\textit{Generic bases for cluster algebras from the cluster category}, 
Int. Math. Res. Not., IMRN 2013, no. 10, 2368--2420.
\bibitem[PY]{PY}
P.-G. Plamondon, T. Yurikusa, with an appendix by B. Keller,
\textit{Tame algebras have dense g-vector fans},
Int. Math. Res. Not., rnab105, \url{https://doi.org/10.1093/imrn/rnab105}.
\bibitem[Ric]{Rickard}
J. Rickard,
\textit{Morita theory for derived categories},
J. London Math. Soc. (2) \textbf{39} (1989), no. 3, 436--456.
\bibitem[STV]{STV}
S. Schroll, H. Treffinger, Y. Valdivieso,
\textit{On band modules and $\tau$-tilting finiteness},
arXiv:1911.09021v3.
\bibitem[Ter]{Terland}
H. U. Terland,
\textit{Reduction Techniques to Identify Connected Components of Mutation Quivers},
arXiv:2109.11464v2.
\bibitem[VG]{VG}
O. van Garderen, 
\textit{Stability over cDV singularities and other complete local rings},
arXiv:2107.07758v1.
\bibitem[WW]{WW}
B. Wald and J. Waschb\"{u}sch, 
\textit{Tame biserial algebras}, 
J. Algebra \textbf{95}, no. 2 (1985), 480--500.
\bibitem[Yur]{Yurikusa}
T. Yurikusa,
\textit{Wide subcategories are semistable},
Doc. Math. \textbf{23} (2018), 35--47.
\bibitem[ZZ]{ZZ}
A. Zimmermann, A. Zvonareva,
Talk in `Alg\`{e}bres amass\'{e}es et th\'{e}ories des repr\'{e}sentations',
at Universit\'{e} de Caen, 
November 2017.
\end{thebibliography}
\end{document}